\documentclass[twoside,a4paper,12pt,centertags]{amsart}
\usepackage{amsmath,amssymb,verbatim,vmargin}
\usepackage[bookmarks=true]{hyperref}   

\theoremstyle{plain}
\newtheorem{thm}{Theorem}[section]
\newtheorem{lem}[thm]{Lemma}
\newtheorem{cor}[thm]{Corollary}
\newtheorem{prop}[thm]{Proposition}

\theoremstyle{definition}
\newtheorem{defn}[thm]{Definition}
\newtheorem{rem}[thm]{Remark}

\mathchardef\semic="303B

\newcommand{\wedg}{\mathbin{\scriptstyle{\wedge}}}

\newcommand{\R}{{\mathbf R}}
\newcommand{\C}{{\mathbf C}}
\newcommand{\Z}{{\mathbf Z}}
\newcommand{\mH}{{\mathcal H}}
\newcommand{\mK}{{\mathcal K}}
\newcommand{\mX}{{\mathcal X}}
\newcommand{\mY}{{\mathcal Y}}
\newcommand{\mL}{{\mathcal L}}

\newcommand{\mP}{{\mathcal P}}
\newcommand{\mA}{{\mathcal A}}

\newcommand{\mD}{{\mathcal D}}
\newcommand{\V}{{\mathcal V}}
\DeclareMathOperator{\re}{Re}
\newcommand{\im}{\text{{\rm Im}}\,}
\newcommand{\sett}[2]{ \{ #1 \, \semic \, #2 \} }

\newcommand{\supp}{\text{{\rm supp}}\,}

\newcommand{\nul}{\textsf{N}}
\newcommand{\ran}{\textsf{R}}
\newcommand{\dom}{\textsf{D}}

\newcommand{\clos}[1]{\overline{#1}}
\newcommand{\conj}[1]{\overline{#1}}

\newcommand{\sgn}{\text{{\rm sgn}}}
\newcommand{\barint}{\mbox{$ave \int$}}
\newcommand{\divv}{{\text{{\rm div}}}}
\newcommand{\curl}{{\text{{\rm curl}}}}

\newcommand{\tdd}[2]{\tfrac{\partial #1}{\partial #2}}

\newcommand{\st}{\tilde}
\newcommand{\ta}{{\scriptscriptstyle \parallel}}
\newcommand{\no}{{\scriptscriptstyle\perp}}
\newcommand{\pd}{\partial}

\newcommand{\loc}{\text{{\rm loc}}}
\newcommand{\tN}{\widetilde N_*}
\newcommand{\hE}{\widehat E}
\newcommand{\vE}{\check E}
\newcommand{\tE}{\widetilde E}
\newcommand{\tS}{\widetilde S}
\newcommand{\tZ}{\widetilde Z}
\newcommand{\hS}{\widehat S}
\newcommand{\vS}{\check S}
\newcommand{\tD}{\widetilde D}
\newcommand{\tP}{\widetilde P}
\newcommand{\E}{{\mathcal E}}
\newcommand{\bphi}{\varphi}
\newcommand{\bx}{{\bf x}}
\newcommand{\by}{{\bf y}}
\newcommand{\bO}{{\bf O}}

\newcommand{\rad}{\vec{n}}
\newcommand{\ang}{\vec{\tau}}
\newcommand{\essup}{\mathop{\rm ess{\,}sup}}

\def\barint_#1{\mathchoice
            {\mathop{\vrule width 6pt
height 3 pt depth -2.5pt
                    \kern -8.8pt
\intop}\nolimits_{#1}}%
            {\mathop{\vrule width 5pt height
3 pt depth -2.6pt
                    \kern -6.5pt
\intop}\nolimits_{#1}}%
            {\mathop{\vrule width 5pt height
3 pt depth -2.6pt
                    \kern -6pt
\intop}\nolimits_{#1}}%
            {\mathop{\vrule width 5pt height
3 pt depth -2.6pt
          \kern -6pt \intop}\nolimits_{#1}}}

\usepackage{color}

\definecolor{gr}{rgb}   {0.,   0.8,   0. } 
\definecolor{bl}{rgb}   {0.,   0.5,   1. } 
\definecolor{mg}{rgb}   {0.7,  0.,    0.7}

\begin{document}

\title[Maximal regularity for elliptic systems II]
{Weighted maximal regularity estimates and solvability of non-smooth elliptic systems II}
\author{Pascal Auscher} 
\author{Andreas Ros\'en$\,^1$}
\thanks{$^1\,$Formerly Andreas Axelsson}
\address{Pascal Auscher, Univ. Paris-Sud, laboratoire de Math\'ematiques, UMR 8628, Orsay F-91405; CNRS, Orsay, F-91405}
\email{pascal.auscher@math.u-psud.fr}
\address{Andreas Ros\'en, Matematiska institutionen, Link\"opings universitet, 581 83 Link\"oping, Sweden}
\email{andreas.rosen@liu.se}

\begin{abstract}
 We continue the development, by reduction to a first order system for the conormal gradient, of  $L^2$ \textit{a priori} estimates and solvability  for boundary value problems  of Dirichlet, regularity, Neumann type   for divergence form second order, complex, elliptic systems. We work here on the unit ball and more generally its bi-Lipschitz images,  assuming  a Carleson condition    as  introduced by Dahlberg which measures the discrepancy of the coefficients  to their boundary trace near the boundary.  We sharpen our estimates by proving a general result concerning \textit{a priori} almost everywhere non-tangential convergence at the boundary. Also, compactness of the boundary yields more solvability results using Fredholm theory.
Comparison between classes of solutions and uniqueness issues are discussed. As a consequence, we are able to solve a long standing regularity problem for real equations, which may not be true on the upper half-space, justifying \textit{a posteriori}  a separate work on bounded domains. 
\end{abstract}



\maketitle

\tableofcontents

\section{Introduction and main results}

We refer to \cite{AA1}, where we took up this study, for a comprehensive historical account of the theory of boundary value problems for second order equations of divergence form. 
 Before we come to our work here, let us  connect  more deeply to even earlier references going back to the seminal work of Stein and Weiss \cite{SW} that paved the way for the development of Hardy spaces $H^p$ on the Euclidean space in several dimensions. Their key discovery  was  to look at the system of differential equations in the upper-half space satisfied by the gradient  $F= (\partial_{t}u, \nabla_{x}u)$ of a harmonic function $u$ on the upper half-space, to which they gave the name of conjugate system or M. Riesz system.  The system of differential equations is in fact a generalized Cauchy-Riemann system which can be put into a vector-valued ODE form. They did not exploit this ODE structure but used instead subharmonicity properties of $|F|^p$ for $p> \frac{n-1}{n}$ to define the (harmonic) Hardy spaces $H^p$ as the space of those conjugate systems satisfying
$$
\sup_{t>0}\int_{\R^n} |F(t,x)|^p\, dx <\infty
$$
and to prove that the elements in this space have boundary values 
$$
F(t,x) \rightarrow F(0,x) 
$$
in  the $L^p$ norm and almost everywhere non-tangentially.  Further, they proved that elements in $H^p$ can be obtained as Poisson integrals of their boundary traces. In other words, there is a one-to-one correspondence between $H^p$ and its trace space $\mH^p$. By using Riesz transforms, the trace space $\mH^p$ is in one-to-one correspondence with the space defined by taking the first component of trace elements. As they pointed out, it was nothing new for $p>1$ as we get $L^p$, but for $p\le 1$ it gave a new space. Over the years,  this last space turned out to have  many characterizations, including  the ones with Littlewood-Paley functionals of Fefferman and Stein \cite{FS} and  the atomic ones of Coifman~\cite{Coif} and Latter~\cite{Lat}, and is now part of a rich and well understood family of spaces. 

In our earlier work with McIntosh \cite{AAH}, and in \cite{AA1}, we wrote down the Cauchy-Riemann equations corresponding to the second order equation and the key point was a further algebraic transformation that transformed this system to a vector-valued ODE. In some sense, we were going back in time since elliptic equations with non-smooth coefficients  have been developed by other methods since then (see \cite{Kenig}).  In this respect,  it is no surprise in view of the above discussion that we denote our trace spaces by $\mH$. They are in a sense generalized Hardy spaces, and this notation was used as well  in our earlier work with Hofmann  \cite{AAH}. We shall use again such notation and terminology here. 
What today allows the methods of Hardy spaces to be applicable in the case of non-smooth coefficients, are the quadratic estimates related to the solution of the Kato conjecture for square roots.  These are a starting point of the analysis. Indeed, the quadratic estimates are equivalent to  the fact that  two Hardy spaces  split the function space topologically, as it is the case for the classical upper and lower Hardy spaces in complex analysis, essentially from the F. and M. Riesz theorem on the boundedness of the Hilbert transform. So in a sense everything looks like the case of harmonic functions (for $p=2$ at this time). But this is not the case. The difference is in the last step, taking only one component of the trace of a conjugate system. This may or may not be a one-to-one correspondence, which translates to well- or ill-posedness for the boundary value problems of the original second order equation.

 See also \cite{AKQ} for a different generalization of Stein--Weiss conjugate systems of harmonic functions. There conjugate differential forms on Lipschitz domains were constructed by inverting a generalized double layer potential equation on the boundary.  

 Let us introduce some notation in order to state our results. 
Our system of equations is  of the form
\begin{equation}  \label{eq:divform}
 \divv_\bx A \nabla_\bx u(\bx)=  
 \left(\sum_{i,j=0}^n\sum_{\beta= 1}^m \pd_i (A_{i,j}^{\alpha, \beta} \pd_j u^{\beta})(\bx)\right)_{\alpha=1,\ldots, m}=0 ,
 \qquad  \bx\in\Omega,
\end{equation}
where $\pd_i= \tdd{}{x_i}$, $0\le i\le n$ and the matrix of coefficients is
$A=(A_{i,j}^{\alpha,\beta}(\bx))_{i,j=0,\ldots,n}^{\alpha,\beta= 1,\ldots,m}\in L_\infty(\Omega;\mL(\C^{(1+n)m}))$, $n,m\ge 1$.  
We emphasize that the methods used here work equally well for systems ($m\ge 2$) as for equations ($m=1$).
For the time being, $\Omega= \bO^{1+n}:= \sett{\bx\in \R^{1+n}}{|\bx|<1}$ for the unit ball in $\R^{1+n}$ (see the end of the introduction 
for more general Lipschitz domains). 
The coefficient matrix $A$ 
is assumed to satisfy the strict accretivity condition 
\begin{equation}   \label{eq:accrasgarding}
  \int_{S^n} \re(A(rx)\nabla_\bx u(rx), \nabla_\bx u(rx)) dx \ge \kappa \int_{S^n} |\nabla_\bx u(rx)|^2 dx
\end{equation}
for some $\kappa>0$, uniformly for a.e. $r\in(0,1)$ and $u\in C^1(\bO^{1+n}; \C^m)$ where we use polar coordinates $\bx=rx$, $r>0, x\in S^n$,  and $dx$ is the standard (non-normalized) surface measure on $S^n=\partial\bO^{1+n}$. The optimal $\kappa$ is denoted $\kappa_{A}$.
This ellipticity condition is natural when viewing $A$ as a perturbation of its boundary trace. See below. 

The boundary value problems we consider are to find $u\in \mD'(\bO^{1+n};\C^m)$ solving \eqref{eq:divform} in distribution sense,
with appropriate interior estimates of $\nabla_\bx u$ and
 Dirichlet data in $L_{2}$, or  Neumann data in $L_{2}$, or regular Dirichlet data with gradient in $L_{2}$. 
Note that since we shall impose distributional $\nabla_\bx u\in L_2^\loc$, $u$ can be identified
with a function $u\in W^{1,\loc}_2(\bO^{1+n}, \C^m)$, \textit{i.e.} with a weak solution.   In order to study these  boundary value problems, our task, and this is the first main core of the work,  is to obtain  $L^2$ \textit{a priori} estimates.  

As in the previous work \cite{AA1} on the upper-half space $\R^{1+n}_{+}$, we  reduce \eqref{eq:divform} to a first order system with the conormal gradient as unknown function,
so the strategy and the scale-invariant estimates are similar. See the Road Map Section in \cite{AA1} for an overview.
Some changes will arise in the algebraic setup and in the analysis though. Here, the curvature of the boundary (the sphere) will play a role in the algebraic setup, making the unit circle slightly different from the higher dimensional spheres.  In addition, owing to the fact that the boundary is compact, we may use Fredholm theory
to obtain representations and solvability by only making assumptions on the coefficients near the boundary.
We shall focus on this part here and give full details.
We also mention that the whole story relies on a quadratic estimate for a first order bisectorial operator
acting on the boundary function space.
On the upper-half space, this estimate was already available from \cite{AKMc} as a consequence of the 
strategy to prove the Kato conjecture on $\R^n$.
We shall need to prove it on the sphere, essentially by localization and reduction
to \cite{AKMc1}, where such estimates were proved for first order operators with boundary conditions.
An implication of independent interest is the solution to the Kato square root on Lipschitz manifolds. This  is explained in Section~\ref{sec:Kato}.

As is known already for real equations ($m=1$) from work of Caffarelli, Fabes and Kenig~\cite{CaFK},
solvability requires a Dini square regularity condition on the coefficients in the transverse direction to the boundary.  So it is natural to work under a condition of this type. We use the discrepancy function and the Carleson condition  introduced by Dahlberg~\cite{D2}. For a measurable function $f$ on $\bO^{1+n}$, set 
  \begin{equation}
   f^*(\bx):= \essup \limits _{\by\in W^o(\bx)} | f(\by) |, 
\end{equation}
where $W^o(\bx)$ denotes a Whitney region around $\bx\in \bO^{1+n}$ and
  \begin{equation}    \label{eq:defCarleson}
  \|f\|_{C}:=\sup_{r(Q)<c} \left(\frac 1{|Q|} \iint_{(e^{-r(Q)}, 1) Q} f^*(\bx)^2 \frac{d\bx}{1-|\bx|}\right)^{1/2}  \qquad \text{for some fixed } c<1,
\end{equation}
where  the supremum is over all geodesic balls $Q\subset S^n$ of radius $r(Q)<c$. We make  the standing assumption on $A$ throughout that there exists  $A_{1}$  a measurable coefficient matrix on $S^n$, identified with radially independent coefficients in $\bO^{1+n}$, such that $\E(\by):=A({\bf y})-A_{1}(y), y=\by/|\by|$,  satisfies the large Carleson condition  
\begin{equation}    \label{eq:limCarleson}
  \|\E\|_{C}<\infty.
\end{equation}
 The choice of $c$ is irrelevant.
 Note that this means in particular that $\E^*$ vanishes on $S^n$ in a
certain sense and so $A_{1}(y)=A({\bf y}/|{\bf y}|)$. In fact, it can be shown as in \cite[Lemma~2.2]{AA1} that if  there is one such $A_{1}$, it is uniquely defined,  $\|A_{1}\|_{\infty}\le \|A\|_{\infty}$ and $\kappa_{A_{1}}\ge \kappa_{A}$. So we call $A_{1}$ the boundary trace of $A$. 
It turns out that this is a very natural assumption with our method, implying a wealth of \textit{a priori} information about weak solutions as stated in Theorem~\ref{thm:apriori}. Such a result applies  in particular to all systems with radially independent coefficients since $\E=0$ in that case.

For a function $f$ defined in $\bO^{1+n}$, its truncated modified non-tangential maximal function  is defined as in \cite{KP}  by
\begin{equation}
\label{eq:No}
 \tN^o(f)(x):= \sup_{1-\tau<r<1} \left( {|W^o(rx)|^{-1}} \int_{W^o(rx)} |f(\by)|^2 d\by\right)^{1/2}, \qquad x\in S^n,
\end{equation}
for some fixed $\tau<1$. Note that changing the value of $\tau$ will not affect the results.  We shall use the notation $f_{r}(x):=f(rx)$ for $0<r<1, x\in S^n$. Our main result is the following.

\begin{thm}[A priori representations and estimates, existence of a trace, Fatou type convergence]    \label{thm:apriori}
   Consider coefficients $A\in L_\infty(\bO^{1+n}; \mL(\C^{(1+n)m}))$ which are strictly accretive in the sense of \eqref{eq:accrasgarding} and satisfy \eqref{eq:limCarleson} with boundary trace  
   $A_1$.
   Consider $u\in W_2^{1,\loc}(\bO^{1+n};\C^m)$ which satisfies
 (\ref{eq:divform}) in $\bO^{1+n}$ distributional sense.
\begin{itemize}
\item[{\rm (i)}]
  If $\|\tN^o(\nabla_\bx u)\|_{L_2(S^n)} <\infty$, then 
  \smallskip
  \subitem{\rm (a)}   $\nabla_\bx u$ has limit
\begin{equation}
\label{eq:limitNeuReg}
  \lim_{r\to 1} \frac 1{1-r} \int_{r<|\bx|<(1+r)/2} | \nabla_\bx u(\bx) - g_1(x) |^2 d\bx =0
\end{equation}
for some  $g_1 \in L_2(S^n ;\C^{(1+n)m})$ with
$\|g_1\|_{L_2(S^n; \C^{(1+n)m})}\lesssim\|\tN^o(\nabla_\bx u)\|_{L_2(S^n)}$.

  \subitem{\rm (b)} $r\mapsto u_r$ belongs to $C(0,1; L_2(S^n;\C^m))$ and has $L_2$ limit $u_{1}$ at the boundary with
$$
  \| u_r - u_1 \|_{L_2(S^n; C^m)} \lesssim 1-r, 
$$
and $u_{1}\in W^1_{2}(S^n;\C^m)$.
\subitem{\rm (c)} Fatou type results: For almost every $x\in S^n$, 
 $$\lim_{r\to 1} |W^o(rx)|^{-1}\int_{W^o(rx)} u(\by) d\by = u_{1}(x),$$
 $$\lim_{r\to 1} |W^o(rx)|^{-1}\int_{W^o(rx)} \pd_{t}u(\by) d\by = (g_{1})_{\no}(x),$$
 $$\lim_{r\to 1} |W^o(rx)|^{-1}\int_{W^o(rx)} (A\nabla_{\bx}u)_{\ta}(\by) d\by = (A_{1}g_{1})_{\ta}(x),$$
 and if $m=1$ (equations) or $n=1$ (unit disk) we also have 
 $$\lim_{r\to 1} |W^o(rx)|^{-1}\int_{W^o(rx)} \nabla_{\bx}u(\by)d\by = g_{1}(x),$$
  $$\lim_{r\to 1} |W^o(rx)|^{-1}\int_{W^o(rx)} (A\nabla_{\bx}u)(\by) d\by = (A_{1}g_{1})(x).$$

\item[{\rm (ii)}]
  If $\int_{\bO^{1+n}} |\nabla_\bx u|^2 (1-|\bx|)d\bx<\infty$, then
  \smallskip
  \subitem{\rm (a)} 
  $r\mapsto u_r$ belongs to $C(0,1; L_2(S^n;\C^m))$ and has $L_2$ limit
$$
  \lim_{r\to 1} \| u_r - u_1 \|_{L_2(S^n; \C^m)} =0
$$
for some $u_1\in L_2(S^n;\C^m)$. 
\subitem{\rm (b)} We have \textit{a priori} estimates
\begin{gather} \label{eq:apriori}
  \|\tN^o(u)\|^2_{L_2(S^n)}  \lesssim  \int_{\bO^{1+n}} |\nabla_\bx u|^2 (1-|\bx|)d\bx
  + \left| \int_{S^n} u_1(x) dx \right|^2, \\
  \|u_r\|_{L_2(S^n; \C^m)}^2  \lesssim r^{-(n-1)} \int_{\bO^{1+n}} |\nabla_\bx u|^2 (1-|\bx|)d\bx
  + \left| \int_{S^n} u_1(x) dx \right|^2,  \quad r\in (0,1).
\end{gather}
\subitem{\rm (c)} Fatou type results:  For almost every $x\in S^n$, 
$$\lim_{r\to 1} |W^o(rx)|^{-1}\int_{W^o(rx)} u(\by) d\by = u_{1}(x). $$
\end{itemize}
\end{thm}

The definition of the normal component $(\cdot)_{\no}$ and tangential part $(\cdot)_{\ta}$ of a vector field will be given later.  Not stated here are representation formulas giving ansatzes to find solutions as they use a formalism defined later.  In particular, we introduce a notion of a pair of conjugate systems associated to a solution. We note that the non-tangential maximal estimate \eqref{eq:apriori} was already proved in the $\R^{1+n}_{+}$ setting of \cite{AA1}. Again, this is an \textit{a priori} estimate showing that, under the assumption 
  $\|\E\|_{C}<\infty$, the class of weak solutions with square function estimate $\int_{\bO^{1+n}} |\nabla_\bx u|^2 (1-|\bx|)d\bx<\infty$ is contained in the class of weak solutions with  non-tangential maximal estimate $ \|\tN^o(u)\|_{2}<\infty$. 
  The almost everywhere convergences of Whitney averages are new. They apply as well to the setup in \cite{AA1}.
  
Theorem~\ref{thm:apriori} enables us to make the following rigorous definition of 
well-posedness of the BVPs. 

\begin{defn}    \label{defn:wpbvp}
   Consider coefficients $A\in L_\infty(\bO^{1+n}; \mL(\C^{(1+n)m}))$ which are strictly accretive in the sense of \eqref{eq:accrasgarding}.
   \begin{itemize}
  \item By the Neumann problem with coefficients  
  $A$ being well-posed, we mean that
  given $\bphi\in L_2(S^n;\C^m)$ with $\int_{S^n}\bphi(x) dx=0$, 
  there is a function $u\in W_2^{1,\loc}(\bO^{1+n};\C^m)$
  with estimates $\|\tN^o(\nabla_\bx u)\|_{L_2(S^n)} <\infty$, unique modulo constants, 
  solving  (\ref{eq:divform}) 
  and having trace $g_1= \lim_{r\to 1} (\nabla_\bx u)_r$ in the sense of \eqref{eq:limitNeuReg}  such that $(A_1 g_1)_\no= \bphi$.
  \item Well-posedness of the regularity problem is defined in the same way, but 
 replacing the boundary condition $(A_1 g_1)_\no= \bphi$ by
  $(g_1)_\ta= \bphi$, for a given $\bphi\in \ran(\nabla_S)\subset L_2(S^n;\C^{nm})$.

  \item By the Dirichlet problem with coefficients $A$ being well-posed, we mean that 
  given $\bphi\in L_2(S^n;\C^m)$, there is a unique function $u\in W_2^{1,\loc}(\bO^{1+n};\C^m)$
  with estimates $\int_{\bO^{1+n}} |\nabla_\bx u|^2 (1-|\bx|)d\bx<\infty$,
  solving (\ref{eq:divform})  and having trace  $\lim_{r\to 1}u_r= \bphi$ in the sense of almost everywhere convergence of Whitney averages.
\end{itemize}
\end{defn}

  For the Neumann and regularity problem when $\|\E\|_{C}<\infty$, for equations  ($m=1$) or in the unit disk ($n=1$) or any system for which $A$ is strictly accretive in pointwise sense, the trace can  also be defined  in the sense of almost everywhere convergence of Whitney averages of $\nabla_{\bx}u$ and the same for the conormal derivative $(A\nabla_{\bx}u)_{\no}$.  The operator $\nabla_{S}$ denotes the tangential gradient. See Section~\ref{sec:ODE}.
   
   For the Dirichlet problem, the trace is defined  for the almost everywhere convergence of Whitney averages. When $\|\E\|_{C}<\infty$,  Theorem~\ref{thm:apriori} shows that it is the same as the trace in $L_{2}$ sense.  We remark that we modified the meaning of the boundary trace in the definition of the Dirichlet problem  compared to \cite{AA1}. This modification can be made there as well and the same results hold. 
  
  We now come to our  general results on these BVPs.  A small Carleson condition, but only near the boundary, is further  imposed to obtain invertibility of some operators. The second result is on the precise relation between Dirichlet and regularity problems. The first and third are perturbations results 
for radially dependent and independent perturbations respectively. The last  is
 a well-posedness result for three classes of radially independent coefficients.

\begin{thm}   \label{thm:rdepLip}
  Consider coefficients $A\in L_\infty(\bO^{1+n}; \mL(\C^{(1+n)m}))$  which are strictly accretive in the sense of \eqref{eq:accrasgarding}.  Then there exists $\epsilon>0$, such that if $A$
 satisfies the small Carleson condition 
\begin{equation}
\label{eq:smalllimCarleson}
\lim_{\tau\to 1}\|\chi_{\tau<r<1} (A-A_1) \|_C <\epsilon
\end{equation}
  and the Neumann problem with coefficients $A_1$ is well-posed,
  then the Neumann problem is well-posed with coefficients $A$. 

  The corresponding perturbation result for the regularity and Dirichlet problems also holds.
 For the Neumann and regularity problems, the solution $u$ for datum $\bphi$ has estimates
 $$
  \int_{|\bx|<1/2} |\nabla_\bx u|^2 d\bx \lesssim \|\tN^o(\nabla_\bx u)\|^2_{2} 
 \approx \|\bphi\|^2_{2}.
 $$
  For the Dirichlet problem, the solution $u$ for datum $\bphi$ has estimates
\begin{equation*}
   \|\tN^o(u)\|_{2}^2 \approx \sup_{1/2<r<1}\|u_r \|_{2}^2 
   \approx
  \int_{\bO^{1+n}} |\nabla_\bx u|^2 (1-|\bx|)d\bx+ \left| \int_{S^n} \bphi(x) dx \right|^2 
  \approx \|\bphi\|_{2}^2.
\end{equation*}
\end{thm}

An ingredient of the proof is the following relation between  Dirichlet and regularity problems,
in the spirit of \cite[Thm.~5.4]{KP}.

\begin{thm}    \label{thm:reg=dirintro}
  Consider coefficients $A\in L_\infty(\bO^{1+n}; \mL(\C^{(1+n)m}))$ which are strictly accretive in the sense of \eqref{eq:accrasgarding}.  Then there exists $\epsilon>0$, such that if $A$
 satisfies the small Carleson condition 
\eqref{eq:smalllimCarleson},
  then the regularity problem with coefficients $A$ is well-posed if and only if the 
  Dirichlet problem with coefficients $A^*$ is well-posed.
\end{thm}

\begin{thm}   \label{thm:rindepLip}
  Consider radially independent coefficients $A_1\in L_\infty(S^n; \mL(\C^{(1+n)m}))$ which are
  strictly accretive in the sense of \eqref{eq:accrasgarding}.
  If the Neumann problem with coefficients $A_1$ is well-posed, 
  then there exists $\epsilon>0$ such that the Neumann problem  with coefficients $A_1'\in L_\infty(S^n; \mL(\C^{(1+n)m}))$ is well-posed
  whenever $\| A_1-A_1' \|_\infty <\epsilon$.
  The corresponding perturbation results for the regularity and Dirichlet problems also hold.
\end{thm}

\begin{thm}   \label{thm:rindepLiprellich}
  Consider radially independent coefficients $A_1\in L_\infty(S^n; \mL(\C^{(1+n)m}))$ which are
  strictly accretive  in the sense of \eqref{eq:accrasgarding}. The Neumann, regularity and Dirichlet problems with coefficients $A_1$ 
  are well-posed if 

  \begin{enumerate}
  \item  either $A_1$ is  Hermitean, \textit{i.e.} $A_1^*=A_1$,
  \item  or  block form, \textit{i.e.} $(A_1)_{\no\ta}=0= (A_1)_{\ta\no}$ in the normal/tangential splitting of $\C^{(1+n)m}$ (See Section~\ref{sec:ODE}),
  \item  or $A_{1}$ has H\"older regularity  $C^{s}(S^n;\mL(\C^{(1+n)m}))$, $s>1/2$.
\end{enumerate} 
   \end{thm}

\begin{proof}[Proof of Theorems~\ref{thm:apriori}, \ref{thm:rdepLip}, \ref{thm:rindepLip} and \ref{thm:rindepLiprellich}]

For Theorem~\ref{thm:apriori},
  the $L_{2}$-limits and $L_{2}$-estimates of solutions   follow from Theorem~\ref{thm:inteqforNeuandg} and Corollary~\ref{cor:diransatz} respectively. The non-tangential maximal estimate \eqref{eq:apriori} is in Theorem~\ref{thm:NTu}.  Almost everywhere convergence of averages follows from Theorems~\ref{thm:aecv} and \ref{thm:aecvReg/Neu}.

  The well-posedness results in Theorem~\ref{thm:rindepLiprellich}
  are in Propositions~\ref{prop:hermwp}, \ref{prop:blockwp} and \ref{prop:smoothwp}.
  The radially independent perturbation result in Theorem~\ref{thm:rindepLip}
  is in Corollary~\ref{cor:rindeppert}.
  The well-posedness result for radially dependent coefficients with good 
  boundary trace in Theorem~\ref{thm:rdepLip} is in Proposition~\ref{prop:rdeppert}.
\end{proof}

 Our next result is the following semigroup representation, analogous to the result in \cite{Au} in the upper half-space. It is interesting to note that  for harmonic functions $u$, it gives a direct proof (without passing through non-tangential maximal function or $\sup-L_{2}$ estimates) that $\int_{\bO^{1+n}} |\nabla_\bx u|^2 (1-|\bx|)d\bx<\infty$ implies a representation by Poisson kernel from its trace (also shown to exist).  We have not seen this argument in the literature. Another interesting feature is that it points out the importance of well-posedness of the Dirichlet problem when dealing with  more general coefficients. 

\begin{thm}\label{thm:semigroup} 
  Consider radially independent coefficients $A_1\in L_\infty(S^n; \mL(\C^{(1+n)m}))$ which are
  strictly accretive  in the sense of \eqref{eq:accrasgarding}. Assume that the Dirichlet problem with coefficients $A_{1}$ is well-posed. Then the mapping 
$$
\mP_{r}: L_{2}(S^n;\C^m)\to L_{2}(S^n;\C^m): u_{1}\mapsto   u_{r},
$$ 
where $u$ is the solution to the Dirichlet problem with datum $u_1$,
defines a bounded operator for each $r\in (0,1]$. 
The family $(\mP_{r})_{r\in (0, 1]}$ is a multiplicative $C_{0}$-semigroup (\textit{i.e.} $\mP_{r}\mP_{r'}=\mP_{rr'}$ and $\mP_{r} \to I$ strongly in $L_{2}$ when $r\to 1$) whose infinitesimal generator $\mA$ (\textit{i.e.} $\mP_{r}=e^{(\ln r) \mA}$) has domain $\dom(\mA)$ contained in $W^1_{2}(S^n;\C^m)$. 
Moreover, $\dom(\mA)=W^1_{2}(S^n;\C^m)$ 
if and only if the Dirichlet problem with coefficients $A_{1}^*$ is well-posed. 
\end{thm}

As mentioned above two classes of weak solutions compare: the one with square function estimates is contained in the one with non-tangential maximal control. It is thus interesting to examine this  further.  Does the opposite containment holds? How do well-posedness in the two classes compare? Clearly uniqueness in the larger class implies uniqueness in the smaller, and conversely for existence.  As we shall see, positive answers  come  \textit{a posteriori} to solvability.  

\begin{defn}\label{def:wpdahlberg}
The Dirichlet problem  with coefficients $A$ is said to be {\em well-posed in   the sense of Dahlberg} if, given $\bphi\in L_2(S^n;\C^m)$, there is a unique weak solution  $u\in W_2^{1,\loc}(\bO^{1+n}; \C^m)$ to $\divv_{\bx} A\nabla_{\bx}u=0$ with estimates $\|\tN^o(u)\|_{2} <\infty$  and convergence of Whitney averages to $\bphi$, almost everywhere with respect to surface measure on $S^n$.
\end{defn}

This definition has the merit to be natural not only for equations but for systems as well. For real equations, this is equivalent to the usual one as $\tN^o$ can be replaced by the usual pointwise non-tangential maximal operator by the DeGiorgi-Nash-Moser estimates on weak solutions. Even in this case, observe that the control $\|\tN^o(u)\|_{2} <\infty$ does not enforce the almost everywhere convergence property. 
Thus existence of the limit is part of the hypothesis in Definition~\ref{def:wpdahlberg}, as compared to Definition~\ref{defn:wpbvp}.
A first result is the following.

\begin{thm}\label{thm:radind}
Consider radially independent coefficients $A_1\in L_\infty(S^n; \mL(\C^{(1+n)m}))$ which are
  strictly accretive  in the sense of \eqref{eq:accrasgarding}. Assume that the Dirichlet and regularity problems with coefficients $A_{1}$  are well-posed in the sense of Definition~\ref{defn:wpbvp}. 
Then, all weak solutions to $\divv_{\bx} A_{1}\nabla_{\bx}u=0$ with $\|\tN^o(u)\|_{2}<\infty$ are given by the semigroup of Theorem~\ref{thm:semigroup}.  In particular, the Dirichlet problem  with coefficients $A_{1}$ is well-posed in the sense of Dahlberg.
\end{thm}

Theorem~\ref{thm:reg=dirintro} implies the same conclusion for the coefficients $A_{1}^*$. 
The next results are only for real equations where the theory based on elliptic measure brings more information. For (complex) equations, the strict accretivity in the sense of \eqref{eq:accrasgarding} is equivalent to the usual pointwise accretivity, which is the same as the strict ellipticity for real coefficients.

 \begin{thm}\label{thm:equiv2intro} 
 Consider an equation with real coefficients $A\in L_\infty(\bO^{1+n}; \mL(\R^{1+n}))$, which are strictly elliptic. 
  Assume further that  the small Carleson condition \eqref{eq:smalllimCarleson} holds. 
Then the following are equivalent.
 \begin{itemize}
  \item[{\rm (i)}] 
  The Dirichlet problems with coefficients $A$ and $A^*$ are well-posed in the sense of Dahlberg.
  \item[{\rm (ii)}]  
  The Dirichlet problems with coefficients $A$ and $A^*$ are well-posed in the sense of Definition~\ref{defn:wpbvp}.  \end{itemize}
Moreover, in this case the solutions for coefficients $A$ (resp. $A^*$) from a same datum are the same.
\end{thm}

Note that, by Theorem~\ref{thm:reg=dirintro},  we can replace (ii) by (ii'): the regularity problems with coefficients $A$ and $A^*$ are well-posed. When $A=A^*$, all the problems in (i) and (ii') are well-posed by \cite{KP} so there is nothing to prove. For (even non-symmetric) real coefficients $A$ alone, the direction from (ii') to (i) was known from \cite{KP} (without assuming the Carleson condition) and the converse is unknown. It seems that making the statement invariant under taking adjoints solves the issue.  We mention  the equivalence  in \cite{Shen1} concerning $L_{p}$ versions of this statement  for self-adjoint constant coefficient systems on Lipschitz domains (in this case, the $L_{2}$ result is known and used). 

Our last result is well-posedness of  the regularity problem  under a transversal square Dini condition on the coefficients, analogous to the result obtained in \cite{FJK} for the Dirichlet problem with real and symmetric $A$. This partly answers  Problem 3.3.13 in \cite{Kenig}. 

\begin{thm}\label{thm:Dinisquarerealintro} 
Consider an equation with coefficients $A\in L_\infty(\bO^{1+n}; \mL(\C^{1+n}))$, 
which are strictly accretive in the pointwise sense. Then there exists $\epsilon>0$, such that if $A$
 satisfies the small Carleson condition 
\eqref{eq:smalllimCarleson}  and its   boundary trace $A_{1}$  is real and continuous, then the Dirichlet problem with coefficients $A$ is well-posed in the sense of Definition~\ref{defn:wpbvp} and in the sense of Dahlberg, and the regularity problem with coefficients $A$ is well-posed. In particular, this holds if $A$ is real, continuous in $\overline{\bO^{1+n}}$ and  the Dini square condition $\int_{0}w_{A}^2(t) \frac{dt}{t}<\infty$ holds, where $w_{A}(t)=\sup\{|A(rx)-A(x)|  ;  x\in S^n, 1-r<t\}$.
 The corresponding results hold in $\bO^2$ for the Neumann problem with coefficients $A$. 
\end{thm}

Proofs  of Theorems \ref{thm:semigroup} and  \ref{thm:radind} are in  Sections~\ref{sec:unique} and proofs of Theorems~\ref{thm:equiv2intro}  and  \ref{thm:Dinisquarerealintro}  are in Section~\ref{sec:realequations}.

We end this introduction with a remark on the Lipschitz invariance of the above results.
Let $\Omega\subset \R^{1+n}$ be a domain which is Lipschitz diffeomorphic to $\bO^{1+n}$
and let $\rho: \bO^{1+n}\to \Omega$ be the Lipschitz diffeomorphism. Denote the boundary
by $\Sigma:= \partial \Omega$ and the restricted boundary Lipschitz diffeomorphism
by $\rho_0: S^n\to \Sigma$.

Given a function $\tilde u: \Omega\to \C^m$, we pull it back to 
$u:= \tilde u\circ \rho: \bO^{1+n}\to \C^m$.
By the chain rule, we have $\nabla_\bx u= \rho^* (\nabla_\by \tilde u)$, where
the pullback of an $m$-tuple of vector fields $f$, is defined as
$\rho^*(f)(\bx)^\alpha:= \underline{\rho}^t (\bx) f^\alpha(\rho(\bx))$, with $\underline{\rho}^t$
denoting the transpose of Jacobian matrix $\underline{\rho}$.
If $\tilde u$ satisfies $\divv_{\by} \tilde A \nabla_{\by} \tilde u=0$ in $\Omega$, 
with coefficients $\tilde A\in L_\infty(\Omega;\mL(\C^{(1+n)m}))$, then $u$ will satisfy 
$\divv_{\bx} A\nabla_{\bx} u=0$ in $\bO^{1+n}$, where 
$A\in L_\infty(\bO^{1+n};\mL(\C^{(1+n)m}))$ are the ``pulled back'' coefficients defined as
\begin{equation}   \label{eq:pullbackA}
  A({\bf x}):= |J(\rho)({\bf x})| (\underline{\rho}({\bf x}))^{-1} \tilde A(\rho({\bf x})) (\underline{\rho}^t({\bf x}) )^{-1}, 
  \qquad {\bf x}\in \bO^{1+n}.
\end{equation}
Here $J(\rho)$ is the Jacobian determinant of $\rho$.

The Carleson condition,  non-tangential maximal functions and square functions on $\Omega$ correspond to ones on $S^n$ under pullback, so that  $1-|\bx|$ becomes $\delta (\by)$ the distance to $\Sigma$.  In particular, the condition for $\tilde A$ amounts to $\|\E\|_{C}<\infty$ with $\E$ defined from $A$. 
We remark that pullbacks allow to replace normal directions by oblique (but transverse) ones to the sphere in the Carleson condition on the coefficients: take $\rho:\bO^{1+n}\to \bO^{1+n}$ to be a suitable Lipschitz diffeomorphism.

The boundary conditions on $\tilde u$ on $\Sigma$ translate in the following way
to boundary conditions on $u$ on $S^n$.
\begin{itemize}
\item The Dirichlet condition $\tilde u= \tilde \bphi$ on $\Sigma$
is equivalent to $u= \bphi$ on $S^n$, where $\bphi:=\tilde \bphi\circ \rho_0\in L_2(S^n;\C^m)$.
\item The Dirichlet regularity condition 
$\nabla_\Sigma \tilde u= \tilde \bphi$ on $\Sigma$ ($\nabla_\Sigma$
denoting the tangential gradient on $\Sigma$)
is equivalent to $\nabla_S u= \bphi$ on $S^n$, 
where $\bphi:=\rho_0^*(\tilde \bphi)\in \ran(\nabla_S)\subset L_2(S^n;\C^{nm})$.
\item The Neumann condition $(\nu, \tilde A\nabla_\by \tilde u)= \tilde \bphi$ on $\Sigma$
($\nu$ being the outward unit normal vector field on $\Sigma$) with $\int_{\Sigma} \tilde\bphi(y) dy=0$ 
is equivalent to $(\rad, A\nabla_\bx u)= \bphi$ on $S^n$ with $\int_{S^n}\bphi(x)dx=0$, where
$\bphi:= |J(\rho_0)| \tilde \bphi\circ \rho_0\in L_2(S^n;\C^{m})$.
\end{itemize}

In this way the Dirichlet/regularity/Neumann problem with coefficients $\tilde A$ in 
the Lipschitz domain $\Omega$ is equivalent to the Dirichlet/regularity/Neumann problem 
with coefficients $A$ in the unit ball $\bO^{1+n}$,
and it is straightforward to extend the results on $\bO^{1+n}$ above to Lipschitz domains $\Omega$.

The plan of the paper is as follows. In Section~\ref{sec:CR}, we transform the 
second order equation (\ref{eq:divform}) into a system of Cauchy-Riemann type equations. 
In Section~\ref{sec:ODE}, the 
Cauchy-Riemann equations are integrated to a vector-valued ODE
for the conormal gradient of $u$ and a second ODE is introduced to construct a vector potential. The infinitesimal generators $D_0$ and $\tD_{0}$ for these ODE with radially independent coefficients are studied in Sections~\ref{sec:infgene} and \ref{sec:resolvents}, and it is shown in 
Section~\ref{sec:squarefcn} that $D_0$ and $\tD_{0}$ have bounded holomorphic functional calculi. Section~\ref{sec:disk} treats special features of elliptic systems in the unit disk.   In
Section~\ref{sec:XYspaces} we define  the natural function spaces $\mX^o$ and $\mY^o$
for the BVPs and we describe in Section~\ref{sec:semigroups} how to construct solutions from the semigroups generated by $|D_0|=\sqrt{D_0^2}$ and $|\tD_{0}|=\sqrt{\tD_0^2}$. 
In Section~\ref{sec:integration}, the ODE with radially dependent coefficients  for the conormal gradient from Section~\ref{sec:infgene} is reformulated
as an integral equation involving an operator $S_A$, which is shown to be bounded
on the natural function spaces for the BVPs.   In Section~\ref{sec:representation}, we obtain representation for $\mX^o$- and $\mY^o$-solutions. These representations are further developed in Section~\ref{sec:conjugate} where we introduce the notion of a pair of conjugate systems for (\ref{eq:divform}), allowing to prove in Sections~\ref{sec:NT} and \ref{sec:aecv} non-tangential maximal estimates and Fatou type results. 
Crucial for the solvability of (\ref{eq:divform}) is the invertibility of $I-S_A$.
In Section~\ref{sec:fredholm}, we apply Fredholm theory to show that $I-S_A$ is 
invertible on the natural spaces 
whenever the small Carleson condition (\ref{eq:smalllimCarleson}) holds, which proves that it suffices to assume transversal 
regularity of $A$ near the boundary only. (For BVPs on the unbounded half-space
studied in \cite{AA1}, the needed compactness was not available.)  We then study  well-posedness in Section~\ref{sec:BVPs}: this is where we prove the perturbation results, the equivalence Dirichlet/regularity up to taking adjoints and obtain classes of radially independent coefficients for which well-posedness holds. The section~\ref{sec:unique} deals with uniqueness issues, on comparisons of different classes of solutions upon some well-posedness assumptions. We conclude the article in Section~\ref{sec:realequations} by a discussion in the special case of real equations ($m=1$) for which we obtain further results. 

{\sc Acknowledgements.} The support of  Mathematical Science Institute of the Australian National University in Canberra that the first author was visiting  is greatly acknowledged.  We thank Steve Hofmann for several discussions on real equations pertaining to this work. We also thank Alano Ancona, Carlos Kenig and Martin Dindos for a few exchanges.

\section{Generalized Cauchy-Riemann system}\label{sec:CR}

Following \cite{AAH, AAM, AA1}, 
the starting point of our analysis is that  solving  for $u$ the divergence form system (\ref{eq:divform})
amounts to solving for its gradient $g$  a system of Cauchy-Riemann equations. 

\begin{prop}\label{prop:CR} 
Consider  coefficients $A\in L_\infty(\bO^{1+n}; \mL(\C^{(1+n)m}))$.  
If $u$ is a weak solution to $\divv_{\bx} A \nabla_{\bx} u=0$ in $\bO^{1+n}$, then $g:=\nabla_{\bx} u \in L_2^\loc(\bO^{1+n};\C^{(1+n)m})$ is a solution of the generalized Cauchy-Riemann system
\begin{equation}  \label{eq:firstorderdiv}
   \begin{cases}
     \divv_{\bx} \, (Ag)=0, \\
     \curl_{\bx} \,  g= 0,
   \end{cases}
\end{equation}
in $\bO^{1+n}\setminus \{0\}$ distribution sense.  Conversely,  if $g \in L_2^\loc(\bO^{1+n};\C^{(1+n)m})$ is a solution to \eqref{eq:firstorderdiv} in $\bO^{1+n}\setminus \{0\}$ distribution sense, then there exists a weak solution $u$ to $\divv_{\bx} A \nabla_{\bx} u=0$ in $\bO^{1+n}$, such that   
 $g=\nabla_{\bx} u$ in $\bO^{1+n}$ distribution sense.
\end{prop}

\begin{proof} If $u$ is given, then $g:=\nabla_{\bx} u$ has the desired properties and the equation is even satisfied in $\bO^{1+n}$ distribution sense. 
Conversely, assume $g$ is given and satisfies \eqref{eq:firstorderdiv} in $\bO^{1+n}\setminus \{0\}$ distribution sense. Then the next lemma applied to both operators $\divv_{\bx} $ and $\curl_{\bx} $  implies that $0$ is a removable singularity and that 
\eqref{eq:firstorderdiv} holds in $\bO^{1+n}$ distribution sense. Thus one can define a distribution $u$ in $\bO^{1+n}$ such that $g=\nabla_{\bx} u$, hence $\divv_{\bx} A \nabla_{\bx} u=0$ in $\bO^{1+n}$. That $u$ is a weak solution follows from $g \in L_2^\loc(\bO^{1+n};\C^{(1+n)m})$.
\end{proof}

\begin{lem}     \label{lem:removesing} 
  Let $X$ be a homogeneous 
  first order partial differential operator on $\R^{1+n}$ mapping $\C^k$-valued distributions to $\C^\ell$-valued distributions, $k,\ell\in \Z_+$. 
  If  $h\in L_2^\loc(\bO^{1+n}; \C^k)$ and $Xh=0$ in distributional sense on $\bO^{1+n}\setminus\{0\}$, then $Xh=0$ in 
  $\bO^{1+n}$-distributional sense.
\end{lem}

\begin{proof}
Let $\phi\in C_0^\infty(\bO^{1+n};\C^\ell)$.
We need to show
that $\int_{\bO^{1+n}} (X^*\phi, h)d\bx =0$.
To this end, let $\eta_\epsilon$
be a smooth radial function with $\eta_\epsilon=0$ on $\{|\bx|<\epsilon\}$, $\eta_\epsilon=1$ on 
$\{2\epsilon<|\bx|<1\}$ and $\|\nabla\eta_\epsilon\|_\infty\lesssim \epsilon^{-1}$.
Then
\begin{multline*}
  \int_{\bO^{1+n}} \eta_\epsilon(X^*\phi, h)d\bx=   \int_{\bO^{1+n}} (X^*(\eta_\epsilon\phi), h)d\bx-
   \int_{\bO^{1+n}}((X^*\eta_\epsilon)\phi, h)d\bx \\
     = -\int_{\bO^{1+n}} ((X^*\eta_\epsilon)\phi, h)d\bx.
\end{multline*}
As $\epsilon\to 0$, the left hand side converges to $\int_{\bO^{1+n}}(X^*\phi, h)d\bx$,
whereas
$$
  \left|\int_{\bO^{1+n}} ((X^*\eta_\epsilon)\phi, h) d\bx \right|\lesssim \frac 1\epsilon\int_{\epsilon<|\bx|<2\epsilon}|h|d\bx
  \lesssim \epsilon^{(n-1)/2}\left( \int_{\epsilon<|\bx|<2 \epsilon} |h|^2d\bx\right)^{1/2}\to 0.
$$
This proves the lemma.
\end{proof}

%
%
%
%
%
\section{The divergence form equation as an ODE}   \label{sec:ODE}

We introduce a convenient framework  to  transform the Cauchy-Riemann system into an ODE. 

We systematically use boldface letters $\bx,\by, \ldots$ to denote variables in $\R^{1+n}$ and indicate the variable for differential operators in $\R^{1+n}$ (e.g. $\nabla_{\bx}\ldots$).
We denote points on $S^n$ by $x,y,\ldots$ and the standard (non-normalized) surface 
measure on $S^n$ by $dx$.
Polar coordinates are written $\bx= rx$, with $r>0$ and $x\in S^n$.
For a function $f$ defined in $\bO^{1+n}$, we write $f_r(x):= f(rx)$, $x\in S^n$, 
for the restriction to the sphere with radius $0<r<1$, parametrized by $S^n$.

The radial unit vector field we denote by $\rad= \rad(\bx):= \bx/|\bx|$.
Vectors $v\in\R^{1+n}$, we split $v= v_\no \rad + v_\ta$, where
$v_\no := (v,\rad)$ is the {\em normal component} and $v_\ta:= v-v_\no \rad$ is the {\em angular or tangential part} of $v$, which is a vector orthogonal to $\rad$.
Note that $v_\no$ is a scalar, but $v_\ta$ is a vector.
In the plane, \textit{i.e.} when $n=1$, we denote the counter clockwise angular unit vector field by $\ang$, and we have $v=v_\no \rad + (v, \ang)\ang$.
For an $m$-tuple of vectors $v=(v^\alpha)_{1\le \alpha\le m}$, we define its
normal components and tangential parts componentwise as
$$(v_\no)^\alpha:= (v^\alpha)_\no, \qquad (v_\ta)^\alpha:= (v^\alpha)_\ta.$$

The tangential gradient, divergence and curl on the unit sphere are denoted by $\nabla_S$, 
$\divv_S$ and $\curl_S$ respectively.
The gradient acts component-wise on tuples of scalar 
functions, whereas the divergence and curl act vector-wise on tuples of vector fields.
In polar coordinates, the $\R^{1+n}$ differential operators are
\begin{gather*}
   \nabla_\bx u=  (\pd_r u_r)\rad + r^{-1} \nabla_S u_r, \\
   \divv_\bx f= r^{-n} \pd_r \big(r^n (f_r)_\no\big)+ r^{-1} \divv_S (f_r)_\ta, \\
   \curl_\bx f=r^{-1} \rad\wedg \big(\pd_r (r (f_r)_\ta)-\nabla_S (f_r)_\no\big) + r^{-1} \curl_S (f_r)_\ta.
\end{gather*}

We use the boundary function space $L_2(S^n; \V)$, writing the norm $\|\cdot\|_2$, 
of $L_2$ sections of the complex vector bundle
$$
  \V:= \begin{bmatrix} \C^m \\ (T_\C S^n)^m \end{bmatrix}
$$
over $S^n$, where $\C^m$ is identified with the trivial vector bundle and  $T_\C S^n$ denotes the complexified tangent bundle of $S^n$. 
The elements of this bundle are written in vector form 
$f=\begin{bmatrix} \alpha \\ \beta \end{bmatrix}=\begin{bmatrix} \alpha & \beta \end{bmatrix}^t$,
and we write $f_\no:= \alpha$, $f_\ta:= \beta$ for the normal component and tangential part.
Note that $\V$ is isomorphic to the trivial vector bundle $\C^{(1+n)m}$, when
identifying scalar, \textit{i.e.} $\C^m$-valued, functions and $m$-tuples of radial vector fields.
More precisely, the isomorphism is 
$\V\ni\begin{bmatrix} \alpha & \beta \end{bmatrix}^t\mapsto \alpha \rad+ \beta\in\C^{(1+n)m}$, 
for $\alpha\in\C^m$ and $\beta\in (T_{\C}S^n)^m$.

The differential operators on $S^n$ can be seen as unbounded operators. We use  $\dom(A), \ran(A) , \nul(A)$ for the domain, range and null space respectively of unbounded operators. Then  $$\nabla_{S}: L^2(S^n;  \C^m) \to L_{2}(S^n; (T_\C S^n)^m)$$
and its adjoint $$-\divv_{S}: L_{2}(S^n; (T_\C S^n)^m) \to L^2(S^n;  \C^m),$$
with domains $\dom(\nabla_{S})=W^1_{2}(S^n; \C^m)$  and 
$\dom(\divv_{S})=\sett{g\in L_{2}(S^n; (T_\C S^n)^m)}{\divv_{S}g\in  L^2(S^n;  \C^m)}$ are  closed unbounded operators with closed range. 
 The condition $g \in \ran(\divv_S)= \nul(\nabla_S)^\perp$ is that $\int_{S^n} g(x)dx=0$ so $\ran(\divv_S)$ is of codimension $m$ in $L^2(S^n;  \C^m)$.  Also when $n\ge 2$,  $ \ran(\nabla_S)= \nul(\curl_S)$,  and  when $n=1$,  $g\in \ran(\nabla_S)$ if and only if $\int_{S^1} (g(x),\ang) dx=0$.
 Thus $\ran(\nabla_S)$ is of codimension $m$ in $L^2(S^1;  \C^m)$ when $n=1$, and infinite codimension when $n\ge 2$.

\begin{defn}    \label{defn:DNops}
  In $L_2(S^n;\V)$, we define operators
$
  D:= 
    \begin{bmatrix} 0 & -\divv_S  \\ 
     \nabla_S & 0 \end{bmatrix}
$
and
$
  N:= 
    \begin{bmatrix} -I & 0  \\ 
     0 & I \end{bmatrix},
$
where $\dom(D):=\begin{bmatrix} \dom(\nabla_S)  \\ 
     \dom(\divv_S) \end{bmatrix}$.
 Write $N^+f:= \tfrac 12(I+N)f= \begin{bmatrix} 0  \\ 
   f_\ta \end{bmatrix},$ and $N^-f:= \tfrac 12(I-N)f= \begin{bmatrix} f_{\no}  \\ 
   0 \end{bmatrix}$. 
\end{defn}

A basic observation is that the two operators $D$ and $N$ anti-commute, \textit{i.e.}
$$
  ND= -DN.
$$
Of fundamental importance in this paper are the closed orthogonal subspaces
$$
  \mH:= \ran(D)= \begin{bmatrix} \ran(\divv_S) \\ \ran(\nabla_S) \end{bmatrix} \qquad\text{and}\qquad
\mH^\perp:=\nul(D)= \begin{bmatrix} \nul(\nabla_S) \\ \nul(\divv_S) \end{bmatrix}.
$$
We consistently denote by $P_{\mH}$ the orthogonal projection onto $\mH$. 
We remark that  
$$
 N^+\mH^\perp=\left\{ \begin{bmatrix} 0  \\ 
   f_\ta \end{bmatrix}, \divv_{S}  f_\ta = 0\right\}
  \qquad
  \text{and} 
  \qquad
     N^-\mH^\perp = \left\{ \begin{bmatrix} c  \\ 
   0 \end{bmatrix}, c\in \C^m \right\},
   $$
   constants being identified to constant functions. 
 It can be checked that \eqref{eq:accrasgarding}
 is equivalent to 
 $A$ is  {\em strictly accretive} on 
\begin{equation}     \label{eq:H1space}
  \mH_1:=\sett{g\in L_2(S^n;\C^{(1+n)m})} {g_\ta \in \ran( \nabla_S)},
\end{equation}
uniformly for  a.e. $ r\in(0,1)$. 
More precisely, the accretivity assumption on $A$ rewrites 
\begin{equation}   \label{eq:accrassumption}
  \sum_{i,j=0}^n\sum_{\alpha,\beta=1}^m \int_{S^n} \re (A_{i,j}^{\alpha,\beta}(rx)g_j^\beta(x) \conj{g_i^\alpha(x)}) dx\ge \kappa 
   \sum_{i=0}^n\sum_{\alpha=1}^m \int_{S^n} |g_i^\alpha(x)|^2dx,
\end{equation}
for all $g\in\mH_1$, a.e. $r\in(0,1)$. In fact, as we shall see in Lemma~\ref{lem:pointwiseaccr} this is equivalent to pointwise strict accretivity when $n=1$ (unit disk), but this is in general not the case when $n\ge 2$ except if $m=1$ (equations).

Using the notation above, we can  identify $\mH_1$   with $$
\begin{bmatrix} L_2(S^n;\C^m)\\ \ran(\nabla_S)\end{bmatrix}$$
and see that $\mH$ is a subspace of codimension $m$ in $\mH_1$.

 On identifying $\C^{(1+n)m}$ with $\V$,  the space of coefficients $L_\infty(\bO^{1+n}; \mL(\C^{(1+n)m}))$  identifies with  
$L_\infty(\bO^{1+n}; \mL(\V))$, so that we can split any coefficients $A$ as
$$
  A(rx)= \begin{bmatrix} A_{\no\no}(rx) & A_{\no\ta}(rx) \\ A_{\ta\no}(rx) & A_{\ta\ta}(rx) \end{bmatrix},
$$
with  $A_{\no\no}(rx)\in \mL(\C^m;\C^m)$,  $A_{\no\ta}(rx)\in \mL((T_{x}S^n)^m, \C^m)$, $A_{\ta\no}(rx)\in \mL(\C^m, (T_{x}S^n)^m)$ and $A_{\ta\ta}(rx)\in \mL((T_{x}S^n)^m, (T_{x}S^n)^m)$. Note also that $A_{\no\no}(rx)= (A(rx)\rad,\rad).$

With our accretivity assumption (\ref{eq:accrassumption}), the component $A_{\no\no}(r\cdot)$ seen as a multiplication operator is  invertible on $L_2(S^n;\C^m)$, thus as a matrix function it is invertible in $L_{\infty}(S^n; \C^m)$.
This is the reason why strict accretivity on $\mH_1$ is needed, and not only on $\mH$, so that 
the transformed coefficient matrix $\hat A$ below can be formed in the next result. We make the  above identification for coefficients $A$ without mention. 

We can now state the two  results on which our analysis stands. 
Proposition~\ref{prop:divformasODE} reformulates this Cauchy-Riemann system (\ref{eq:firstorderdiv})
further, by solving for the $r$-derivatives, as the vector-valued ODE (\ref{eq:firstorderODE}) for the conormal gradient $f$ defined below. This formulation is well suited for the Neumann and regularity problems.
For the Dirichlet problem, we use instead a similar first order system formulation of the 
equation; see Proposition~\ref{prop:ODEtoPotential}. As explained in \cite[Sec.~3]{AA1},
the vector-valued potential $v$ appearing there should be thought of as containing some
generalized conjugate functions as tangential part. 
In the case of the unit disk, we make this rigorous in Section~\ref{sec:disk} and come back to this in Section~\ref{sec:conjugate}. The fundamental object is the following. 

\begin{defn}   \label{defn:gradtoconormal}
The conormal gradient of a weak solution $u$ to  $\divv_{\bx} A \nabla_{\bx} u=0$ in $\bO^{1+n}$ is the section $f: \R^+ \times S^n \to \V$ defined by
\begin{equation}
\label{eq:gradtoconormal}
f_t  = e^{-(n+1)t/2}\begin{bmatrix} (Ag_r)_\no \\ (g_r)_\ta \end{bmatrix}, 
 \end{equation} 
 where $r=e^{-t}$ and $g=\nabla_{\bx} u$. The map $g_{r}\mapsto f_{t}$  is called the gradient-to-conormal gradient map. 
\end{defn} 

\begin{prop}  \label{prop:divformasODE}  
  The pointwise transformation 
$$
   A\mapsto \hat A:=    \begin{bmatrix} A_{\no\no}^{-1} & -A_{\no\no}^{-1} A_{\no\ta}  \\ 
     A_{\ta\no}A_{\no\no}^{-1} & A_{\ta\ta}-A_{\ta\no}A_{\no\no}^{-1}A_{\no\ta} \end{bmatrix}
$$
is a self-inverse bijective transformation of the set of bounded matrices which are strictly accretive on $\mH_1$.

For a pair of coefficients $A\in L_\infty(\bO^{1+n}; \mL(\C^{(1+n)m}))$ and $B\in L_\infty(\R_+\times S^n; \mL(\V))$
which are strictly accretive on $\mH_{1}$ and such that $B= \hat A$, 
the gradient-to-conormal gradient map 
gives a one-to-one correspondence, with inverse the conormal gradient-to-gradient map
\begin{equation}
\label{eq:conormaltograd}
f_{t} \mapsto g_r= r^{-\frac{n+1}{2}}( (B f_t)_\no \rad+  (f_t)\ta),
\end{equation}
where $t=\ln(1/r)$, 
between solutions $g\in L_2^\loc(\bO^{1+n};\C^{(1+n)m})$ to the Cauchy-Riemann system 
\eqref{eq:firstorderdiv}
in $\bO^{1+n}\setminus \{0\}$ distribution sense, and solutions $f\in L_2^\loc(\R_+;\mH)$, with $\int_1^\infty\|f_t\|_2^2 dt<\infty$, to the equation
\begin{equation}  \label{eq:firstorderODE}
  \pd_t f+ (DB+\tfrac{n-1}2 N) f=0,
\end{equation}
in $\R_+ \times S^n$ distributional sense.
\end{prop}

Recall that the Ricci curvature of $S^n$ is $n-1$, so the constant $\frac{n-1}{2}$ is related to curvature.
On the other hand, the exponent  $\frac{n+1}{2}$ appearing in the correspondence $g_r\leftrightarrow f_t$
is the only exponent for which no powers of $r$ remain in equation (\ref{eq:firstorderODE}).
It turns out that this also makes the  gradient-to-conormal gradient map an $L_{2}$ isomorphism since
\begin{equation}  \label{eq:dfl2iso}
   \int_{\bO^{1+n}} |g|^2 d\bx \approx \int_0^1 \|g_r\|^2_2 r^n dr \approx \int_0^\infty \|f_t\|_2^2 dt.
\end{equation}

\begin{proof}
  The stated properties of the matrix transformation are straightforward to verify, using the observation
  that $e^{(n+1) t}\re(B_t f_t, f_t)= \re(A_r g_r, g_r)$.
  See \cite[Prop. 4.1]{AA1} for details.
  
(i) 
 Assume first that the equations (\ref{eq:firstorderdiv}) hold on $\bO^{1+n}\setminus\{0\}$.
  In polar coordinates $\bx=rx$, the equations $\divv_{\bx}(Ag)=0, \curl_{\bx}(g)=0$ give
\begin{equation*}    \label{eq:splitdivform}
\begin{cases}  r^{-n} \pd_r (r^n (Ag)_\no) + r^{-1} \divv_S(Ag)_\ta=0, \\  \pd_r (r g_\ta) - \nabla_S g_\no=0.
\end{cases}
\end{equation*}
Next we pull back the equations to $\R_+\times S^n$.
Write $(Ag)_\no= r^{-(n+1)/2} f_\no$ and  $(Ag)_\ta= A_{\ta\no}g_\no + A_{\ta\ta}g_\ta$.
Then $g_\no= r^{-(n+1)/2} A_{\no\no}^{-1}(f_\no-A_\no\ta f_\ta)$ and $g_\ta=r^{-(n+1)/2} f_\ta$,
and
the equations further become
$$
\begin{cases}
  r^{-n} \pd_r (r^{(n-1)/2} f_\no) + r^{-(n+3)/2} \divv_S( B_{\ta\no} f_\no + B_{\ta\ta} f_\ta )=0, \\
  \pd_r (r^{(1-n)/2} f_\ta)- r^{-(n+1)/2} \nabla_S( B_{\no\no}f_\no + B_{\no\ta} f_\ta)=0.
\end{cases}
$$
Using product rule for $\pd_r$ and the chain rule $-r\pd_r= \pd_t$, this yields the equation 
(\ref{eq:firstorderODE}).

It remains to check that $f_{t}\in \mH$ for almost every $t>0$. This is equivalent to $(A_{r}g_{r})_\no\in \ran(\divv_S)$ and $(g_{r})_{\ta}\in \ran(\nabla_{S})$ for a.e. $r\in (0,1)$.
To see $(A_{r}g_{r})_\no\in \ran(\divv_S)$ amounts  to seeing that $\int_{S^n} (A_rg_r)_\no dx=0$.  We apply Gauss' theorem as follows.
For any radial function $\phi\in C_0^\infty(\bO^{1+n};\C^m)$, the divergence equation gives
$\int_{\bO^{1+n}} (  Ag, \nabla\phi )d\bx=0$.
Taking, for a.e. $r\in (0,1)$,
 the limit as $\phi$ approaches the characteristic function for balls $\{|\bx|<r\}$ shows  
that $\int_{rS^n} (Ag)_\no dx=0$. To check $(g_{r})_{\ta}\in \ran(\nabla_{S})$ we distinguish first $n=1$. In that case, a similar application of Stokes' theorem shows that $\int_{S^1} (\ang, g_r) dx=0$
for a.e. $r\in (0,1)$. For $n\ge 2$, that $\curl_S ((g_r)_\ta)=0$ is a consequence of $\curl_\bx g=0$ and the general fact that pullbacks and the exterior derivative commute.
Hence $f_t\in\mH$.

(ii)
Conversely, assume that equation (\ref{eq:firstorderODE}) holds and $f_t\in\mH$ for a.e.~$t>0$.
Define the corresponding function $g\in L_2^\loc(\bO^{1+n};\C^{(1+n)m})$ by the conormal gradient-to-gradient map and note that $\curl_S ((g_r)_\ta)=0$.
Reversing the rewriting of the equations in (i) shows that  $\divv_{\bx}(Ag)=0, \curl_{\bx}(g)=0$
hold on $\bO^{1+n}\setminus\{0\}$. 
This proves the proposition.
\end{proof}

\begin{cor}\label{cor:conormal} For any coefficients $A\in L_\infty(\bO^{1+n}; \mL(\C^{(1+n)m}))$ which are strictly accretive in the sense of \eqref{eq:accrasgarding},
gradients of weak solutions to \eqref{eq:divform} in $\bO^{1+n}$ are in one-to-one correspondence with $\R_{+}\times S^n$ distributional  solutions   to the   equation  \eqref{eq:firstorderODE}, belonging to
$L_2^\loc(\R_+;\mH)$ with estimate  $\int_1^\infty\|f_t\|_2^2 dt<\infty$.
\end{cor}

\begin{proof} Combine
 Proposition~\ref{prop:CR} and Proposition~\ref{prop:divformasODE}.
 \end{proof}

There is a second way of constructing weak solutions that we now describe.

\begin{prop}  \label{prop:ODEtoPotential}
Let $A$ and $B= \hat A$ be as in Proposition~\ref{prop:divformasODE}.
Assume that $v\in L_2^\loc(\R_+;\dom(D))$  with $\int_1^\infty \|Dv_t\|_2^2 dt<\infty$
satisfies 
\begin{equation}   \label{eq:BDODE}
\pd_t v+ (BD-\tfrac {n-1}2 N)v=0
\end{equation}
in $\R_+\times S^n$ distributional sense.
Then
$$
  u_r:= r^{-\frac{n-1}{2}} (v_{t})_\no,\qquad r=e^{-t} \in (0,1),
$$
extends to  a weak solution of $\divv_\bx A \nabla_\bx u=0$ in $\bO^{1+n}$,
and $Dv$ equals the conormal gradient of $u$.
\end{prop}

\begin{proof}
  By definition of $u$ in the statement, 
 $(Dv)_\ta= \nabla_S v_\no= r^{(n+1)/2}(r^{-1}\nabla_S u)$. 
On the other hand, taking the normal component of $\pd_t v+ (BD-\tfrac {n-1}2 N)v=0$ gives
$$
  \pd_t v_\no-A_{\no\no}^{-1} ( \divv_S v_\ta + A_{\no\ta}\nabla_Sv_\no) + \sigma v_\no=0,
$$
or equivalently
\begin{multline*}
  (Dv)_\no= -\divv_S v_\ta= -A_{\no\no}(\pd_t +\sigma) v_\no + A_{\no\ta}\nabla_S v_\no \\
  = r^{(n+1)/2} (A_{\no\no} \pd_r u+ A_{\no\ta}r^{-1} \nabla_S u)= r^{(n+1)/2} (A \nabla_\bx u)_\no.
\end{multline*}
These equations hold in $\bO^{1+n}\setminus \{0\}$.
Next, applying $D$  to \eqref{eq:BDODE}
yields 
$$
  (\pd_t + DB+\tfrac{n-1}2N)(Dv)=0.
$$
Thus   $f:=Dv$  satisfies \eqref{eq:firstorderODE} and $f_{t}\in \ran(D)=\mH$. By Corollary~\ref{cor:conormal}, there is a weak solution  $\tilde u$  in $\bO^{1+n}$ of the divergence form equation associated to $f$. In particular, $f_{\ta}=r^{(n+1)/2}(r^{-1}\nabla_S \tilde u)$ and $f_{\no}=r^{(n+1)/2} (A \nabla_\bx \tilde u)_\no$.
   Applying the  conormal gradient-to-gradient map, we deduce $\nabla_{\bx}\tilde u=\nabla_{\bx}u$ in $\bO^{1+n}\setminus \{0\}$ distribution sense. 
   In particular, $u=\tilde u +c$ in 
   $\bO^{1+n}\setminus \{0\}$
 for some constant $c$. As $\tilde u +c$ is also a weak solution in   $\bO^{1+n}$ to the divergence form equation with coefficients $A$,  this provides us with the desired extension for $u$.    \end{proof}

For perturbations $A$ of radially independent coefficients, 
Corollary~\ref{cor:diransatz}(i) proves a converse of this result, \textit{i.e.}~the existence of 
such a vector-valued potential $v$ containing a given solution $u$ to $\divv_\bx A \nabla_\bx u=0$
as normal component. We do not know whether such $v$ can be defined for general coefficients (except in $\bO^2$, see Section~\ref{sec:disk}). 

\begin{rem}\label{rem:extension}
Assume that the coefficients $A$ are defined in $\R^{1+n}$ and that the accretivity condition \eqref{eq:accrasgarding} or (\ref{eq:accrassumption}) holds for a.e $r\in (0,\infty)$. 
As in Proposition~\ref{prop:divformasODE},
there is also a one-to-one correspondence between solutions $g\in L_{2}^{loc}(\R^{1+n}\setminus\clos{\bO^{1+n}}; L_{2}(S^n;\V))$ to
$ \divv_\bx (Ag)=0$, $\curl_\bx g=0$ in the exterior 
of the unit ball and solutions $f:\R_-\to \mH$ to the equation
$\pd_t f+ (DB+\tfrac{n-1}2 N) f=0$ for $t<0$ in $L_{2}(\R_{-}; \mH)$.
Also, as in Proposition~\ref{prop:ODEtoPotential}, $L_{2}^{loc}$-solutions $v:\R_-\to L_2(S^n; \dom(D))$ to the equation
$\pd_t v+ (BD-\tfrac{n-1}2 N) v=0$ for $t<0$, give weak solutions $u$ to $\divv_\bx A \nabla_\bx u=0$
in the exterior  of the unit ball.
\end{rem}

%
%
%
%
%
\section{Study of the infinitesimal generator}   \label{sec:infgene}

In this section, we study the {\em infinitesimal generators} 
$D B_0+ \tfrac {n-1}2 N$ and $B_0 D- \tfrac {n-1}2 N$ 
for the vector-valued ODEs appearing in (\ref{eq:firstorderODE}) and \eqref{eq:BDODE}
for radially independent coefficients $B_0= \widehat{A_1}\in L_\infty(S^n;\mL(\V))$, strictly accretive on 
$\mH$ with constant $\kappa= \kappa_{B_0}>0$. 
Note that strict accretivity of $A_1$ on $\mH_1$ is needed for the construction of $B_0= \widehat{A_1}$
as a mutiplication operator.
Once we have $B_0$, only strict accretivity of $B_0$ on $\mH$ is needed in our analysis. This has the following consequences used often in this work. First, $B_{0}:\mH \to B_{0}\mH$ is an isomorphism. Second, the map $P_{\mH}B_{0}$ is an isomorphism of $\mH$.

The first operator will be used to get estimates of $\nabla_\bx u$, needed for the Neumann and 
regularity problems.
The second operator will be used to get estimates of the potential $u$, needed for the Dirichlet problem.

\begin{defn}\label{defn:operators} 
   Let $\sigma\in \R$.
   Define the unbounded linear operators
$$
  D_{0} := DB_0 +\sigma N \qquad\text{and}\qquad
  \tD_{0} := B_0D -\sigma N
$$
in $L_2(S^n;\V)$, 
with domains $\dom(D_0):= B_0^{-1} \dom(D)$ and
$\dom(\tD_0):= \dom(D)$ respectively.
Here $B_0^{-1}(X):= \sett{f\in L_2}{B_0 f\in X}$.
When more convenient,  we use the  notation $D_{A_1}:=D_0$ and $ \tD_{A_1}:=\tD_0$.
\end{defn}

For these two operators, we have the following intertwining and duality relations.

\begin{lem}     \label{lem:intertwduality}
In the sense of unbounded operators, we have
$D_{0}D= D \tD_{0}$ and 
$(\tD_{A_{1}})^* = DB_0^*-\sigma N= -N (D\widehat{A_1^*}+\sigma N)N$.
\end{lem}

\begin{proof}
  The proof is straightforward, using the identity $B_0^*= N\widehat{A_1^*}N$ for the second statement.
\end{proof}

\begin{prop}   \label{prop:DBprops}
In $L_2=L_2(S^n; \V)$, the operator $D_0$ is a closed unbounded 
operator with dense domain.
There is a topological Hodge splitting 
$$
  L_2= \mH\oplus B_0^{-1}\mH^\perp,
$$
i.e the projections $P^1_{B_0}$ and $P^0_{B_0}$ onto $\mH$ and 
$B_0^{-1}\mH^\perp$ in this splitting are bounded.
The operator $D_0$ leaves $\mH$ invariant, and the restricted operator
$D_0:\mH\to\mH$, with domain $\dom(D_0)\cap\mH$, 
is closed, densely defined, injective, onto, and has a compact inverse.

If $\sigma\ne 0$, then $D_0:L_2\to L_2$ is also injective and onto,
and $D_0|_{B_0^{-1}\mH^\perp}= \sigma N$.

If $\sigma=0$, then $D_0= D B_0$, $\nul(D_0)= B_0^{-1}\mH^\perp$ and $\ran(D_0)=\mH$
are closed and invariant.
In particular, when $n=1$, $\dim\nul(D_0)=2m= \dim(L_2/\ran(D_0))$.
\end{prop}

\begin{proof}
The splitting is a consequence of the strict accretivity of $B_0$ on $\mH$, and it is clear that $\mH$
is invariant under $D_0$.
Note that 
$$
  (iN) (DB_0+ \sigma N)= (iND)B_0+ i\sigma,
$$
where $iN$ is unitary on $L_2$ as well as $\mH$, and where $iND= -iDN$ is a self-adjoint operator
with range $\mH$.
This shows that $D_0$ is closed, densely defined, injective and onto on $\mH$, 
and on $L_2$ when $\sigma\ne 0$, as a consequence of
properties of operators such as $(iND)B_0$ stated in \cite[Prop. 3.3]{AAM}.

Next we show that $D_0:\mH\to\mH$ has a compact inverse.
Write
$
  D_{0}= D(P_\mH B_0) + \sigma N.
$
  Since  $P_\mH B_0$ is an isomorphism on $\mH$,  
  it suffices to prove that the inverse of $D:\mH\to \mH$ is compact.
  Note that $\dom(\nabla_S)= W^1_{2}(S^n;\C^m)$ is compactly embedded in $L_2(S^n;\C^m)$
  by Rellich's theorem. In particular 
  $\nabla_S: \ran(\divv_S)\to \ran(\nabla_S)$ has compact inverse.
  Since $\nabla_S^*=-\divv_S$, 
  it follows that $\divv_S: \ran(\nabla_S)\to \ran(\divv_S)$ has a compact 
  inverse as well.
   This proves that the inverse of $D$ is compact on 
  $\mH$.
  
  The remaining properties when $\sigma \ne 0$ and $\sigma =0$
   are straightforward and are left to the reader.
\end{proof}

\begin{prop}    \label{prop:BDprops}
In $L_2=L_2(S^n; \V)$, the operator $\tD_0$ is a closed unbounded 
operator with dense domain.
There is a topological Hodge  splitting 
$$
  L_2= B_0 \mH\oplus \mH^\perp,
$$
i.e the projections $\tP^1_{B_0}$ and $\tP^0_{B_0}$ onto $B_0\mH$ and 
$\mH^\perp$ in this splitting are bounded.
Here $\mH^\perp\subset\dom(\tD_0)$ and $\tD_0$ leaves $\mH^\perp$ invariant.

If $\sigma\ne 0$, then $\tD_0:L_2\to L_2$ is also injective and onto,
and $\tD_0|_{\mH^\perp}= -\sigma N$.

If $\sigma=0$, then $\tD_0=B_0D$, $\nul(\tD_0)= \mH^\perp$ and $\ran(\tD_0)=B_0 \mH$ is closed.
In particular,  the subspace $B_0 \mH$ is invariant under $\tD_0$ and 
when $n=1$,  $\dim\nul(\tD_0)=2m= \dim(L_2/\ran(\tD_0))$.
\end{prop}

\begin{proof}
  These results for $\tD_0$ follow from Proposition~\ref{prop:DBprops} by
  duality, using Lemma~\ref{lem:intertwduality}.
\end{proof}

\begin{rem}
  The reader familiar with \cite{AKMc, AA1} should carefully note the following fundamental 
  difference between the cases $\sigma \ne 0$ and $\sigma =0$. 
  When $\sigma=0$, each of the operators $D_0$ and $\tD_0$ is of the type considered in \cite{AKMc, AA1},
  and each has two complementary invariant subspaces.
  On the other hand when $\sigma\ne 0$, the operator $D_0$ 
  has in general only the invariant subspace
  $\mH$, and $\tD_0$ only has the invariant subspace $\mH^\perp$.
  One can define an induced operator $\tD_0$ 
  on the quotient space $L_2/\mH^\perp$, but this cannot 
  be realized as an action in a subspace complementary to $\mH^\perp$ in $L_2$
  in general.
   As $\sigma$ will be set to $ \frac{n-1}{2}$ this means  for us a difference in the treatment of
   $n=1$ (\textit{i.e.} space dimension 2) and $n\ge 2$ (\textit{i.e.}  space dimension $3$ and higher). 
\end{rem} 

We prove here a technical lemma for later use.

\begin{lem}\label{lem:htotildeh} There is a unique isomorphism
\begin{equation}   \label{eq:htotilde}
  \mH\to L_2 / \mH^\perp : h\mapsto \tilde h
\end{equation}
such that $D_0 h = D \tilde h$ for $h\in \mH\cap\dom (D_0)$.
\end{lem}

\begin{proof} When $\sigma=0$, we can take 
$\tilde h:= B_0h  \in B_{0}\mH \approx L_{2}/\mH^\perp$ as $D_{0}h=DB_{0}h=D\st h$.

When $\sigma\ne 0$, we use that $D: L_{2}\to \mH$ is surjective with null space $\mH^\perp$. 
This defines $\st h$ for $h\in \mH\cap \dom(D_{0})$. With   $D^{-1}$  the
compact inverse of $D: \mH\to \mH$,
the equation $D_0 h = D \tilde h$ is equivalent to
\begin{equation}   \label{eq:ptildeeqh}
  P_{\mH}B_0 h + \sigma  D^{-1}Nh = P_{\mH} \tilde h.
\end{equation}
This shows that (\ref{eq:htotilde}) extends to a bounded map since 
$\|\tilde h\|_{L_2/\mH^\perp}\approx \|P_{\mH}\tilde h\|_2$. 
Moreover, since $P_{\mH}B_0$ is an isomorphism on $\mH$, we have also the lower bound
$\|h\|_2 \lesssim \|P_{\mH}B_0 h\|_2 \lesssim \|\tilde h\|_{L_2/\mH^\perp}+ \|D^{-1}h\|_2$,
which shows that  (\ref{eq:htotilde}) is a semi-Fredholm operator.
If $\tilde h=0$, then (\ref{eq:ptildeeqh}) implies $h\in \mH\cap\dom (D_0)$.
Therefore $D_0 h=0$ and (\ref{eq:htotilde}) is injective.
Since the range contains the dense subspace $\dom(D)/\mH^\perp$,
invertibility follows.
\end{proof}

%
%
%
%
%
\section{Elliptic systems in the unit disk}  \label{sec:disk}

In dimension $n=1$, \textit{i.e.} for the unit disk $\bO^2\subset \R^2$ with boundary $S^1$, some special phenomena
occurs. In this section we collect these results.

\begin{lem}   \label{lem:pointwiseaccr}
  If $n=1$ and $A$ is strictly accretive in the sense of \eqref{eq:accrasgarding}, then $A$ is pointwise strictly accretive, \textit{i.e.}
$$
  \re(A(\bx)v,v)\ge \kappa |v|^2,\qquad \text{for all } v\in\C^{2m}, \text{ and a.e. }  \bx\in \bO^2.
$$
\end{lem}

\begin{proof}
  By scaling and continuity, it suffices to consider 
  $v=\begin{bmatrix} (z_\alpha) & (w_\alpha) \end{bmatrix}^t \in \C^{2m}$, with $w_\alpha\ne 0$,
  $\alpha=1,\ldots, m$.
  In (\ref{eq:accrasgarding}), let
$$
  u^\alpha(re^{i\theta}):= (ik)^{-1}w_\alpha e^{ik\frac r{r_0} \frac{z_\alpha}{w_\alpha}} \eta(e^{i\theta})e^{ik\theta}, \qquad \alpha=1,\ldots, m, 
$$
with a smooth function $\eta:S^1\to\R$, $k\in \Z_{+}$ and $r_0\in (0,1)$.
Using polar coordinates and letting $k\to \infty$ yields
$$
   \re\int_{S^1}\left(A(r_0x)v, v\right) 
   |\eta(x)|^2 dx
   \ge \kappa |v|^2\int_{S^1}|\eta(x)|^2 dx,  \qquad \text{ for a.e. } r_0\in (0,1).
$$
Taking $|\eta|^2$ to be an approximation to the identity at a given point $x\in S^1$ now proves
the pointwise strict accretivity in the statement.
\end{proof}

\begin{defn}    \label{defn:conjugate}
 Assume that $A\in L_\infty(\bO^2; \mL(\C^{2m}))$ is  pointwise strictly accretive.
Given a weak solution $u\in W_2^{1,\loc}(\bO^2;\C^{m})$ to $\divv_\bx A \nabla_\bx u=0$, 
we say that a solution $\tilde u\in W_2^{1,\loc}(\bO^2;\C^{m})$ to
$
  J\nabla_\bx \tilde u= A \nabla_\bx u
$
is a {\em conjugate} of $u$ {where } $ J:= \begin{bmatrix} 0 & -I \\ I & 0  \end{bmatrix}$.
\end{defn}
We note that since $A\nabla_\bx u$ is divergence-free, there always exists a conjugate of $u$, unique modulo constants in $\C^m
$.
The notion of conjugate solution for two dimensional divergence form equations,
in the scalar case $m=1$, goes back to Morrey. See \cite{morrey:mult}.
Note that when $A=I$, the system $J\nabla_\bx \tilde u= \nabla_\bx u$ is the anti Cauchy--Riemann
equations.

\begin{lem}\label{lem:conjugate}
Assume that $A\in L_\infty(\bO^2; \mL(\C^{2m}))$ is  pointwise strictly accretive. Let $u\in W_2^{1,\loc}(\bO^2;\C^{m})$ be a weak solution  to $\divv_\bx A \nabla_\bx u=0$. Then     
$$A \nabla_\bx u = J\nabla_\bx \tilde u  \Leftrightarrow \begin{cases}  (A \nabla_\bx u)_{\no}=- (\nabla_{\bx}\tilde u)_{\ta}\\ (\st A \nabla_\bx \st u)_{\no}= (\nabla_{\bx}u)_{\ta}\end{cases}  \Leftrightarrow \tilde A\nabla_\bx \tilde u= J^t \nabla_\bx u  \Rightarrow \divv_\bx\tilde A \nabla_\bx \tilde u=0
$$
where $\tilde A$ is  the {\em conjugate coefficient}   defined by
$$
  \tilde A:= J^t A^{-1} J.
$$
We have
 $$
   \tilde A
   = \begin{bmatrix}  (d-ca^{-1}b)^{-1} &  (d-ca^{-1}b)^{-1} ca^{-1}   \\ 
  a^{-1}b(d-ca^{-1}b)^{-1} & a^{-1} +a^{-1} b(d-ca^{-1}b)^{-1} ca^{-1}  \end{bmatrix}
  \quad\text{if}\quad
   A=\begin{bmatrix} a & b \\ c & d \end{bmatrix}.
$$
When $m=1$, this reduces to $\tilde A=(\det A)^{-1}A^t$. 
\end{lem}

Here, we have identified the tangential part $(\cdot)_{\ta}$ with its component along $\ang$. (See below.)  

\begin{proof} 
The equivalences and implication are verified from  $\tilde A= J^t A^{-1} J$. The explicit formula for $\tilde A$ is classical if $m=1$.  If $m\ge 2$, the proposed formula for $\tilde A$ can be checked by a straightforward computation. Note that  $a,b,c,d\in L_{\infty}(\bO^2; \mL(\C^{m}))$ and all the entries of $\tilde A$ as well: the inverses are pointwise multiplications. We omit further details.
\end{proof}

We next show that the vector-valued potential $v$ in Proposition~\ref{prop:ODEtoPotential} contains,
along with $u$ as normal component, its conjugate $\tilde u$ as tangential component. 
To do that, it is convenient to identify $\V$ with the trivial bundle $ \C^{2m}$ by  identifying the tangential component $\beta$  to the tangential part $\beta\ang\in (T_{\C}S^1)^m$.

Given this identification, $D$ becomes
$$
D= \begin{bmatrix} 0 & -\pd_{\ang} \\ \pd_{\ang} & 0  \end{bmatrix},
$$
where $\pd_{\ang}$ denotes the tangential counter clockwise derivative of $m$-tuples
of scalar functions on $S^1$. A coefficient $A \in L_\infty(\bO^2; \mL(\C^{2m}))$ is thus identified with its matrix representation in the moving frame  $\{\rad,\ang\}$. We remark that this identification commutes with the matrix $J$.

\begin{prop}    \label{prop:conjugates}
  Let $A\in L_\infty(\bO^2; \mL(\C^{2m}))$ be pointwise strictly accretive and let 
  $B:= \hat A \in L_\infty(\bO^2; \mL(\C^{2m}))$. Assume that
  $v=\begin{bmatrix} u & \tilde u \end{bmatrix}^t\in L_2^\loc(\R_+;\dom(D))$ with  $\int_1^\infty \|Dv_t\|_2^2 dt<\infty$ is a  $\R_+\times S^n$ distributional solution to 
  $\pd_t v+ BDv=0$ as in Proposition~\ref{prop:ODEtoPotential}, so that $u_r= (v_t)_\no$, $r=e^{-t}$,
  is a weak solution to $\divv_\bx A \nabla_\bx u=0$ in $\bO^2$. Then $\tilde u$ is a conjugate to $u$. 
  
  Conversely, given a weak solution $u$ to $\divv_\bx A \nabla_\bx u=0$ in 
  $\bO^2$ and a conjugate $\st u$, the potential vector $v=\begin{bmatrix} u & \tilde u \end{bmatrix}^t$ has the above properties. 
    \end{prop}

Note that the construction of $v$ this way  is a feature of two-dimensional systems as compared to higher dimensions. 

\begin{proof} 
Applying $J^t$ to  $\pd_t v+ BDv=0$ gives $\pd_t (J^t  v)+ \widetilde B D (J^t v)=0$ with $\widetilde B=J^t BJ$, since $JD=DJ$. A calculation shows that $\widetilde B= \widehat{\tilde A}$. Applying Proposition~\ref{prop:ODEtoPotential} shows that $\tilde u_{r}= (J^t v_t)_\no$ is a weak solution to $\divv_\bx \tilde A \nabla_\bx \tilde u=0$.  Also we know that $Dv$ and $D\tilde v$ are respectively equal to the conormal gradients of $u$ and $\st u$, and since $J^t v=\tilde v$,  this gives the middle term in the equivalence of  Lemma~\ref{lem:conjugate}. Thus  $\tilde u$ is a conjugate of $u$.  The converse is immediate to check and left to the reader. 
\end{proof}

We finish this section with  the following simple expressions for the projections
$P^0_{B_0}$ and $\tP^0_{B_0}$ of Propositions~\ref{prop:DBprops}~and~\ref{prop:BDprops} when $n=1$. We still make the identification $\V\approx \C^{2m}$. 

\begin{lem}   \label{lem:proj1D}  Let $A_{1}\in L_\infty(\bO^2; \mL(\C^{2m}))$ be pointwise strictly accretive radially independent coefficients,  and let 
  $B_{0}:= \widehat A_{1} \in L_\infty(\bO^2; \mL(\C^{2m}))$ the corresponding coefficients. 
Then
$$
  \tP^0_{B_0}g= \left( \int_{S^1} B_0^{-1} \, dx\right)^{-1}  \int_{S^1} B_0^{-1}g \, dx, 
  \qquad g\in L_2(S^1;\C^{2m}),
$$
and $P^0_{B_0}= B_0^{-1} \tP^0_{B_0}B_0$.
\end{lem}

\begin{proof}
By accretivity, $(\int_{S^1} B_0^{-1} \, dx)^{-1}$ is a bounded operator (called the harmonic mean of 
$B_{0}$). 
If $g\in B_0 \mH$, then $B_0^{-1}g\in \mH$ and $\int_{S^1} B_0^{-1}g \, dx=0$, hence $\tP^0_{B_0}g=0$, follows.
On the other hand, if $g\in\mH^\perp$, then $g$ is constant, and therefore
the right hand side equals
$$
  \left( \int_{S^1} B_0^{-1} \, dx\right)^{-1}  \left(\int_{S^1} B_0^{-1}\, dx\right) g=g. 
$$
This proves the expression for $\tP^0_{B_0}$.
The formula for $P^0_{B_0}$ comes from the similarity relation $DB_{0}=B_{0}^{-1}(B_{0}D)B_{0}$.
\end{proof}

%
%
%
%
%
\section{Resolvent estimates}     \label{sec:resolvents}

In this section we prove that the spectra of $D_0$ and $\tD_0$
are contained in certain double hyperbolic regions, 
and we estimate the resolvents.
For parameters $0<\omega<\nu<\pi/2$ and $\sigma\in\R$, 
define closed and open hyperbolic regions in the complex plane by
\begin{align*}
    S_{\omega,\sigma} &:= \sett{x+iy\in\C}{(\tan^2\omega) x^2 \ge y^2+ \sigma^2},  \\
    S_{\nu,\sigma}^o &:=  \sett{x+iy\in\C}{(\tan^2\nu) x^2 > y^2+ \sigma^2}, \\
    S_{\omega,\sigma+} &:= \sett{x+iy\in\C}{(\tan\omega) x \ge (y^2+ \sigma^2)^{1/2}},  \\
    S_{\nu,\sigma+}^o &:=  \sett{x+iy\in\C}{(\tan\nu) x > (y^2+ \sigma^2)^{1/2}}.
\end{align*}
When $\sigma=0$, we drop the subscript $\sigma$ in the notation for the sectorial regions.

\begin{prop}   \label{prop:spectrum} On $L_{2}=L_{2}(S^n;\V)$, 
there is a constant $\omega\in (0, \pi/2)$, depending only on $\|B_0\|_\infty$ and the accretivity
constant $\kappa_{B_0}$, such that the spectra of the operators $D_0$ and $\tD_0$ 
are contained in the double hyperbolic region $S_{\omega,\sigma}$.
Moreover, there are resolvent bounds
$$
  \|(\lambda-D_0)^{-1}\|_{L_2\to L_2}, \|(\lambda-\tD_0)^{-1}\|_{L_2\to L_2}
  \le \frac{1}{\sqrt{y^2+\sigma^2}/\tan\omega- |x|},
$$
for all $\lambda= x+iy\notin S_{\omega,\sigma}$.
These same estimates hold for the restriction
$D_0:\mH\to\mH$.
\end{prop}

\begin{proof}
(i)
To prove the spectral estimates for $D_0$, assume that 
$$
  (DB_0+ \sigma N-x-iy)u=f.
$$
Introduce the auxiliary operator $N_y:= i\sigma N-yI$, and note that 
$\|N_y\|= \|N_y^{-1}\|^{-1}= \sqrt{y^2+\sigma^2}$.
Multiply with $N_y$ and rewrite as
\begin{equation}   \label{eq:resolveq1}
  (N_yD) B_0u + i(y^2+\sigma^2)u= N_y f+ x N_y u.
\end{equation}
Now split the function $u$ as
$$
  u= u_1+u_0\in \mH\oplus B_0^{-1}\mH^\perp,
$$
and note that $\|u\|\approx \|u_1\|+ \|u_0\|$.
Apply the associated bounded projections $P^i_{B_0}$
to (\ref{eq:resolveq1}) to get
\begin{align*} 
  (N_yD) B_0u_1 + i(y^2+\sigma^2)u_1 &= P^1_{B_0}N_y f+ x P^1_{B_0}N_y u, \\
  0+ i(y^2+\sigma^2)u_0 &= P^0_{B_0} N_y f+ x P^0_{B_0}N_y u.
\end{align*}
Take the imaginary part of the inner product between the first equation and $B_0u_1$
(using that $N_y D$ is self-adjoint), and the second equation and $u_0$ to get
\begin{align*} 
  (y^2+\sigma^2) \re (u_1,B_0 u_1) =  \im(P^1_{B_0}N_y f, B_0u_1)+ \im (xP^1_{B_0}N_y u,B_0 u_1), \\
  (y^2+\sigma^2) \|u_0\|^2 =\im(P^0_{B_0}N_y f, u_0)+  \im (xP^0_{B_0}N_y u,u_0).
\end{align*}
Using the strict accretivity of $B_0$ on $\mH$ gives the estimate
$$
  (y^2+\sigma^2)\|u\|^2  \le C_1 \sqrt{y^2+\sigma^2}( \|f\|\|u\|+ |x| \|u\|^2),
$$
for some constant $C_1<\infty$.
Thus $\|u\|\le (\sqrt{y^2+\sigma^2}/C_1-|x|)^{-1} \|f\|$.

(ii)
To prove a similar lower bound on $\tD_0$, assume that 
$(B_0D- \sigma N-x-iy)u=f$,
and rewrite as
\begin{equation}   \label{eq:resolveq2}
   B_0 DN_y^{-1} N_y u + iN_y u= f+ x u.
\end{equation}
Write $N_y u= B_0u_1+u_0\in B_0\mH\oplus \mH^\perp$.
Apply the bounded projections $\tP^i_{B_0}$
to (\ref{eq:resolveq2}) to get
\begin{align*} 
  B_0 (DN_y^{-1}) B_0u_1 + iB_0 u_1 &= \tP^1_{B_0} f+ x \tP^1_{B_0} u, \\
  0+ iu_0 &= \tP^0_{B_0} f+ x \tP^0_{B_0} u.
\end{align*}
Recall that $B_0: \mH\to B_0\mH$ is an isomorphism and 
apply its inverse $B_0^{-1}: B_0\mH\to \mH$ to the first equation.
Then take the imaginary part of the inner product between the first equation and $B_0u_1$
(using that $DN_y^{-1}$ is self-adjoint),
and the second equation and $u_0$ to get
\begin{align*} 
  \re (u_1,B_0 u_1) =  \im(B_0^{-1}\tP^1_{B_0} f, B_0u_1)+ \im (xB_0^{-1} \tP^1_{B_0}u,B_0 u_1), \\
  \|u_0\|^2 = \im(\tP^0_{B_0}f, u_0)+ \im (x\tP^0_{B_0}u, u_0).
\end{align*}
Using the strict accretivity of $B_0$ on $\mH$ gives the estimate
$$
   (y^2+\sigma^2)\|u\|^2  \le C_2( |x| \|u\| + \|f\|) (y^2+\sigma^2)^{1/2} \|u\|,
$$
for some constant $C_2<\infty$.
Thus $\|u\|\le (\sqrt{y^2+\sigma^2}/C_2-|x|)^{-1} \|f\|$.

(iii)
Using that $DB_0 +\sigma N$ and $B_0^* D+ \sigma N$ are adjoint operators,
combining the results in (i) and (ii) shows that both operators $D_0-\lambda$
and $\tD_0-\lambda$ are onto, with bounded inverse, when 
$\lambda\notin S_{\omega,\sigma}$. Here $\omega:= \arctan(\max(C_1,C_2))$.
The estimates on $\mH$ follow.
\end{proof}

We shall also need the following off-diagonal estimates for the resolvents,
both in $L_2$ and in $L_p$ for $p$ near $2$.

\begin{lem}  \label{lem:offdiagonal}
 (i) There exist $\epsilon, \alpha>0$ such that 
  for $|\frac{1}{p}-\frac{1}{2}|<\epsilon$, closed sets $E,F\subset S^n$ and $f\in L_p(S^n;\V)$
  with $\supp f\subset E$ and $t\in \R$,
$$
  \| (I+itD_0)^{-1}f \|_{L_p(F)}\lesssim e^{-\alpha d(E,F)/|t|}  \|f\|_{L_p(E)}, 
$$
where $d(E,F)$ is the distance between the sets $E$ and $F$.

(ii) There exist $q>2$ with $\frac{1}{2}-\frac{1}{q}<\epsilon$, 
and $\alpha>0$ such that  for closed sets $E,F\subset S^n$ and $f\in L_2(S^n;\V)$
  with $\supp f\subset E$  and $f_{\ta}=0$ and $|t|\le 1$,
$$
  \| (I+itD_0)^{-1}f \|_{L_q(F)}\lesssim |t[^{-n(\frac{1}{2}-\frac{1}{q})} e^{-\alpha d(E,F)/|t|}  \|f\|_{L_2(E)}, 
$$
\end{lem}

\begin{proof} 
We first prove (i). 
The case $p=2$ follows  the argument in \cite[Prop.~5.1]{elAAM}. 
It remains to prove $L_{p}$ boundedness for $p$ near 2 as, the $L_{q}$ off-diagonal bounds follow by interpolation with the $L_{2}$ off-diagonal bounds for $q$ between $p$ and $2$.  

For $f\in L_p\cap L_{2}$, we let $h=(I+itD_0)^{-1}f$ and wish to prove
 $\|h\|_p\lesssim \|f\|_p$ when $p$ is near $2$ and uniformly in $t$.
To prove this, we rewrite the equation $(I+itD_0)h=  f$ first as
$(I+it\sigma N + it DB_{0})h=f$ and then in terms of a divergence 
form equation, with coefficients $A_1=\widehat B_0$. 
Write $h=  \begin{bmatrix} (A_1\tilde h)_\no & \tilde h_\ta \end{bmatrix}^t$ and $f= \begin{bmatrix} (A_1\tilde f)_\no & \tilde f_\ta \end{bmatrix}^t$. Then
$$
\begin{cases}
(1-it\sigma)  (A_1\tilde h)_\no - it \divv_S (A_1 \tilde h)_\ta=  (A_1 \tilde f)_\no, \\
 (1+it\sigma) \tilde h_\ta +it  \nabla_S \tilde h_\no = f_\ta.
\end{cases}
$$
Using the second equation to eliminate $\tilde h_\ta$ in the first equation, and letting
$z=(1+ it\sigma)^{-1}$, we obtain
$$
 L\tilde h_\no  
   =    \begin{bmatrix} 1  & -it\conj z \divv_S \end{bmatrix} 
    \begin{bmatrix} \conj z (A_1)_{\no\no} \tilde f_\no + (\conj z-z) (A_1)_{\no\ta} \tilde f_\ta  \\ -z(A_1)_{\ta\ta} \tilde f_\ta \end{bmatrix}
$$
with 
$$
L:=
   \begin{bmatrix} 1  &  -it\conj z\divv_S \end{bmatrix} A_1
    \begin{bmatrix} 1  \\ -itz\nabla_S \end{bmatrix} 
  =   \begin{bmatrix} 1  &  -i\tau\divv_S \end{bmatrix} A_\theta
    \begin{bmatrix} 1  \\ -i\tau\nabla_S \end{bmatrix} 
   $$
   and $A_{\theta}= D_{-\theta}A_{1}D_{\theta}$ with $tz=e^{i\theta} \tau$, $\tau=|tz|$, and $D_{\theta}$ the diagonal matrix with entries $1, e^{i\theta}$ in the normal/tangential splitting. We note that $A_{\theta}$ is strictly accretive on $\mH_{1}$ with the same constants as $A_{1}$, and that $|z|\le 1$ and $|\tau|\le |\sigma|^{-1}$.  
   We claim that   $L$ is invertible from the Sobolev space $W^1_{p}(S^n; \C^m)$ equipped with the scaled norm 
   \begin{equation}   \label{eq:sobnorm}
  \|u\|_{W^1_p} := \left(\int_{S^n} (|u(x)|^2+ |\tau\nabla_S u(x)|^2)^{p/2} dx\right)^{1/p}
\end{equation}
to its dual, with bounds independent of $\tau, \theta$, for $p$ in a neighborhood of 2.

To prove this, if we rescale from the sphere $S^n$ of radius $1$ to the sphere $S_{1/\tau}^n$ of radius $1/\tau$,
we obtain the same equation with $A_{\theta}, A_{1}, z_{\pm}$ unchanged,  $f(x)$, $h(x)$ replaced by $f(\tau x)$, $h(\tau x)$, 
and $\tau\divv_S$, $\tau\nabla_S$ replaced by $\divv_{S^n_{1/\tau}}$, $\nabla_{S^n_{1/\tau}}$,
and we want to show $\|h(\tau\cdot)\|_{L_p(S^n_{1/\tau})}\lesssim \|f(\tau\cdot)\|_{L_p(S^n_{1/\tau})}$
(with implicit constant uniform in $\tau, \theta$). Thus it is enough to set $\tau=1$ and work on $S^n$,
as long as we only use estimates on $S^n$ which hold (with same constant) on $S^n_{1/\tau}$
as well.

Having set $\tau=1$,  we have, for $1<p,q<\infty$ such that $1/p+1/q=1$, estimates
$$
  \|Lu\|_{W^{-1}_p} \le \|A_\theta\|_\infty \|u\|_{W^1_p}=  \|A_1\|_\infty \|u\|_{W^1_p},
$$
where 
$\|u\|_{W^{-1}_p} := \sup_{\|v\|_{W^1_q}=1} |(u,v)|$
and $(u,v)$ denotes the $L_2(S^n; \C^m)$ pairing extended in the sense of distributions.
For $p=q=2$, the accretivity assumption on $A_\theta$ yields
$\|Lu\|_{W^{-1}_2} \ge \kappa \|u\|_{W^1_2}$.
Applying the extrapolation result of {\v{S}}ne{\u\i}berg~\cite{sneiberg} to the 
complex interpolation scale $\{W^1_p\}_{1<p<\infty}$, shows the existence of $\epsilon>0$
such that
$$
  \|Lu\|_{W^{-1}_p} \approx \|u\|_{W^1_p},
$$
for $|\frac 1p-\frac 12|<\epsilon$.
(Even for $\tau\ne 1$, one can verify that the Sobolev norms given by (\ref{eq:sobnorm})
on $S^n_{1/\tau}$ (with $\tau\nabla_S$ replaced  $\nabla_{S^n_{1/\tau}}$) are equivalent to the ones given by the complex interpolation
method, with constant independent of $\tau$. Hence $\epsilon$ depends only on the ellipticity constants and dimension,  and is thus independent of $\tau, \theta$.)
Applying this isomorphism, we obtain the resolvent estimate
$$
  \|h\|_p \approx \|\tilde h_\no \|_p + \|\tilde h_\ta\|_p
  \lesssim \|\tilde h_\no \|_{W^1_p} + \|\tilde f_\ta\|_p
  \lesssim \| f\|_p.
$$
In the second step we used  $\tilde h_\ta=-i e^{i\theta}\nabla_S \tilde h_\no + z\tilde f_\ta$ and $|z|\le 1$ (recall we have rescaled and set $\tau=1$),
and in the third step that $\begin{bmatrix} 1 & -i\divv_S \end{bmatrix}: L_p(S^n; \C^{(1+n)m})\to W^{-1}_p$ is an isometry
since $\begin{bmatrix} 1 & -i\nabla_S \end{bmatrix}^t: W^1_q\to L_q(S^n; \C^{(1+n)m})$ is so.
This finishes the proof of (i).

To prove the inequality (ii), the above argument shows that  $L\tilde h_{\no}= z_{-}(A_{1})_{\no\no}\tilde f_{\no}$ and $\tilde h_{\ta}= -itz_+ \nabla_{S}\tilde h_{\no}$. Having rescaled in the same way,  the Sobolev embedding  $L_{2}\subset W^{-1}_{q}$ for some $q>2$ with $\frac{1}{2}-\frac{1}{q}<\epsilon$, allows us to conclude  that $h\in L_{q}(S^n;\V)$ and since we assume $|t|\le 1$, we have $|\tau|\approx |t|$ and obtain $\|h\|_{q} \lesssim  |t[^{-n(\frac{1}{2}-\frac{1}{q})}  \|f\|_{2}$, the power coming from scaling. 
It suffices to interpolate again with the $L_{2}$ off-diagonal decay, and conclude for any exponent between 2 and $q$. 
\end{proof}

We state the following useful corollary.
Here and subsequently, $\tN^p$ is defined as $\tN$ replacing $L_2$ averages by $L_{p}$ averages and $M$ is the Hardy-Littlewood maximal operator.

\begin{cor} \label{cor:resolventNT} 
For the $\epsilon$ as above and $|\frac{1}{p}-\frac{1}{2}|<\epsilon$, we have the pointwise inequalities
$$
\tN^p((I+itD_{0})^{-1} f) \lesssim M(|f|^p)^{1/p}
$$
$$
\tN^p((I+it\tD_{0})^{-1} f) \lesssim M(|f|^p)^{1/p}
$$
and for some $p<2$ with $\frac{1}{p}-\frac{1}{2}<\epsilon$,
$$
\tN(((I+it\tD_{0})^{-1}f)_{\no}) \lesssim M(|f|^p)^{1/p}.
$$
\end{cor}

\begin{proof} 
We fix a Whitney region $W_{0}=W({t_{0},x_{0}})$ in $\R_{+}\times S^n$. Then   
$|W_0|^{-1}\int_{W_{0}} |(I+itD_{0})^{-1} f(x)|^p dtdx\lesssim M(|f|^p)(x_{0}) $ follows directly from the off-diagonal decay 
of Lemma~\ref{lem:offdiagonal} as in \cite[Prop.~2.56]{AAH}. 
Next, 
$|W_0|^{-1}\int_{W_{0}} |(I+it\tD_{0})^{-1} f(x)|^p dtdx \lesssim M(|f|^p)(x_{0}) $ follows by testing against $g\in L_{q}(W_{0};\V)$, supported in $W_{0}$ with $1/p + 1/q =1$.
We have 
$$
\int_{W_{0}} \Big((I+it\tD_{0})^{-1} f(x), g(t,x)\Big) dtdx = 
 \int_{t_0/c_0}^{c_0t_0} ( f, (I-it\tD_{0}^*)^{-1}g_t ) dt 
$$
so that for each fixed $t$, using that $\tD_{0}^*=DB_{0}^*-\sigma N$ has the same form as $D_{0}$, we can use the $L_{q}$ off-diagonal decay for each $t\approx t_{0}$ and obtain
for any $M>0$, 
\begin{multline} \label{eq:decay0}
|W_0|^{-1}\int_{W_{0}} |(I+it\tD_{0})^{-1} f(x)|^p dtdx\\
 \lesssim \sum_{j\ge 2} 2^{-jM} |B(x_{0},2^jt_{0})|^{-1}\int_{B(x_{0},2^jt_{0})}  | f(x)|^p dx,
\end{multline}
using standard computations on annuli around $B(x_{0},t_{0}) $ in $S^n$. 
Details are left to the reader. 
 
The last estimate starts in the same way with $g\in L_{2}(W_{0};\V)$, but since we want to estimate the normal component of $(I+it\tD_{0})^{-1} f$ we assume that $(g_t)_{\ta}=0$ for each $t$.  The second estimate in Lemma~\ref{lem:offdiagonal}, implies that 
$(I-it\tD_{0}^*)^{-1}g_t=:h_t$ has $L_{q}$ estimates with decay. Thus using H\"older's inequality on $\int_{t_0/c_0}^{c_0t_0}  (f, h_t) dt$ with exponent $q$ on $h_t$ and dual exponent on $f$ yields
\begin{align}
\label{eq:decay}
\bigg(|W_0|^{-1}\int_{W_{0}} |((I &+it\tD_{0})^{-1} f)_{\no}(x)|^2 dtdx\bigg)^{1/2} \\
& \lesssim \sum_{j\ge 2} 2^{-jM} \left(|B(x_{0},2^jt_{0})|^{-1}\int_{B(x_{0},2^jt_{0})}  | f(x)|^p dx\right)^{1/p} \nonumber
\end{align}  
and the conclusion follows.
\end{proof}

%
%
%
%
%
\section{Square function estimates and functional calculus}     \label{sec:squarefcn}

All the remainder of this article rests on the square function estimate below. 

\begin{thm}     \label{thm:QE}
  Let $n\ge 1$.
  The operator $D_0= DB_0+\sigma N$, with $\sigma\in\R$ fixed but arbitrary, has square function estimates
$$
  \int_0^\infty \| t D_0 (1+t^2 D_0^2)^{-1} f \|_2^2 \frac{dt}t \approx \|f\|_2^2,\qquad \text{for all } f\in \ran(D_{0}).
$$
The estimate $\lesssim$ holds for all $f\in L_{2}(S^n,\C^m)$. The same estimates hold for $\tD_0= B_0D-\sigma N$.
\end{thm}

\begin{proof} 
Note that the equivalence can only  hold on  $\clos{\ran(D_0)}=\ran(D_{0})$  which equals  $L_{2}(S^n;\V)$ if $\sigma\ne 0$ and $\mH$ if $\sigma=0$. 
By standard duality arguments, the estimates $\gtrsim$ on $\ran(D_0)$ follows from the estimates $\lesssim$ for $D_{0}^*$. See \cite{ADMc}. Note that $D_{0}^*$ is of type $\tD_{0}$. Hence it is enough to  prove 
\begin{equation}
\label{eq:SF}
 \int_0^\infty \| t D_0 (1+t^2 D_0^2)^{-1} f \|_2^2 \frac{dt}t \lesssim \|f\|_2^2
\end{equation}
  for all $f\in L_2(S^n;\V)$,  and similarly for $\tD_{0}$.  Consider first  the operator $D_0$.

(i) We first reduce \eqref{eq:SF} to 
\begin{equation}
\label{eq:SF1}
 \int_0^1 \| t DB_0 (1+t^2 (DB_0)^2)^{-1} f \|_2^2 \frac{dt}t \lesssim \|f\|_2^2
\end{equation}
 for all $f\in L_2(S^n;\V)$. First note that
$$
   \int_1^\infty \|  t D_0 (I+t^2 D_0^2)^{-1}f\|_2^2 \frac{dt}t 
   \lesssim \int_1^\infty \| t^2 D_0^2 (I+t^2 D_0^2)^{-1}f \|_2^2 \frac{dt}{t^3}\lesssim \int_1^\infty \| f \|_2^2 \frac{dt}{t^3}  
   \approx \|f\|_2^2,
$$
using that $D_0$ has bounded inverse by Proposition~\ref{prop:DBprops}.
(When $n=1$, write $f= f_1+ f_0\in \mH\oplus B_0^{-1}\mH^\perp$. The above estimate goes through for
$f_1$, and the contribution from $f_0$ is zero.)
For the integral $\int_0^1$, we may ignore the zero order term in $D_0$, using the idea from
\cite[Sec. 9]{elAAM}. Indeed,
$$
  \| (I+it D_0)^{-1}f-(I+itDB_0)^{-1}f \|_2 = \|  (I+it D_0)^{-1} it \sigma N(I+itDB_0)^{-1}f \|_2 \lesssim |t|\|f\|_{2}.
$$
Since $2i  t D_0 (I+t^2 D_0^2)^{-1}= (I-it D_0)^{-1} - (I+it D_0)^{-1}$, and similarly for $DB_0$, subtraction yields
$$
  \int_0^1 \|  t D_0 (I+t^2 D_0^2)^{-1} f \|_2^2 \frac{dt}t \lesssim \int_0^1\| t DB_0 (I+t^2 (DB_0)^2)^{-1} f \|_2^2 \frac{dt}t +
  \int_0^1 t dt \|f\|_2^2.
$$

(ii) Next, using a partition of unity, it suffices to show that
\begin{equation}     \label{eq:reducetolocalQE}
  \int_0^1 \| \zeta t DB_0 (I+t^2 (DB_0)^2)^{-1} f \|_2^2 \frac{dt}t \lesssim \|f\|_2^2,
\end{equation}
when $\zeta$ is a smooth cutoff that is 1 on a neighborhood of $\supp f $. 
Indeed,  $L_2$-off diagonal estimates 
of $t DB_0 (1+t^2 (DB_0)^2)^{-1}$ from Lemma~\ref{lem:offdiagonal}  and  again
$2i  t DB_0 (I+t^2 (DB_0)^2)^{-1}= (I-it DB_0)^{-1} - (I+it DB_0)^{-1}$ show in this case that 
$$
  \| (1-\zeta) t DB_0 (I+t^2 (DB_0)^2)^{-1} f \|_2^2  \lesssim  t^2
  \|f\|^2_2.
$$

(iii)  To prove (\ref{eq:reducetolocalQE}), we assume that $f$ and $\zeta$ are supported inside  the lower hemisphere which we parametrize by $\bO^n$ using stereographic coordinates, 
\textit{i.e.} using the map
$$
  \rho: \R^n\to S^n: y\mapsto x= \frac{|y|^2-1}{|y|^2+1}e_0+ \frac {2y}{|y|^2+1},
$$
where $e_0\in \R^{1+n}$ is a fixed unit normal vector to $\R^n\subset\R^{1+n}$,
which covers all $S^n$, except the north pole $e_0\in S^n$.
Note that $\rho$ is a conformal map with length dilation $d^{-1}$ and Jacobian determinant $dx/dy= d^{-n}$, where
$$
  d(y):= (|y|^2+1)/2.
$$
Let $T: \R^n\to \R^{1+n}: y\mapsto \pd_y \rho(y)$ be the differential of $\rho$,
and note that $T^tT= d^{-2} I$.
Define adjoint rescaled pullbacks and pushforwards
\begin{align*}
  \rho^* : L_2(\rho(\bO^n);\V)\to L_2(\bO^n;\C^{(1+n)m}) &: 
  \begin{bmatrix} f_\no & f_\ta \end{bmatrix}^t \mapsto 
  \begin{bmatrix} d^{-n} (f_\no\circ \rho) & T^t (f_\ta\circ\rho) \end{bmatrix}^t, \\
   \rho_* : L_2(\bO^n;\C^{(1+n)m}) \to L_2(\rho(\bO^n);\V) &: 
 \begin{bmatrix} g_\no & g_\ta \end{bmatrix}^t \mapsto 
 \begin{bmatrix} (g_\no\circ \rho^{-1}) &  ( d^n T g_\ta)\circ\rho^{-1} \end{bmatrix}^t.
\end{align*}
Note that $(\rho_*)^{-1}= \begin{bmatrix} d^n & 0 \\ 0 & d^{2-n} \end{bmatrix} \rho^*$.
We claim that 
$$
  \rho^* D = D_\rho \begin{bmatrix} d^n & 0 \\ 0 & d^{2-n} \end{bmatrix} \rho^*, \qquad
  \text{where }  D_\rho :=
  \begin{bmatrix} 0 & -\divv_y \\ \nabla_y & 0 \end{bmatrix}.
$$
Indeed, the tangential part of the equation is the chain rule, and the normal component is the adjoint statement.
We consider $D_\rho$ as a self-adjoint closed unbounded operator in $L_2(\bO^n;\C^{(1+n)m})$ with domain
$\dom(D_\rho):= \begin{bmatrix}H^1_{0}(\bO^n; \C^m) \\ \dom(\divv_y)\end{bmatrix}$,
where $H^1_0$ denotes the Sobolev $W^1_2$ functions vanishing at the boundary $S^{n-1}$.

Next we map coefficients $B_0$ in $\rho(\bO^n)$ to coefficients $B_\rho:= (\rho_*)^{-1} B_0 (\rho^*)^{-1}$
in $\bO^n$, and claim that $B_\rho$ is strictly accretive on $\ran(D_\rho)$.
To see this, let $g\in \ran(D_\rho)$. 
Then $\curl_y g_\ta=0$ and $g_\ta$ is normal on $\partial \bO^n$ (or if $n=1$ we have 
$\int_{-1}^1 g_\ta dy=0$).
Writing $g=\rho^* f$ and extending $f$ by $0$ outside $\rho(\bO^n)$, it follows that
$f\in \mH_1$.
(To see this, write $g_\ta=\nabla_y u$ with $u\in H^1_0(\bO^n; \C^m)$, and extend $u$ by $0$
to an $H^1(\R^n;\C^m)$-function.)
The assumed strict accretivity of $B_0$ on $\mH_1$ gives
\begin{multline*}
 \re \int_{\bO^n} (B_\rho g, g) dy=\re  \int_{\bO^n} ((\rho_*)^{-1} B_0 f, \rho^* f) dy \\
 =\re \int_{S^n}(B_0 f,f) dx
  \ge \kappa \int_{S^n}|f|^2 dx\approx \int_{\bO^n}|g|^2 dy.
\end{multline*}
Thus we obtain a bisectorial operator $D_\rho B_\rho$ in $L_2(\bO^n; \C^{(1+n)m})$, and we observe the intertwining relation
$$
  \rho^* DB_0f=  D_\rho \begin{bmatrix} d^n & 0 \\ 0 & d^{2-n} \end{bmatrix} \rho^*\rho_* B_\rho  \rho^*f= 
  D_\rho B_\rho \rho^*f.
$$
for $f$ supported in the lower hemisphere.
In $\bO^n$, let $K:= \{|y|\le 1/4\}$. By rotational invarince, it is enough to consider those $f=(\rho^*)^{-1}g$
with  $g$ supported on $K$ and  $\zeta=(\rho^*)^{-1}\eta=\eta\circ \rho^{-1}$ with $\eta\in C^\infty_0(\R^n)$ be such that $\eta=1$
on $\{|y|\le 1/2\}$ and $\supp \eta\subset\{|y|\le 3/4\}$. Using  $\eta g=g$ and understanding $\eta, \zeta$ as the operators of pointwise multiplication by $\eta,\zeta$,  one can check the identity 
\begin{multline*}
\zeta (I+it DB_0)^{-1}f  -(\rho^*)^{-1} \eta^2 (I+ it D_\rho B_\rho)^{-1} g \\ =
\zeta (I+it DB_0)^{-1}(\rho^*)^{-1}(\eta g)  -\zeta (\rho^*)^{-1} \eta (I+ it D_\rho B_\rho)^{-1} g \\
 = \zeta (I+it DB_0)^{-1}(\rho^*)^{-1} \left( \eta (I+it D_\rho  B_\rho)- \rho^* (I+it DB_0)(\rho^*)^{-1}\eta  \right) (I+ it D_\rho B_\rho)^{-1} g  \\
   =\zeta (I+it DB_0)^{-1}(\rho^*)^{-1}  it[\eta, D_\rho]  B_\rho (I+ it D_\rho B_\rho)^{-1} g. 
\end{multline*}
As in (i) above, subtracting the corresponding equation with $t$ replaced by $-t$, yields the estimate
$$
  \| \zeta t DB_0 (I+t^2 (DB_0)^2)^{-1} f   - (\rho^*)^{-1} \eta^2 t D_\rho B_\rho (I+t^2 ( D_\rho B_\rho)^2)^{-1} g 
      \|_2\lesssim |t| \|g\|_2,
$$
since  $[\eta, D_\rho]$ is bounded. As $\|f\|_{2}\approx \|g\|_{2}$ by the support conditions,  (\ref{eq:reducetolocalQE}) will follow from 
$$
  \int_0^1 \| t D_\rho  B_\rho (I+t^2 ( D_\rho B_\rho)^2)^{-1} g \|_2^2 \frac{dt}t \lesssim \|g\|_2^2,\qquad \text{for all } g\in L_2(\bO^n; \C^{(1+n)m}).
$$

(iv)
The latter square function estimate
follows from combining \cite[Thm.~2]{AKMc1} and \cite[Prop.~3.1(iii)]{AKMc}, 
the latter purely being of functional analytic content. (See \cite[Sec.~10.1]{elAAM} where this is pointed out.)

(v) Consider now $\tD_{0}$. Similarly one can reduce to prove $\lesssim$ for  $B_{0}D$. On $\nul(B_{0}D)=\mH^\perp$, this is trivial. On $\ran(B_{0}D)=B_{0}\mH$ we use that $B_{0}D$ is similar to $DB_{0}$ on $\ran(DB_{0})=\mH$ through the isomorphism $B_{0}: \mH \to B_{0}\mH$. Thus the square function upper estimate for $B_{0}D$ follows by similarity from the one for $DB_{0}$.\end{proof}

The square function estimates from Theorem~\ref{thm:QE} 
provide bounds on the $S^o_{\nu,\sigma}$-holomorphic functional calculus
of the operators $D_0$ and $\tD_0$, adapting the techniques described in \cite{ADMc}.
Write
\begin{align*}
  H(S^o_{\nu,\sigma}) &:= \{ \text{holomorphic } b: S^o_{\nu,\sigma}\to \C\}, \\
  H_\infty(S^o_{\nu,\sigma}) &:= \sett{b\in H(S^o_{\nu,\sigma})}{\sup\sett{ |b(\lambda)|}{\lambda\in S^o_{\nu,\sigma}}<\infty  }, \\
    \Psi(S^o_{\nu,\sigma}) &:= \sett{b\in H(S^o_{\nu,\sigma}) }{|b(\lambda)| \lesssim \min(|\lambda|^a, |\lambda|^{-a}),
  \text{ for some }  a>0}.
\end{align*}

We summarize the result for the $S^o_{\nu,\sigma}$-holomorphic functional calculus in the following corollary.
The proof is a straightforward adaption of the results in \cite{ADMc}.

\begin{cor}     \label{cor:fcalc} 
  Assume $\sigma\in \R$ and $D_{0}=DB_{0}+\sigma N$. 
  Fix $\omega<\nu<\pi/2$.
There is a unique continuous Banach algebra homomorphism 
$$
  H_\infty(S^o_{\nu,\sigma}) \to \mL(\ran(D_{0})): b\mapsto b(D_0),
$$
with bounds $\| b(D_0)f \|_2\le C(\sup_{S^o_{\nu,\sigma}}|b(\lambda)| )\|f\|_2$ for all $f\in \ran(D_{0})$,
where $C$ only depends on $\|B_0\|_\infty$, $\kappa_{B_0}$, $n$ and $\sigma$, and with the following two properties.
If $b\in \Psi(S^o_{\nu,\sigma})$ then 
$$
 b(D_0)= \frac 1{2\pi i} \int_\gamma b(\lambda)(\lambda- D_0)^{-1} d\lambda \in \mL(\ran(D_{0})),
$$
where $\gamma:= \partial S_{\theta,\sigma}$, $\omega<\theta<\nu$, oriented counter clockwise around $S_{\omega,\sigma}$.
For any $b\in H_\infty(S^o_{\nu,\sigma})$ we have strong convergence
$$
  \lim_{k\to\infty}  \|b_k(D_0)f-b(D_0)f\|_2=0,\qquad\text{for each } f\in \ran(D_{0}), 
$$
whenever $b_k\in \Psi(S^o_{\nu,\sigma})$, $k=1,2,\ldots$, are uniformly bounded, \textit{i.e.} $\sup_{k,\lambda}|b_k(\lambda)|<\infty$, and converges pointwise to $b$.

The corresponding results hold for $\tD_0=B_{0}D-\sigma N$  replacing $D_0$ by  throughout.
\end{cor}

We remark that the square function estimates in Theorem~\ref{thm:QE} hold when $\psi(z)= z(1+z^2)^{-1}$ is replaced
by any $\psi\in \Psi(S^o_{\nu})$ which is non-zero on both components of $S^o_{\nu,\sigma}$.
We have
\begin{equation}    \label{eq:genQE}
  \int_0^\infty \| \psi(t D_0) f \|_2^2 \frac{dt}t\approx \|f\|_2^2,\qquad \text{for all } 
   f\in \ran(D_{0}).
\end{equation}
A similar extension of the square function estimates holds for $\tD_0$.

Fundamental operators in this paper are the following.

\begin{defn}    \label{defn:hardyprojs}
(i)
Let $\chi^+(\lambda)$ and $\chi^-(\lambda)$ be the characteristic functions
for the right and left half planes. Define spectral projections
$E_0^\pm:=\chi^\pm(D_0)$ and $\tE_0^\pm:=\chi^\pm(\tD_0)$
on $\ran(D_0)$ and $\ran(\tD_0)$ respectively.

(ii)
Define closed and dense defined operators $\Lambda= |D_0|:= \sgn(D_0)D_0$ 
and $\tilde\Lambda= |\tD_0|:= \sgn(\tD_0)\tD_0$ on $L_2(S^n;\V)$.
Here $|\lambda|:= \lambda\sgn(\lambda)$ and $\sgn(\lambda):= \chi^+(\lambda)-\chi^-(\lambda)$.

Define operators $e^{-t\Lambda}$ and $e^{-t\tilde\Lambda}$
on $\ran(D_0)$ and $\ran(\tD_0)$ respectively by applying Corollary~\ref{cor:fcalc} 
with $b(\lambda)=e^{-t|\lambda|}$, $t>0$.
\end{defn}

When $\sigma= 0$,  $ \ran(\tD_{0})=B_{0}\mH=B_{0}\ran(D_0)$ are strict subspaces of $L_{2}$ and it is
convenient to extend the above operators  to all $L_2$. 
 Using the Hodge splitting 
$L_2= B_0\mH\oplus \mH^{\perp}$,   on $\mH^\perp$ the operator $\tD_{0}=B_{0}D$ is already 0 and $\st\Lambda= |B_{0}D|$ is naturally  defined by  0. Using the other Hodge splitting $L_2=\mH\oplus B_0^{-1}\mH^{\perp}$, on $B_0^{-1}\mH^{\perp}$ the operator $D_{0}=DB_{0}$ is already 0 and  $\Lambda=|DB_{0}|$ is naturally defined by $0$.  It follows that 
$e^{-t\st \Lambda}$ and $e^{-t\Lambda}$ are naturally extended to  $L_2$, by  letting  $e^{-t\tilde\Lambda}|_{\mH^\perp}:=I$  and $e^{-t\Lambda}|_{B_0^{-1}\mH^\perp}:=I$.

However, for the projections the extension is more subtle. 
 Indeed,
we see  for the functional calculus of $\tD_0= DB_{0}-\sigma N$ that 
$$
   b(\tD_0)=b(-\sigma N)= \begin{bmatrix} b(\sigma)I & 0 \\ 0 & b(-\sigma)I \end{bmatrix}
$$
on $\mH^\perp$ when $\sigma\ne 0$ using the definition of $N$.
As we are mainly interested in  $\sigma=\frac{n-1}{2}$, it is more natural  for consistency of notation towards applications to divergence form equations to define the operators for $\sigma=0$ by continuity
$\sigma\to 0^+$.
Thus set
$$
  b(B_{0}D):= b(B_0D|_{B_0\mH}) \tP^1_{B_0} + 
  \begin{bmatrix} b(0+)I & 0 \\ 0 & b(0-)I \end{bmatrix} \tP^0_{B_0},
$$
where $b(0\pm):= \lim_{\pm\lambda\in S_{\omega+}, \lambda\to 0} b(\lambda)$,
assuming the limits exist, 
$b(B_0D|_{B_0\mH})$ is the operator from Corollary~\ref{cor:fcalc} and 
$\tP^i_{B_0}$, $i=0,1$,
denote the projections from Proposition~\ref{prop:BDprops} onto the subspaces 
in the Hodge splitting
$L_2= B_0\mH\oplus \mH^{\perp}$.

Similarly, for $\sigma\ne 0$, we have $D_0= DB_0+\sigma N$ so $D_{0}=\sigma N$ on $B_{0}^{-1}\mH^\perp$. For $\sigma=0$, set
$$
  b(DB_{0}):=    b(DB_0|_{\mH}) P^1_{B_0}    + 
  P^0_{B_0} \begin{bmatrix} b(0-)I & 0 \\ 0 & b(0+)I \end{bmatrix},
$$
where $P^i_{B_0}$, $i=0,1$,
denote the projections from Proposition~\ref{prop:DBprops} onto the subspaces 
in the Hodge splitting
$L_2=\mH\oplus B_0^{-1}\mH^{\perp}$.  Remark that $ P^0_{B_0}$ on the left of  the matrix is needed to obtain an element in $B_0^{-1}\mH^{\perp}$.  An elementary calculation shows that this extension of the functional calculus coincides with $(\conj b(B_0^*D))^*$ with $\conj b(\lambda)=\overline{b(\bar \lambda)} 
$ and that the extended functional calculi of $D_{0}$ and $\tD_{0}$ thus obtained are intertwined by $D$.  

Taking $b(\lambda)=\lambda$ or $\lambda\sgn (\lambda)$, this  provides us with the zero extension that we already chose so this is consistent. 
For the projections, this leads to the following definition.

\begin{defn}    \label{defn:hardyprojs1D} 
When $\sigma=0$,  extend $\tE_0^\pm$, $E_{0}^\pm$ originally defined on $\ran(B_{0}D)=B_{0}\mH$ and $\ran(DB_{0})=\mH$ respectively from Definition~\ref{defn:hardyprojs}  to operators on all $L_2(S^n; \V)$, letting
$$
\begin{cases}
  \tE_0^\pm f:= N^\mp f, & \qquad \text{for all } f\in \mH^\perp,
  \\
  E_0^\pm f:= P^0_{B_{0}}N^\pm f, & \qquad  \text{for all } f\in B_0^{-1}\mH^\perp.
  \end{cases}
$$
\end{defn}

\begin{lem}   \label{lem:spectralsplits}  
  With $L_{2}=L_{2}(S^n;\V)$,  the spectral projections $E_{0}^\pm$ and $\tE_{0}^\pm$ are bounded, 
 we have topological spectral splittings 
$$
L_2= E_0^+L_2 \oplus E_0^- L_2,
$$
  restricting to
  $\mH= E_0^+\mH \oplus E_0^- \mH$ in the subspace $\mH$ invariant under $D_0$,
  and 
$$
L_2= \tE_0^+L_2 \oplus \tE_0^- L_2, 
$$
  restricting to
  $\mH^\perp= \tE_0^+\mH^\perp \oplus \tE_0^- \mH^\perp$ 
  in the subspace $\mH^\perp$ invariant under $\tD_0$.
  We also have the intertwining relation
  \begin{equation}
  E_{0}^\pm D= D\tE_{0}^\pm
  \end{equation} 
   so that $D: \tE_0^\pm L_2 \to E_0^\pm\mH$ is surjective.

If $\sigma\ge 0$,
  then in the latter splitting we have $\tE_0^\pm= N^\mp$ in $\mH^\perp$.
  Hence $\tE_0^+\mH^\perp= N^-\mH^\perp$ 
  and $\tE_0^-\mH^\perp= N^+\mH^\perp$.
  (On the other hand, if $\sigma<0$, then $\tE_0^\pm= N^\pm$ in $\mH^\perp$.)  
\end{lem}

\begin{proof} 
When $\sigma\ne 0$, $\ran(\tD_{0})=L_{2}$ and $L_{2}=\ran(D_{0})$ by Proposition~\ref{prop:spectrum}. Boundedness on $L_{2}$ follows from Corollary~\ref{cor:fcalc}. The intertwining property is a consequence of Lemma~\ref{lem:intertwduality}. The surjectivity  of $D$ easily follows from the spectral subspaces and using $D:L_{2}\to \mH$ surjective and the splittings.  
  That $\tE_0^\pm= N^\mp$ in $\mH^\perp$ when $\sigma > 0$
  comes from  $\tD_0=-\sigma N$ in $\mH^\perp$ 
  and $\chi^\pm(-\sigma N)= N^\mp$.  The case $\sigma=0$ follows from 
  Definition~\ref{defn:hardyprojs1D}. We leave further details  to the reader.
  \end{proof}

%
%
%
%
%
\section{A detour to Kato's square root on Lipschitz surfaces}      \label{sec:Kato}

Let $\Sigma$ be a surface in $\R^{1+n}$, assumed to be Lipschitz diffeomorphic to
$S^n$ through a bilipschitz map $\rho_0: S^n\to \Sigma$.
Let $d\sigma$ denote surface measure on $\Sigma$.
Consider, for $n,m\ge 1$, coefficient matrices 
$H\in L_\infty(\Sigma; \mL((T_\C \Sigma)^m))$ (with $T_\C\Sigma$ denoting the complexified tangent 
bundle) and $h\in L_\infty(\Sigma;\mL(\C^m))$, assumed to be strictly accretive in the sense that
\begin{gather*}
  \re\int_\Sigma (H(x)\nabla_\Sigma u(x), \nabla_\Sigma u(x)) d\sigma(x) \ge 
  \kappa \int_\Sigma |\nabla_\Sigma u(x)|^2 d\sigma(x), \\
  \re (h(x)z,z)\ge \kappa |z|^2,\qquad\text{a.e. } x\in\Sigma, 
\end{gather*}
for all $u\in W_2^1(\Sigma;\C^m)$ and $z\in\C^m$, and some $\kappa>0$.
Then $L:= -\divv_\Sigma H \nabla_\Sigma$, with $\divv_\Sigma:= -(\nabla_\Sigma)^*$
in $L_2(\Sigma; d\sigma)$, constructed by the method of sesquilinear forms,
is a maximal accretive operator and $hL$ is defined on $\dom(L)$ and can be shown to 
be an $\omega$-sectorial operator on $L_2(\Sigma; d\sigma)$ for some $0<\omega<\pi$.
Thus it has a square root and we have 

\begin{thm}   \label{thm:Kato} 
   The square root of the operator $hL= -h\divv H\nabla_\Sigma$ has domain 
   $\dom(\sqrt{hL})= W_2^1(\Sigma;\C^m)$,
   and estimates $\|\sqrt{hL}u\|_2\approx \|\nabla_\Sigma u\|_2$.
\end{thm}

In particular for $h=1$, we obtain a version of the Kato square root problem on 
Lipschitz surfaces $\Sigma$. The presence of $h$ makes the theorem 
invariant under bilipschitz changes of variables as we shall see in the proof.

Our Theorems~\ref{thm:QE} and \ref{thm:Kato} are inspired by  
\cite[Thm.~7.1]{AKMc}, and a comparison of these two results is in order. 
The main novelty in Theorems~\ref{thm:QE} and \ref{thm:Kato}, is that 
these do not require the coefficients $B_0$ or $H$ to be pointwise strictly accretive,
which was needed for the localization argument in \cite[Thm.~7.1]{AKMc}.
This theorem considered more general forms on $\Sigma$, and more general
compact Lipschitz surfaces $\Sigma$. 
It is straightforward to extend our results Theorems~\ref{thm:QE} and \ref{thm:Kato} 
here to more general compact Lipschitz manifolds. 
On the other hand, we do not know how to extend our localization argument here
to the case of forms, unless pointwise strict accretivity is assumed.

We also mention that in his PhD thesis \cite{morris}, A. Morris 
proved similar results on embedded (possibly non-compact) Riemannian manifolds with bounds on the 
second fundamental form and a lower bound on Ricci curvature.

\begin{proof}[Proof of Theorem~\ref{thm:Kato}]
  A calculation shows the pullback formula
$$
  (h \divv_\Sigma H \nabla_\Sigma u)(\rho_0(x))=
  (\tilde h \divv_S \widetilde H \nabla_S (u\circ\rho_0))(x), \qquad x\in S^n,
$$
where $\tilde h(x)= |J(\rho_0)(x)|^{-1} h(\rho_0(x))$ and
$\widetilde H(x):= |J(\rho_0)(x)| (\underline{\rho_0}(x))^{-1} H(\rho_0(x)) (\underline{\rho_0}^t(x) )^{-1}$. 
So we assume that $\Sigma=S^n$ from now on.
Let $D$ be as in Definition~\ref{defn:DNops} and let
$B_0:= \begin{bmatrix} h & 0 \\ 0 & H \end{bmatrix}\in L_\infty(S^n;\mL(\V))$.
Then $B_0$ is strictly accretive on the space $\mH_1$ from (\ref{eq:H1space})
and 
$$
  B_0 D = \begin{bmatrix} 0 & -h\divv_S \\ H\nabla_S & 0 \end{bmatrix}.
$$
Thus by Theorem~\ref{thm:QE}, with $\sigma=0$, we have bounded 
functional calculus of $B_0 D$ in $B_0\mH$.
Following \cite{AMcN}, we have for $u\in \dom(\nabla_S)$ that
$$
  \begin{bmatrix} \sqrt{hL}u \\ 0 \end{bmatrix}
  = \sqrt{(B_0 D)^2}   \begin{bmatrix} u \\ 0 \end{bmatrix}
 = \sgn(B_0 D) B_0 D   \begin{bmatrix} u \\ 0 \end{bmatrix}
= \sgn(B_0 D)   \begin{bmatrix} 0 \\ H\nabla_S u \end{bmatrix},
$$
so that $\|\sqrt{hL}u\|_2 \approx \|H\nabla_S u\|_2\approx \|\nabla_S u\|_2$,
using that $\sgn(B_0 D)$ is bounded and invertible on $B_0\mH$ and that
$H$ is bounded above and below on $\ran(\nabla_S)$. 
\end{proof}

\begin{rem}
It is interesting to note that we apply Theorem~\ref{thm:QE} with $\sigma=0$
no matter what the dimension is.
If $n\ge 2$, Kato's square root problem on $S^n$ is not directly linked
to the boundary operator appearing in (\ref{eq:firstorderODE}),
associated to the equation $\divv_\bx A\nabla_\bx u=0$ on $\bO^{1+n}$,
with $\hat A=\begin{bmatrix} h & 0 \\ 0 & H \end{bmatrix}$, \textit{i.e.} when one can separate in the equation radial derivatives from tangential derivatives.
This is different from the case of the half space ($\R^n$ replacing $S^n$)
and emphasizes the role of curvature.

In view of Section~\ref{sec:infgene}, the second order operator on the boundary
associated to this $\divv_\bx A \nabla_\bx$ on $\bO^{1+n}$, comes from
$$
  (B_0 D-\sigma N)^2 =  
  \begin{bmatrix} -hL+ \sigma^2 & 0 \\ 0 & -H\nabla_Sh\divv_S+\sigma^2 \end{bmatrix},
$$
with $\sigma=(n-1)/2$.
Thus, the naturally associated Kato square root is $\sqrt{-hL+\sigma^2}$,
and one has 
$$
\Big\| \sqrt{-hL+(\tfrac{n-1}2)^2}\ u  \Big\|_2\approx \|\nabla_S u\|_2 + \tfrac{n-1}2 \|u\|_2.
$$
\end{rem}

%
%
%
%
%
\section{Natural function spaces}      \label{sec:XYspaces}

By Corollary~\ref{cor:conormal},
our method to study and construct solutions $u$  to the divergence form equation (\ref{eq:divform}) 
consists in translating this equation  to the ODE (\ref{eq:firstorderODE}) 
for the conormal gradient $f$  in $\R_+\times S^n$.
Conormal gradients of variational solutions  belong to $L_2(\R_+\times S^n;\V)$ as noted in 
(\ref{eq:dfl2iso}). 
On the other hand, the appropriate function spaces for  $f$ with 
Dirichlet/Neumann boundary data for $u$  in $L_2(S^n;\C^m)$ are the following.

\begin{defn}    \label{defn:NTandC}
The (truncated) {\em modified non-tangential maximal function} of $f$ defined on $\R_+\times S^n$, is
$$
 \tN(f)(x):= \sup_{0<t<c_0}  t^{-(1+n)/2} \|f\chi_{s<1}\|_{L_2(W(t,x))}, \qquad x\in S^n,
$$
where $W(t,x):= \sett{(s,y)\in \R_+\times S^n}{|y-x|<c_1 t, c_0^{-1}<s/t<c_0}$
for some fixed constants $c_0>1$, $c_1>0$.
We assume that $c_0\approx 1$ and $c_1 <<1$, so that the {\em Whitney 
regions} $W(t,x)$ are non-degenerate for $t<c_0$.
For a function $f_0$ on $\bO^{1+n}$, we have  $\tN^o(f_0)= \tN(f)$ where $f(t,x):= f_0(e^{-t}x)$,
which properly defines $\tN^o$  in the introduction.

The (truncated) {\em modified Carleson norm} of $f$ in $\R_+\times S^n$ is
$$
  \|f\|_C:= \left( \sup_{r(Q)<r_0} \frac 1{|Q|}\iint_{(0,r(Q))\times Q}
  \essup_{W(t,x)}|f|^2 \frac{dtdx}t \right)^{1/2},
$$
and the $\sup$ is taken over geodesic balls $Q\subset S^n$ with volume $|Q|$, and 
with radius $r(Q)$ less than some fixed constant $r_0<<1$.
For a function $f_0$ on $\bO^{1+n}$, we have  $\|f_0\|_C=\ \|f\|_C$ where $f(t,x):= f_0(e^{-t}x)$,
which corresponds to $\|f_{0}\|_{C}$ as in \eqref{eq:defCarleson}.
\end{defn}

Note  that changing
the parameters $c_1, c_1$ does not affect the results. 

\begin{defn}    \label{defn:XY}
(i)  For $g:\bO^{1+n}\to \C^{(1+n)m}$, define norms
\begin{align*}
  \|g\|_{\mY^o}^2 &:= \int_{\bO^{1+n}} | g(\bx) |^2 (1-|\bx|) d\bx, \\
  \|g\|_{\mX^o}^2 & :=  \|\tN^o(g)\|^2_2 +\int_{|\bx|<e^{-1}} | g(\bx) |^2 d\bx.
\end{align*} 
Let $\mY^o$ and $\mX^o$ be the Hilbert/Banach spaces of  
functions $g$ for which the respective norm is finite.

(ii) For $f:\R_+\times S^n \to \V$, define norms
\begin{align*}
  \|f\|_{\mY}^2 &:= \int_0^\infty \| f_t \|^2_2 \min(t,1)dt, \\
  \|f\|_{\mX}^2 & :=  \|\tN(f)\|^2_2 +\int_1^\infty \| f_t \|^2_2 dt. 
\end{align*} 
Let $\mY$ and $\mX$ be the Hilbert/Banach spaces of sections 
$f$ for which the respective norm is finite.
\end{defn}

The gradient-to-conormal gradient map of Proposition~\ref{prop:divformasODE} is an isomorphism $\mY^o\to \mY$ and $\mX^o \to \mX$.  

\begin{lem}      \label{lem:XlocL2}
  There are estimates
$$
  \sup_{0<t<1/2} \frac 1t\int_t^{2t} \| f_s \|_2^2 ds \lesssim \| \tN(f) \|^2_2 \lesssim \int_0^1 \| f_s \|^2_2 \frac {ds}s,
  \qquad f\in L_2^\loc(\R_+\times S^n;\V).
$$
Denoting by $\mY^*$ the dual space of $\mY$ relative to $L_2(\R_+\times S^n; \V)$,
\textit{i.e.} the space of functions $f$ such that 
$\int_0^\infty \| f_t \|^2_2 \max(t^{-1},1)dt<\infty$,
we have continuous inclusions of Banach spaces
$$
  \mY^*\subset \mX \subset L_2(\R_+\times S^n; \V)\subset \mY.
$$
\end{lem}

Note that Lemma~\ref{lem:XlocL2} shows that another choice of threshold than $t=1$ in the definition of the norms for $\mX$ and $\mY$ would result in equivalent norms.

\begin{proof}
The $L_2^\loc (L_2)$ estimates of $\|\tN(f)\|_2$ is an adaption of the corresponding 
result for $\R^{1+n}_+$, proved in \cite[Lem. 5.3]{AA1}.
The remaining statements, except possibly that $\mX \subset L_2(\R_+\times S^n; \V)$,
are straightforward consequences.
To verify this embedding of $\mX$, we use the lower bound on $\|\tN(f)\|_2$ to estimate
\begin{multline*}
  \int_0^\infty \|f_t\|_2^2 dt = \sum_{k=0}^\infty  \int_{2^{-k-1}}^{2^{-k}} \|f_t\|_2^2 dt +  \int_1^\infty \|f_t\|_2^2 dt \\
  \lesssim \sum_{k=0}^\infty 2^{-k-1} \|\tN(f)\|_2^2 +  \int_1^\infty \|f_t\|_2^2 dt 
  = \|f\|_{\mX}^2.
\end{multline*}
\end{proof}

The following lemma gives necessary and (different) 
sufficient conditions for a multiplication operator $\E$
to map $\mX$ into $\mY^*$. Write
$$
 \|\E\|_{C\cap L_\infty}:= \|\E\|_C+  \|\E\|_{L_\infty(\R_+\times S^n)}.
$$

\begin{lem}
  For functions $\E:\R_+\times S^n\to \C^{(1+n)m}$, define the multiplicator norm
$\|\E\|_*:= \|\E\|_{\mX\to \mY^*}=\sup_{\|f\|_\mX=1}\|\E f\|_{\mY^*}$. 
Then we have estimates
$$
  \|\E\|_{L_\infty(\R_+\times S^n)}\lesssim \|\E\|_*\lesssim  \|\E\|_{C\cap L_\infty}.
  $$
\end{lem}

\begin{proof}
  This is an adaption to the unit ball of \cite[Lem.~5.5]{AA1}. 
  As in that proof, the estimate $\|\E\|_\infty\lesssim \|\E\|_*$ 
  follows from the $L_2^\loc$ estimates in Lemma~\ref{lem:XlocL2}.
  For the second estimate we write
$$
  \|\E f\|_{\mY^*}^2 = \int_0^{a} \|\E_t f_t\|_2^2 \frac {dt}t+ \int_{a}^\infty \|\E_t f_t\|_2^2 dt.
$$
As in \cite[Lem. 5.5]{AA1}, the first term is estimated with Whitney averaging and Carleson's theorem.
The second term is controlled with $\|\E\|_\infty$.
In total, this gives the bound $\|\E f\|_{\mY^*}\lesssim \|\E\|_C \|f\|_\mX+ \|\E\|_\infty\|f\|_\mX$
as desired.
\end{proof}

\begin{rem} It has been recently proved in \cite{AH} that $ \|\E\|_*\gtrsim  \|\E\|_{C\cap L_\infty}$ so all of our results use in fact the same condition on $\E$.
\end{rem}

We end this section by introducing an auxiliary subspace $\mY_\delta$ of $\mY$.

\begin{defn}    \label{defn:Ydelta}
For $\delta>0$, define the norm
$$
  \|f\|_{\mY_\delta}^2 := \int_0^\infty \| f_t \|^2_2 \min(t,1) e^{\delta t}dt.
$$
Let $\mY_\delta$ be the Hilbert spaces of sections 
$f:\R_+\times S^n\to \V$ such that $\|f\|_{\mY_\delta}$ is finite.
\end{defn}

Clearly $\mY_\delta\subset \mY$. The motivation for introducing $\mY_\delta$ is the following
result.

\begin{prop}     \label{prop:reverseholder} 
  Given coefficients $A\in L_\infty(\bO^{1+n}; \mL(\C^{(1+n)m}))$, which are
  strictly accretive on $\mH_1$, there is $\delta>0$ such that 
$$
  \int_{1}^\infty \|f_{t}\|_{2}^2 e^{\delta t} \, dt \lesssim \int_{1/2}^\infty \|f_t\|_2^2 \, dt,
$$
for all $f\in L_2^\loc(\R_+;\mH)$ solving $\pd_t f+ (DB+\tfrac{n-1}2 N) f=0$.
Hence, if $f\in\mY\cap L_2^\loc(\R_+;\mH)$ and $\pd_t f+ (DB+\tfrac{n-1}2 N) f=0$, then $f\in \mY_\delta$
and $\|f\|_{\mY_\delta}\lesssim \|f\|_\mY$.
\end{prop}

The proof of Proposition~\ref{prop:reverseholder} uses {\em reverse H\"older  inequalities}.

\begin{thm}   \label{thm:revholder} Fix $c>1$.
There exist $C<\infty$ and $p>2$ depending only on $n,m$, the ellipticity constants  $\|A\|_{\infty}, \kappa_{A}$ of $A$ and $c$, 
such that  for any ball $B$ with $c\overline{B}\subset \bO^{1+n}$ and any weak solution to  $\divv_\bx (A\nabla_{\bx}u)=0$ in $\bO^{1+n}$, we have 
$$
\left(\int_{B} |\nabla_{\bx}u|^p\, d\bx\right)^{1/p} \le C \left(\int_{cB} |\nabla_{\bx}u|^2\, d\bx\right)^{1/2}.
$$
\end{thm}

\begin{proof} 
This result is due to N. Meyers  \cite{Me} for equations. Here, we make sure that the result extends to elliptic systems in the sense of G\aa rding by giving appropriate references.  We begin by noting that the usual Caccioppoli inequality for weak solutions
$$\left(\int_{B} |\nabla_{\bx}u|^2\, d\bx\right)^{1/2} \le Cr \left(\int_{cB} |u|^2\, d\bx\right)^{1/2}$$
for any ball $B$ so that  $c\overline{B}\subset \bO^{1+n}$, with $r$ its radius, holds for any system that is elliptic in the sense  of the G\aa rding inequality \eqref{eq:accrasgarding}.
Although not stated like this in  \cite[Thm.~1.5, p.46]{C},  the proof  only uses G\aa rding's inequality. See also \cite{AQ} where it is done  explicitly  for second (and higher order)  equations  and it is said \cite[p.315]{AQ} that this applies \textit{in extenso} to such systems. 
 The constant $C$ depends only on $n$, $m$, $\kappa$, $\|A\|_{\infty}$ and $c$. Now, this combined with Poincar\'e inequality yields
$$
\left(\int_{B} |\nabla_{\bx}u|^2\, d\bx\right)^{1/2} \le \left(\int_{cB} |\nabla_{\bx}u|^q\, d\bx\right)^{1/q}
$$
for $\frac{2(n+1)}{n+3}<q<2$. Finally, Gerhing's method for improvement of reverse H\"older inequalities with increase of radii, presented  in \cite[Thm.~6.3]{Gia2}, applies.
 \end{proof}

\begin{proof}[Proof of Proposition~\ref{prop:reverseholder}]
Corollary~\ref{cor:conormal} shows  that  $f$  is the conormal gradient   of    a weak solution  to 
$\divv_\bx A\nabla_\bx u=0$ in $\bO^{1+n}$.
By H\"older's inequality and Theorem~\ref{thm:revholder}, we have
for $g= \nabla_\bx u$ the estimate
$$
\left(\int_{|\bx|<e^{-1}} |g(\bx)|^2 |\bx|^{-\delta} \, d\bx\right)^{1/2} 
\lesssim \left(\int_{|\bx|<e^{-1}} |g(\bx)|^p\, d\bx\right)^{1/p}
\lesssim  \left(\int_{|\bx|<e^{-1/2}} |g(\bx)|^2\, d\bx\right)^{1/2}
$$ 
for $0<\delta <\frac{(n+1)(p-2)}{p}$.
This translates to the stated estimate for $f$, using the gradient-to-conormal gradient map 
from Definition~\ref{defn:gradtoconormal}.
\end{proof}

%
%
%
%
%
\section{Semi-groups and radially independent coefficients}\label{sec:semigroups}

\noindent\textbf{In this section and the subsequent ones, we set $\sigma=\frac{n-1}{2}$.}

In this section, fix  radially independent coefficients $A_{1}$ and $B_{0}=\widehat {A_{1}}$. We show how to obtain weak solutions of $\divv_{\bx}A_{1}\nabla_{\bx}u=0$ inside and outside $\bO^{1+n}$ using the semi-groups associated to $\Lambda$ and $\tilde\Lambda$. 
Later, we show all weak solutions with prescribed growth towards the boundary have a representation in terms of these semi-groups. 

\begin{thm}     \label{thm:NT}
  Let $f_0$ belong to the spectral subspace $E_0^+ \mH$. Then 
$$
f_t:= e^{-t\Lambda}f_0
$$ 
  gives an $\mH$-valued solution
  to $\pd_t f+ D_0 f=0$, in the strong sense $f\in C^1(\R_+; L_2)\cap C^0(\R_+; \dom(D_0))$ and in $\R^+\times S^n$ distribution sense.
  (In particular $f$ is the conormal gradient of a weak solution of 
 $\divv_{\bx}A_{1}\nabla_{\bx}u=0$ in $\bO^{1+n}$.)
  The function $f$ has $L_2$ limit $\lim_{t\to 0} f_t =f_0$ and rapid decay 
  $\| \pd_t^j f_t \|_2 \le C_{j,k}/t^k \|f_0\|_2$, for each $k\ge j\ge 0$.
  Moreover, we have estimates
$$
  \|\pd_t f\|_\mY\approx \|f_0\|_2 \approx \|f\|_\mX.
$$
  
  If instead $f_0$ belongs to the spectral subspace $E_0^- \mH$, then 
  define $f_t:= e^{t\Lambda}f_0$ for $t<0$.
  Then $\pd_t f+ D_0 f=0$ for $t<0$.   
  (In particular $f$ is the conormal gradient of a weak solution of 
  $\divv_{\bx}A_{1}\nabla_{\bx}u=0$ in $\R^n\setminus\clos{\bO^{1+n}}$.)
  Limits and estimates as above hold for $f_t$, $t<0$.
\end{thm}

\begin{proof}
  (i)
  The rapid decay of $f_t$ follows from the lower bound on $D_0|_\mH$ from Proposition~\ref{prop:DBprops},
  giving
$$
  \|\pd_t^j f_t\|_2 = \|\Lambda^j e^{-t\Lambda} f_0\|_2
  \lesssim \|(D_0)^{k-j} \Lambda^j e^{-t\Lambda}f_0\|_2\approx t^{-k} \|(t\Lambda)^k e^{-t\Lambda}f_0\|_2
  \lesssim t^{-k}\|f_0\|_2.
$$

(ii) That $f$ is the conormal gradient of a solution follows from Corollary~\ref{cor:conormal} and
it is straightforward to show that the ODE $\pd_t f+ D_0 f=0$ is satisfied in the strong and distribution sense.

(iii) Next,  $\|\pd_t f\|_\mY^2 \le  \int_0^\infty \|\pd_t f_t\|_2^2 tdt$, and 
the square function estimate $\int_0^\infty \|\pd_t f_t\|_2^2 tdt\approx \|f_0\|_2^2$ follows from 
(\ref{eq:genQE}), since $\pd_t f_t= -\Lambda e^{-t\Lambda} f_0$.
This together with the decay from (i) with $j=1$ shows $\|f_0\|_2\approx \|\pd_t f\|_\mY$.

(iv) It remains to show that $\|f_{0}\|_{2}\approx \|f\|_\mX$. For this, the decay from (i) with $j=0$
implies it is enough to prove 
$ \|\tN f\|_{2} \approx \|f_0\|_2 .$ The proof is an adaptation of the 
results on $\R^{1+n}_+$ from \cite[Prop. 2.56]{AAH} as follows.  

 The estimate $\|\tN(f)\|_2\gtrsim \|f_0\|_2$ follows from Lemma~\ref{lem:XlocL2}. Next consider the estimate $\lesssim$. 
 We follow the argument in \cite[Prop.~2.56]{AAH}. By the reverse H\"older inequalities noted in the proof of Proposition \ref{prop:reverseholder} applied to a weak solution of the divergence form equation with coefficients $A_{1}$ associated with $f=e^{-t|D_0|}f_0$, we can bound  $L_{2}$ averages   by $L_{p}$ averages for some $p<2$, \textit{i.e.}  $\tN f\lesssim \tN^p f$ in a pointwise sense (up to changing to constants $c_{0},c_{1}$).  Since $\psi(\lambda)=e^{-|\lambda|}-(1+i\lambda)^{-1}\in \Psi(S_{\nu,\sigma}^o)$, 
it follows from Lemma~\ref{lem:XlocL2} and Theorem~\ref{thm:QE}, 
or more precisely (\ref{eq:genQE}), that
$$
 \|\tN^p(\psi(tD_{0})f_0)\|_2 \lesssim   \|\tN(\psi(tD_{0})f_0)\|_2\lesssim \|f_0\|_2.
$$
For   
$h_t:= (I+itD_0)^{-1} f_0$ we have $\|\tN^p(h)\|_2 \lesssim \|M(|f_{0}|^p)^{1/p}\|_{2} \lesssim \|f_{0}\|_{2}$ by Corollary~\ref{cor:resolventNT} and the boundedness of $M$ on $L_{2/p}$. 
We have proved that $\|\tN f\|_{2}\lesssim \|f_{0}\|_{2}$. 

(v) The modifications for   $f_{0} \in E_{0}^-\mH$ are straightforward, and the correspondence with $u$ follows from applying  the methods of Proposition \ref{prop:divformasODE}.
\end{proof}

\begin{rem} 
The assumption $\sigma=\frac{n-1}{2}$ is used in part (iv) to pass from $\tN$ to $\tN^p$ with some $p<2$. Thus, for any $\sigma\in \R$, $f_{0} \in \mH$ and $p<2$,
we have $\|\tN^p(f)\|_{2}\lesssim \|f_{0}\|_{2}$. The converse, however, is not clear because $p<2$, and this shows that the value of $\sigma$ is significant. 
\end{rem}

\begin{thm}     \label{thm:NTtilde}
  Let $v_0\in \tE_{0}^+ L_{2}$. Then 
$$
v_t:= e^{-t\tilde\Lambda}v_0
$$
 gives a solution
  to $\pd_{t}v+ \tD_{0}v=0$, in the strong sense $v\in C^1(\R_+; L_2)\cap C^0(\R_+; \dom(\tD_0))$ and in $\R^+\times S^n$ distributional sense.
  (In particular $r^{-\sigma}(v_{t})_{\no}$ extends to a weak solution of 
  $\divv_{\bx}A_{1}\nabla_{\bx}u=0$ in $\bO^{1+n}$ as in Proposition~\ref{prop:ODEtoPotential}.)
  The function $v$ has $L_2$ limit $\lim_{t\to 0} v_t =v_0$ and rapid decay $\| \pd_t^j v_t \|_2 \le C_{j,k}/t^k \|v_0\|_2$ for each $k\ge j\ge 0$. (When $\sigma=0$, this estimate for $j=0$ only holds for 
  $v_0\in \ran(\tD_0)\cap \tE_{0}^+ L_{2}$.)
  Moreover, for $p<2$, we have estimates
$$
\|\pd_t v\|_\mY + \|\tN^p (v)\|_{2} +\|\tN(v_{\no})\|_{2} \lesssim \|v_0\|_2.
$$
  In dimension $n=1$, we have $\|v\|_{\mX}\approx \|v_{0}\|_{2}$. 
  
  If instead $v_0\in \tE_{0}^- L_{2}$, then 
  define $v_t:= e^{t\tilde\Lambda}v_0$ for $t<0$.
  Then $\pd_{t}v+ \tD_{0}v=0$ for $t<0$.   
  (In particular $r^{-\sigma}(v_{t})_{\no}$ satisfies 
  $\divv_{\bx}A_{1}\nabla_{\bx}u=0$ in $\R^n\setminus\clos{\bO^{1+n}}$ as in Proposition~\ref{prop:ODEtoPotential}.)
  Limits and estimates as above hold for $v_t$, $t<0$.
\end{thm}

\begin{proof} 
The proof, except for the non-tangential maximal estimates, 
is identical to that of Theorem~\ref{thm:NT}, 
using Proposition~\ref{prop:BDprops} and Corollary~\ref{cor:resolventNT}. 
When $n\ge 2$, the estimate of $\|\tN(v_{\no})\|_{2}$ follows, using the same $\psi$ 
as above and reduction to 
$\|\tN((I+it\tD_{0})^{-1}v_{0})_{\no})\|_{2}$, from Corollary~\ref{cor:resolventNT} and the maximal theorem. 
When $n=1$, one uses the splitting in Proposition~\ref{prop:BDprops}:
we have that $e^{-t\tilde\Lambda}$ is the identity on $\mH^\perp$ and that $\tilde \Lambda$ 
on $B_{0}\mH$ is similar to $\Lambda$ on $\mH$, so $\|v\|_{\mX}\approx \|v_{0}\|_{2}$ follows from Theorem~\ref{thm:NT}.

The modifications when $v_{0}\in \tE_{0}^- L_{2}$ are straightforward.
\end{proof}

%
%
%
%
%
\section{The ODE in integral form}     \label{sec:integration}

  Following \cite{AA1}, for radially dependent coefficients
  we solve   (\ref{eq:firstorderODE})  for $f$ by rewriting it as
$$
  \pd_t f + (D B_0+ \sigma N) f = D \E f,\qquad \text{where } \E_t:= B_0-B_t.$$
Recall that solutions $f_t$ belong to $\mH$, where $\mH$ splits into $E_{0}^+\mH$ and 
$E_{0}^-\mH$ by Lemma \ref{lem:spectralsplits}, with $E_{0}^\pm=\chi^\pm(D_{0})$ on $\mH$.
Applying $E_{0}^\pm$,  integrating formally each subequation and subtracting 
the obtained equations  we obtain
\begin{equation}    \label{eq:formalint}
  f_t = e^{-t\Lambda} E_0^+ f_0 +  \int_0^t e^{-(t-s)\Lambda} E_0^+ D \E_s f_s ds -
  \int_t^\infty  e^{-(s-t)\Lambda} E_0^- D \E_s f_s ds,
\end{equation}
provided $\lim_{t\to 0} f_t =f_0$ and $\lim_{t\to \infty} f_t =0$ in appropriate sense. We first study proper definition, boundedness of the integral operators in \eqref{eq:formalint}
 on appropriate spaces and their limits.   The justification of \eqref{eq:formalint} is done in Section \ref{sec:representation}.

\begin{lem}     \label{lem:preints}
  If $f\in L_2^\loc(\R_+; \mH)$ satisfies $\pd_t f+ (DB+\sigma N) f=0$ in $\R_+\times S^n$ distributional sense,
  then 
\begin{align*}
   -\int_0^t \pd_s \eta^+_\epsilon(t,s) e^{-(t-s)\Lambda} E_0^+ f_s ds &= \int_0^t \eta^+_{\epsilon}(t,s) e^{-(t-s)\Lambda} E_0^+ D \E_s f_s ds, \\
   -\int_t^\infty \pd_s \eta^-_\epsilon(t,s) e^{-(s-t)\Lambda} E_0^- f_s ds &= \int_t^\infty \eta^-_{\epsilon}(t,s) e^{-(s-t)\Lambda} E_0^- D \E_s f_s ds, 
\end{align*}
for all $t>0$. 
The bump functions $\eta_{\epsilon}^\pm$ are constructed as follows.
Let $\eta^0(t)$ to be the piecewise linear continuous function with support
$[1,\infty)$, which equals $1$ on $(2,\infty)$ and is linear on $(1,2)$.
Then let $\eta_\epsilon(t):= \eta^0(t/\epsilon)(1- \eta^0(2\epsilon t))$ and
$\eta_\epsilon^\pm(t,s):= \eta^0(\pm (t-s)/\epsilon) \eta_\epsilon(t)\eta_\epsilon(s)$.
\end{lem}

\begin{proof} Follow \cite[Prop. 4.4]{AA1}.\end{proof}

Define for $f \in L_{2}^{loc}(\R^+; L_{2}(S^n; \V))$,
$$
  S_A^\epsilon f_t:= \int_0^t \eta_\epsilon^+(t,s) e^{-(t-s)\Lambda} E_0^+ D \E_s f_s ds -
  \int_t^\infty  \eta_\epsilon^-(t,s) e^{-(s-t)\Lambda} E_0^- D \E_s f_s ds.
$$
In fact, this formula makes sense by extension  thanks to the following algebraic relations.
\begin{lem}    \label{lem:SAdecomp}
  We have
$S_A^\epsilon f_t = \hS_{A}^\epsilon f_t -\sigma \vS_{A}^\epsilon f_t = D \tilde S_A^\epsilon f_t$,
where 
\begin{align*}
  \hS_A^\epsilon f_t &:= \int_0^t \eta_\epsilon^+(t,s) \Lambda e^{-(t-s)\Lambda} \hE_0^+ \E_s f_s ds +
  \int_t^\infty \eta_\epsilon^-(t,s) \Lambda  e^{-(s-t)\Lambda} \hE_0^- \E_s f_s ds, \\
  \vS_A^\epsilon f_t &:= \int_0^t \eta_\epsilon^+(t,s) e^{-(t-s)\Lambda} \vE_0^+ \E_s f_s ds -
  \int_t^\infty  \eta_\epsilon^-(t,s) e^{-(s-t)\Lambda} \vE_0^-  \E_s f_s ds, \\
  \tS_A^\epsilon f_t &:= \int_0^t \eta_\epsilon^+(t,s) e^{-(t-s)\tilde\Lambda} \tE_0^+ \E_s f_s ds -
  \int_t^\infty  \eta_\epsilon^-(t,s) e^{-(s-t)\tilde \Lambda} \tE_0^-  \E_s f_s ds.
\end{align*}
Here  $\hE^\pm_0 := E_0^\pm B_0^{-1}\tP^1_{B_0}$, $\vE^\pm_0 := E_0^\pm N B_0^{-1}\tP^1_{B_0}$,
with $\tP^1_{B_0}$ as in Proposition~\ref{prop:BDprops}.
\end{lem}

\begin{proof}    Here, $B_0^{-1}$ denotes the inverse of the isomorphism $B_0: \mH\to B_0\mH$.
  Since $\nul(D)=\mH^\perp$, we have
$$
  E_0^\pm D = E_0^\pm D\tP^1_{B_0}= E_0^\pm \big((DB_0+ \sigma N)-\sigma N \big)B_0^{-1}\tP^1_{B_0}= 
  D_0 \hE_0^\pm - \sigma \vE_0^\pm,
$$
Using that $e^{-u\Lambda} $ and $e^{-u\Lambda} \Lambda$ extend to bounded operators on $\mH$, this also shows that 
$e^{-u\Lambda} E_0^+ D$ extend to bounded operators on $L_2$ for $u>0$. We now readily obtain $S_A^\epsilon = \hS_{A}^\epsilon -\sigma \vS_{A}^\epsilon$.
The identity $S_A^\epsilon= D \tS_A^\epsilon$ is a consequence of the intertwining relation 
$$
  b(D_0) D= D b(\tD_0)
$$
between the two functional calculi.
\end{proof}

\begin{thm}    \label{thm:XYbounded} 
  Assume  $\|\E\|_*<\infty$.
  We have bounded operators
$$
  S_A^\epsilon: \mX\to \mX,\qquad S_A^\epsilon: \mY\to \mY,
$$
with norms $\lesssim \|\E\|_*$, uniformly for $\epsilon>0$.
In the space $\mX$ there is a limit operator $S_A^\mX\in \mL(\mX; \mX)$
such that 
$$
  \lim_{\epsilon\to 0}\|S_A^\epsilon f- S_A^\mX f\|_{L_2(a,b;L_2)}=0,\qquad \text{for any } f\in \mX, 0<a<b<\infty.
$$
The same bounds and limits hold for $\hS_A^\epsilon$ and $\vS_A^\epsilon$ on $\mX$.

In the space $\mY$, there is a limit operator $S_A^\mY\in \mL(\mY; \mY)$ such that
$$
\lim_{\epsilon\to 0}\|S_A^\epsilon f- S_A^\mY f\|_{\mY}= 0,\qquad\text{for any }f\in \mY.
$$
The same bounds and limits hold for $\hS_A^\epsilon$ and $\vS_A^\epsilon$ on $\mY$.
\end{thm}

Let $S_A:= \lim_{\epsilon\to 0} S_A^\epsilon$, $\hS_A:= \lim_{\epsilon\to 0} \hS_A^\epsilon$ and $\vS_A:= \lim_{\epsilon\to 0} \vS_A^\epsilon$ denote the limit operators on $\mY$ from Theorem~\ref{thm:XYbounded}. Since $\mX$ is densely embedded in $\mY$, these limit operators restricts to the 
  corresponding limit operators on $\mX$ from Theorem~\ref{thm:XYbounded}.
  
  One sees that $S_A= \hS_A-\sigma \vS_A$ holds, and that
$$
S_A f_t= \lim_{\epsilon\to 0}\left(
\int_\epsilon^{t-\epsilon} e^{-(t-s)\Lambda} E_0^+ D \E_s f_s ds -
  \int_{t+\epsilon}^{\epsilon^{-1}}  e^{-(s-t)\Lambda} E_0^- D \E_s f_s ds \right),
$$
with convergence in $L_2(a,b;L_2)$ for any $0<a<b<\infty$,  both on $\mY$ and $\mX$.

\begin{proof}
  The proof is essentially an application of \cite[Sec.~6]{AA1}, where the results were proved abstractly.
  Given Theorems~\ref{thm:QE} and \ref{thm:NT}, these results from \cite{AA1} apply.
  In particular, this makes use of the holomorphic $S^o_{\omega, \sigma}$ operational calculus of $D_0$,
  where more general operator-valued holomorphic functions are applied to $D_0$.
  It is straightforward, given Theorem~\ref{thm:QE} , to adapt the results in \cite[Sec.~6-7]{AA1} 
  and construct this $S^o_{\omega, \sigma}$ operational calculus of $D_0$, and we omit the details.
  
(i)
  Consider the operators $\hS_A^\epsilon:\mX\to \mX$. Here \cite[Thm.~6.5]{AA1} shows that 
  $\hS_A^\epsilon: L_2(\R_+, dt;L_2)\to L_2(\R_+, dt;L_2)$ are uniformly bounded, with norm $\lesssim \|\E\|_\infty$,
  and converge strongly in $\mL(L_2(\R_+, dt;L_2))$ as $\epsilon\to 0$.
  Moreover, \cite[Thm.~6.8]{AA1} applies and shows
  that
$$
  \hS_A^\epsilon f_t = \widehat Z^\epsilon (\E f)_t +  
  \eta_\epsilon(t)e^{-t \Lambda} \int_0^\infty \eta_\epsilon(s) \Lambda e^{-s\Lambda}\hE_0^+ \E_s f_s ds
$$
where $ \widehat Z^\epsilon: L_2(\R_+, dt/t;L_2)\to L_2(\R_+, dt/t;L_2)$ are uniformly bounded and converge strongly 
as $\epsilon\to 0$. These estimates build on the square function estimates and make use of the operational calculus for $D_0$.
On the other hand, using Theorem~\ref{thm:NT} and Theorem~\ref{thm:QE}, the last term has estimates
\begin{multline*}  
\left\| \eta_\epsilon(t)e^{-t \Lambda} \int_0^\infty \eta_\epsilon(s) \Lambda e^{-s\Lambda}\hE_0^+ \E_s f_s ds \right\|_\mX\lesssim
 \left\| \int_0^\infty \eta_\epsilon(s) \Lambda e^{-s\Lambda}\hE_0^+ \E_s f_s ds \right\|_2 \\
 = \sup_{\|h\|_2=1}
  \left| \int_0^\infty  (s\Lambda^* e^{-s\Lambda^*} h, \eta_\epsilon(s)\hE_0^+ \E_s f_s) \frac{ds}s\right| \lesssim \|\eta_\epsilon \E f\|_{\mY^*}
  \lesssim \|\E\|_* \|f\|_{\mX},
\end{multline*}
and is seen to converge strongly in $\mL(\mX, L_2(a,b;L_2))$ for any $0<a<b<\infty$, as in \cite[lem. 6.9]{AA1}.
Piecing these estimates together, we obtain
\begin{multline*}
  \|\hS_A^\epsilon f\|_\mX \lesssim \| \widehat Z^\epsilon(\E f)\|_{L_2(dt/t;L_2)}+
  \| \widehat Z^\epsilon(\E f)\|_{L_2(dt;L_2)}+ \|\hS_A^\epsilon f-  \widehat Z^\epsilon(\E f)\|_\mX \\
  \lesssim  \|\E\|_* \|f\|_\mX + \|\E\|_\infty \|f\|_{L_2(dt;L_2)} + \|\E\|_* \|f\|_\mX,
\end{multline*}
with strong convergence in $\mL(\mX, L_2(a,b;L_2))$.

(ii)
For the operators $\vS_A^\epsilon:\mX\to\mX$, we note that the estimates for $\hS_A^\epsilon$ go through when replacing
$\hE_0^\pm$ by $\vE_0^\pm$. Since $\vS_A^\epsilon=\Lambda^{-1} \hS_A^\epsilon$ (with $\hE_0^\pm$ replaced
by $\vE_0^\pm$) and $\Lambda^{-1}:L_2(dt/t;\mH)\to L_2(dt/t;\mH)$ is bounded, it only remains to estimate the term
$\eta_\epsilon(t)e^{-t \Lambda} \int_0^\infty \eta_\epsilon(s) e^{-s\Lambda}\vE_0^+ \E_s f_s ds$.
But again using the boundedness of $\Lambda^{-1}$ gives
\begin{multline*}  
\left\| \eta_\epsilon(t)e^{-t \Lambda} \int_0^\infty \eta_\epsilon(s) e^{-s\Lambda}\vE_0^+ \E_s f_s ds \right\|_\mX\lesssim
 \left\| \int_0^\infty \eta_\epsilon(s) e^{-s\Lambda}\vE_0^+ \E_s f_s ds \right\|_2 \\
\lesssim \left\| \Lambda \int_0^\infty \eta_\epsilon(s) e^{-s\Lambda}\vE_0^+ \E_s f_s ds \right\|_2,
\end{multline*}
and the rest of the estimates go though as for $\hS^\epsilon_A$.
Altogether, this proves the stated bounds and convergence for $S_A^\epsilon: \mX\to\mX$.

(iii)
 Next consider the operators $\hS_A^\epsilon:\mY\to \mY$.
 We have
\begin{multline*}
  \|\hS_A^\epsilon f\|_\mY \le \|\hS_A^\epsilon (\chi_{t<1}f)\|_\mY + \|\hS_A^\epsilon(\chi_{t>1}f)\|_\mY \\
  \le \|\hS_A^\epsilon (\chi_{t<1}f)\|_{L_2(tdt; L_2)} + \|\hS_A^\epsilon(\chi_{t>1}f)\|_{L_2(dt;L_2)} \\
  \lesssim \|\E\|_*\|\chi_{t<1}f\|_{L_2(tdt; L_2)} + \|\E\|_\infty \|\chi_{t>1}f\|_{L_2(dt;L_2)} \lesssim \|\E\|_* \|f\|_\mY,
\end{multline*}
 where the $L_2(tdt; L_2)$ estimate follows from \cite[Prop.~7.1]{AA1} and the $L_2(dt; L_2)$ estimate from \cite[Prop.~6.5]{AA1}, along with convergence.
 This immediately gives the estimates for $\vS_A^\epsilon:\mY\to \mY$ since $\Lambda^{-1}:L_2(tdt;\mH)\to L_2(tdt;\mH)$ and $\Lambda^{-1}:L_2(dt;\mH)\to L_2(dt;\mH)$
are bounded.
 \end{proof}

Denote by $C(a,b; L_2)$ the space of continuous functions
$(a,b) \ni t \mapsto v_t\in L_2(S^n;\V)$.

\begin{thm}    \label{thm:Cbounded} 
Assume  $\|\E\|_*<\infty$.
If  $n\ge 2$,  then $\tS_A^\epsilon f\in C(0,\infty; L_2)$ for any $f\in\mY$.
There are bounds $\|\tS_A^\epsilon f_t\|_2 \lesssim \|\E\|_* \|f\|_{\mY}$,
uniformly for all $f\in\mY$, $t,\epsilon>0$,
and for each $f\in\mY$ there is a limit function $\tS_A f\in C(0,\infty; L_2)$
such that $\lim_{\epsilon\to 0}\|\tS_A^\epsilon f_t- \tS_A f_t\|_2= 0$ 
locally uniformly for $t>0$. We have  the expression 
\begin{equation}  \label{def:tildeSA}
     \tS_A f_t = \int_0^t e^{-(t-s)\tilde\Lambda} \tE_0^+ \E_s f_s ds -
      \int_t^\infty e^{-(s-t)\tilde\Lambda} \tE_0^- \E_s f_s ds,
 \end{equation}
  where the integrals are weakly convergent in $L_2$ for all $f\in \mY$ and $t>0$. 
  Finally, $S_Af = D\tS_A f$ holds in $\R_+\times S^n$ distributional sense for each $f\in\mY$.

If $n=1$, then the above results hold if $\mY$ is replaced by $\mY_\delta$,
for any fixed $\delta>0$.
\end{thm}

\begin{proof}
(i)
  Consider first the case  $n\ge 2$.
  The proof is a adaption of the proof of \cite[Prop.~7.2]{AA1},
  which we refer to for further details.
  We split the $(0,t)$-integral
$$
   \int_0^t \eta_\epsilon^+(t,s) e^{-(t-s)\tilde\Lambda} (I- e^{-2s\tilde\Lambda}) \tE_0^+ \E_s f_s ds 
   +   e^{-t\tilde\Lambda}\int_0^{t} \eta_\epsilon^+(t,s) e^{-s\tilde\Lambda} \tE_0^+ \E_s f_s ds,
$$
The same duality estimate of the second term as in \cite[Prop.~7.2]{AA1}, 
given Theorem~\ref{thm:NT} and Lemma~\ref{lem:intertwduality}, goes through here.
For the first term, we note the estimates
$\|e^{-(t-s)\tilde\Lambda} (I- e^{-2s\tilde\Lambda})\|\lesssim \min(s/t, 1, 1/(t-s))$.
For $t\le 2$, this yields the bound
$\|\E\|_\infty \int_0^t (s/t)\|f_s\|_2 ds\lesssim \|\E\|_\infty \|f\|_\mY$.
On the other hand, for $t\ge 2$ we have the estimate
$$
  \|\E\|_\infty \left( \int_0^1 \frac st \|f_s\|_2 ds + \int_1^{t-1} \frac1{t-s}\|f_s\|_2 ds
  + \int_{t-1}^t \|f_s\|_2 ds \right) \lesssim \|\E\|_\infty \|f\|_\mY.
$$
The $(t,\infty)$-integral is estimated similarly, by splitting it
$$
   \int_t^\infty \eta_\epsilon^-(t,s) e^{-(s-t)\tilde\Lambda} (I- e^{-2t\tilde\Lambda}) \tE_0^- \E_s f_s ds 
   +   e^{-t\tilde\Lambda}\int_t^\infty \eta_\epsilon^-(t,s) e^{-s\tilde\Lambda} \tE_0^+ \E_s f_s ds,
$$
The second term is estimated as before, and for the first term we note the estimates
$\|e^{-(s-t)\tilde\Lambda} (I- e^{-2t\tilde\Lambda})\|\lesssim \min(t/s, 1, 1/(s-t))$,
which give the bound
$$
  \|\E\|_\infty \left( \int_t^{t+1} \frac ts \|f_s\|_2 ds + \int_{t+1}^\infty \frac1{s-t}\|f_s\|_2 ds
   \right) \lesssim \|\E\|_\infty \|f\|_\mY.
$$

(ii)
Consider next the case  $n=1$. Since $e^{-t\tilde \Lambda}=I$ on $\mH^\perp$
and $\tE_0^\pm= N^\mp$ on $\mH^\perp$,
we also need to estimate the $L_2$-norm of
$$
  \left( \int_0^t \eta_\epsilon^+(t,s) \tP^0_{B_0}\E_s f_s \right)_\no - 
  \left( \int_t^\infty \eta_\epsilon^-(t,s) \tP^0_{B_0}\E_s f_s \right)_\ta,
$$
uniformly for $t>0$,
where $\tP^0_{B_0}$ is projection onto $\mH^\perp$ from Proposition~\ref{prop:BDprops}.
So it is enough to obtain the bound 
$$
 \left\|  \int_0^\infty |\tP^0_{B_0}\E_s f_s| ds \right\|_2 \lesssim  \|\E\|_* \|f\|_{\mY_\delta}.
 $$
On one hand, we obtain from Proposition~\ref{prop:reverseholder} the estimate
\begin{multline*}
  \left\| \int_1^\infty |\tP^0_{B_0}\E_s f_s| ds \right\|_2 \lesssim \|\E\|_\infty \int_1^\infty \|f_s\|_2 ds \\
  \lesssim  \|\E\|_\infty\left( \int_1^\infty \|f_s\|_2^2 e^{\delta s} ds \right)^{1/2}
  \lesssim \|\E\|_\infty \|f\|_{\mY_\delta}.
\end{multline*}
On the other hand, note that
$A$, hence $B_0^{-1}$, is pointwise strictly accretive by Lemma~\ref{lem:pointwiseaccr} 
and by the  explicit expression  in Lemma~\ref{lem:proj1D} (expressed in other coordinates),    $\tP^0_{B_0}$  maps into constant functions and $|\tP^0_{B_0}u| \lesssim \int_{S^1}|u(x)|dx$.
Thus
$$
  \left\| \int_0^1 |\tP^0_{B_0} \E_s f_s| ds \right\|_2   \lesssim \int_0^1 \int_{S^1}|\E_s(x)| |f_s(x)|  dxds.
$$

Pick $h:\R_+\times S^1\to \V$ such that $|h_s(x)|=1$ and $|\E_s(x)h_s(x)|= |\E_s(x)|$ when
$s<1$, and $h_s(x)=0$ when $s>1$.
Cauchy--Schwarz inequality yields
$$
  \int_0^1 \int_{S^1}|\E_s(x)| |f_s(x)|  dsdx\lesssim \|\E h\|_{\mY^*} \|f\|_\mY\le \|\E\|_*\|h\|_\mX\|f\|_\mY\lesssim \|\E\|_* \|f\|_\mY.
$$
This completes the proof of the estimate of
$\| \tS_A^\epsilon f_t \|_2$.

(iii) As in the proof of \cite[Prop.~7.2]{AA1}, replacing $\eta^\pm_\epsilon$ by $\eta^\pm_\epsilon-\eta^\pm_{\epsilon'}$
in the estimates shows convergence of $\tS_A^\epsilon$ and yield the expression for the limit operator. The relation $S_{A}=D\tS_{A}$ follows at the limit from the relation in Lemma~\ref{lem:SAdecomp}.
\end{proof}

We turn to boundary behavior of the integral operators at $t=0$.

\begin{lem}   \label{lem:inttodiff} 
  Assume $\|\E\|_{*}<\infty$.

(i)   Let $f\in \mX$ (or $f\in\mY$) and define $f^0:= S_A f$.
   Then $f^0$ and $f$ satisfy
$$
  (\pd_t + D_0)f^0= D\E f
$$
in $\R_+\times S^n$ distributional sense.
If $f\in \mX$, then there are limits 
$$
   \lim_{t\to 0} t^{-1}\int_t^{2t} \| S_A f_s - h^- \|_2^2 ds=0,
$$
where 
$
h^- := - \int_0^\infty e^{-s\Lambda} E_0^- D \E_s f_s ds\in E_0^-\mH$
has bounds $\|h^-\|_{2}\lesssim \|f\|_{\mX}$. 

(ii) Let $n\ge 2$. If $f\in\mY$ and $v:= \tS_A f$, then
$$
  (\pd_t + \tD_0 )v= \E f
$$
in $\R_+\times S^n$ distributional sense, and there are limits
$$
   \lim_{t\to 0}  \| \tS_A f_t - \tilde h^- \|_2=0,
$$
where
$\tilde h^- := -\int_0^\infty e^{-s\tilde \Lambda} \tE_0^- \E_s f_s ds\in \tE_0^- L_2$
has bounds $\|\tilde h^-\|_{2}\lesssim \|f\|_{\mY}$.
If $n=1$, these results for $\tS_A f$ hold when replacing $\mY$ by $\mY_\delta$,
for any fixed $\delta>0$.
\end{lem}

\begin{proof}
(i)
By the convergence properties of $S_A^\epsilon$ from Theorem~\ref{thm:XYbounded},
it suffices to show
that for $\phi\in C^\infty_0(\R_+\times S^n; \C^{(1+n)m})$ there is convergence
$$
  \int \Big((-\pd_t \phi_t+ B_0^*D+\sigma N)\phi_t, f^\epsilon_t\Big) dt\to \int (D\phi_s, \E_s f_s)ds,\qquad \epsilon\to 0,
$$
where $f^\epsilon_t:=S_A^\epsilon f_t$.
For the term $(0,t)$-integral, Fubini's theorem and integration by parts gives
\begin{multline*}
  \int_0^\infty \int_0^t \eta^+_\epsilon(t,s) ((-\pd_t+ \Lambda^*)\phi_t, e^{-(t-s)\Lambda} E_0^+D \E_s f_s ) ds dt \\
= - \int_0^\infty \left( \int_s^\infty \eta^+_\epsilon(t,s) D (E_0^+)^*  \pd_t (e^{-(t-s) \Lambda^*} \phi_t)dt,  \E_s f_s  \right) ds \\
  =\int_0^\infty \left( \int_s^\infty (\pd_t\eta^+_\epsilon)(t,s) D(E_0^+)^* e^{-(t-s) \Lambda^*}  \phi_t dt, \E_s f_s  \right) ds 
    \to \int_0^\infty ( D(E_0^+)^*\phi_s, \E_s f_s  ) ds.
\end{multline*}
Adding the corresponding limit for the $(t,\infty)$-integral, we obtain in total the limit
$\int_0^\infty( D\phi_s, \E_s f_s)ds$, since $D((E^+_0)^*+ (E^-_0)^*)= ((E^+_0+ E^-_0)D)^*= D^*=D$.

To prove the limit of $S_A f_t$ for $f\in\mX$, we note from the proof of Theorem~\ref{thm:XYbounded}
that 
$$
  S_A f_t= Z_{A}f_t + e^{-t\Lambda} \int_0^\infty e^{-s\Lambda} E_0^- D \E_s f_s ds,
$$
where $Z_{A}f\in \mY^*$. When taking limits $\epsilon\to 0$, we have used \cite[Thm.~6.8, Lem.~6.9]{AA1}.
This proves the stated limit.

(ii)
To prove $(\pd_t + \tD_0 )v= \E f$, we let $t\in (a,b)$ and 
differentiate $\tS_A^\epsilon f$ to get
$$
  \pd_t \tS_A^\epsilon f_t = \frac 1\epsilon \int_\epsilon^{2\epsilon}  
  e^{-s\tilde\Lambda} (\tE_0^+ \E_{t-s}f_{t-s}+ \tE_0^-\E_{t+s}f_{t+s}) ds- \tD_0 (\tS_A^\epsilon f_t),
$$
for small $\epsilon$.
The first term on the right is seen to converge to $\E f$ in $L_2(a,b;L_2)$ as $\epsilon\to 0$, with an 
argument as in \cite[Thm.~8.2]{AA1}.
Note that this uses $\tE_0^+ + \tE_0^-= I$, which holds also when $n=1$ by Definition~\ref{defn:hardyprojs1D}.
Letting $\epsilon\to 0$, we obtain $\pd_t v= \E f- \tD_0 v$ in distributional sense, 
since $(a,b)$ was arbitrary.

The limit for $\tS_A f_t$ when $f\in \mY$ (or $\mY_\delta$ when $n=1$) 
is proved as in \cite[Prop.~7.2, Lem.~6.9]{AA1}. Note in particular this uses an identity
$$
  \tS_A f_t= \tZ_{A}f_t + e^{-t\st\Lambda} \int_0^\infty e^{-s\st\Lambda} \tE_0^-  \E_s f_s ds,
$$
with $\tZ_{A}f \in C(0,\infty; L_{2})$ and $\lim_{t\to 0} \tZ_{A}f_{t}=0$ in $L_{2}$.
\end{proof}

%
%
%
%
%
\section{Representation and traces of solutions}     \label{sec:representation}

We now come to the heart of the matter. The natural classes of solutions for the Dirichlet and Neumann problems, with $L_2$ boundary data,
use the spaces $\mY^o\approx \mY$ and $\mX^o\approx \mX$ from Definition~\ref{defn:XY}.

\begin{defn}    \label{defn:XYsol}
(i)
  By a $\mY^o$-{\em solution to the divergence form equation}, with coefficients $A$, 
  we mean a weak solution $u$ of $\divv_{\bx} A \nabla u=0$ in $\bO^{1+n}$ with 
  $\| \nabla_{\bx} u \|_{\mY^o}<\infty$.

(ii)
  By an $\mX^o$-{\em solution to the divergence form equation}, with coefficients $A$, 
  we mean the gradient $g:= \nabla_\bx u$ of a weak solution $u$ of 
  $\divv_{\bx} A \nabla u=0$ in $\bO^{1+n}$ with 
  $\| g \|_{\mX^o}<\infty$.
\end{defn}

Note the slight abuse of notation when referring to the gradient $\nabla_\bx u$ rather than $u$
as an $\mX^o$-solution. 
The reason for this convention, here as well as in \cite{AA1}, is that the Neumann and regularity
problems are BVPs for $g$ (and not for the potential $u$), and $\mX^o$-solutions is the natural class 
of solutions for these problems. This point of view is the one that lead us to our representations.  However, when more convenient we  call the potential $u$ itself an $\mX^o$-solution.

\begin{rem}
(i) No boundary trace is assumed in our definitions, but will be deduced.

(ii) The semi-norm $\|\nabla_\bx u\|_{\mY^o}$ on $\mY^o$-solutions is modulo constants,
which is unusual for Dirichlet problems. Once we have shown that $\mY^o$-solutions
have boundary traces, we will be able to put constants back in the norm in a natural way.

(iii) For any $\mX^o$-solution $g$, the potential $u$ has a boundary trace in 
appropriate sense (replacing pointwise values by averages) and the trace belongs to
$W^1_2(S^n; \C^m)$. This is essentially in \cite{KP}. We also recover this from our representations. See Section~\ref{sec:conjugate}.
\end{rem}

Here and subsequently, we use the notation $e^{-t\Lambda} g$ to denote the function 
$(t,x)\mapsto (e^{-t\Lambda} g)(x)$. Similarly for $e^{-t\tilde\Lambda} g$.

\subsection{$\mX^o$-solutions}

We begin with representation and boundary trace for solutions of the  corresponding ODE.

\begin{thm}   \label{thm:inteqforNeu}
  Assume that $\|\E\|_*<\infty$. Let $f\in \mX$.
  Then $f\in L_2^\loc(\R_+;\mH)$ satisfies $\pd_t f+ (DB+\tfrac {n-1}2 N) f=0$ in $\R_+\times S^n$ distributional sense 
  if and only if $f$
   satisfies the equation
\begin{equation}   \label{eq:ODEonintegral}
    f_t= e^{-t\Lambda} h^+ + S_Af_t,\qquad \text{for some } h^+\in E_0^+\mH.
 \end{equation}
In this case, $f$ has limit 
\begin{equation}  \label{eq:neuavlim}
\lim_{t\to 0} t^{-1} \int_t^{2t} \| f_s -f_0 \|_2^2 ds =0,
\end{equation}
where $f_0:= h^++h^-$ and $h^-:= -\int_0^\infty e^{-s\Lambda} E_0^- D \E_s f_s ds\in E_0^-\mH$,
with estimates 
$$
 \max (\|h^+\|_2, \|h^-\|_2) \approx \|f_0\|_2 \lesssim \|f\|_\mX.
$$
If furthermore $I-S_A$ is invertible on $\mX$, then 
\begin{equation}\label{eq:ansatzRN}
f=(I-S_{A})^{-1} e^{-t\Lambda} h^+
\end{equation}
and $\| f \|_\mX\lesssim \|h^+\|_2$.
\end{thm}

\begin{proof}
  The proof is an adaption of \cite[Thm.~8.2]{AA1} to which we refer for details.
  Here is a quick summary.

We show that $f$ satisfies (\ref{eq:firstorderODE}) if and only if $f$ satisfies
(\ref{eq:ODEonintegral}).
Assume (\ref{eq:firstorderODE}) and apply Lemma~\ref{lem:preints}. 
Letting $\epsilon\to 0$ and applying Theorem~\ref{thm:XYbounded}, 
we obtain the stated equation for $f$, with $h^+$ as a certain weak limit as
in part (i) of the proof of \cite[Thm.~8.2]{AA1}, with $\Lambda= |D_0|$ here.

Conversely, if $f\in\mX$ satisfies (\ref{eq:ODEonintegral}), 
then we apply Lemma~\ref{lem:inttodiff}
with $f^o:= f- e^{-t\Lambda} h^+$. Since $(\pd_t+ D_0)e^{-t\Lambda}h^+=0$ and 
$e^{-t\Lambda}h^+\in \mX$ by 
Theorem~\ref{thm:NT}, it follows that $f$ satisfies (\ref{eq:firstorderODE}).

Lemma~\ref{lem:inttodiff} also shows existence of the limit $f_0$.
The stated estimates follow as in part (iii) of the proof of \cite[Thm.~8.2]{AA1}.

If $I-S_{A}$ is invertible, the equation \eqref{eq:ansatzRN} follows immediatly from (\ref{eq:ODEonintegral}), and the estimate  $\| f \|_\mX\lesssim \|h^+\|_2$ follows again from 
Theorem~\ref{thm:NT}.
\end{proof}

\begin{thm}   \label{thm:inteqforNeuandg}
  Assume that $\|\E\|_*<\infty$.
   Then $g$ is an $\mX^o$-solution to the divergence form equation with coefficients $A$ if and only if
   the corresponding conormal gradient $f \in \mX$ 
   satisfies the equation
\begin{equation}  
    f_t= e^{-t\Lambda} h^+ + S_Af_t,\qquad \text{for some } h^+\in E_0^+\mH.
 \end{equation}
In this case, $g$ has limit 
$$
  \lim_{r\to 1} \frac 1{1-r} \int_{r<|\bx|<(1+r)/2} | g(\bx) - g_1(x) |^2 d\bx =0,
$$
where $g_1:= (B_0 f_0)_\no\rad+ (f_0)_\ta$ and $\|g_1\|_2\lesssim \|g\|_{\mX^o}$ holds.
If furthermore $I-S_A$ is invertible on $\mX$, then $\|h^+\|_2\approx \|g_1\|_2\approx \| g\|_{\mX^o}$.
\end{thm}

\begin{proof} The equivalence follows from Corollary~\ref{cor:conormal} and Theorem~\ref{thm:inteqforNeu}. The limit and the estimates follow on applying the conormal gradient-to-gradient map of Proposition~\ref{prop:divformasODE} from the ones satisfied by $f$.
\end{proof}

It is worth specifying the previous theorem in the case of radially independent coefficients.

\begin{cor} Assume $A$ is radially independent. Then any $\mX^o$-solution has corresponding conormal gradient given by $f=e^{-t\Lambda}h^+$ for a unique $h^+\in E_{0}^+\mH$. 
\end{cor}

\begin{rem} 
A careful examination of the proof of Theorem~\ref{thm:inteqforNeu} 
in the case of radially independent coefficients, shows in fact that for $f\in L_2^{loc}(\R_+;\mH)$ 
 the weaker condition 
$\sup_{0<t<1/2} \frac 1t\int_t^{2t} \| f_s \|_2^2 ds<\infty$ 
is sufficient to obtain this corollary, as in this case $S_{A}=0$.
\end{rem}

\subsection{$\mY^o$-solutions}

We now turn to representations and boundary behavior pertaining to $\mY^o$-solutions.

\begin{thm}   \label{thm:inteqforDir}
  Assume that $\|\E\|_*<\infty$ and $f\in \mY$.

(i)
    Then $f\in L_2^\loc(\R_+;\mH)$ satisfies $\pd_t f+(DB+\tfrac {n-1}2 N) f=0$ in $\R_+\times S^n$ distributional sense  if and only if
  $f$
   satisfies the equation
\begin{equation}   \label{eq:ODEonintegralDir}
   f_t = De^{-t\tilde\Lambda}\tilde h^+ + S_A f_t, \qquad \text{for some } \tilde h^+\in \tE_0^+ L_2.
\end{equation}
Here 
$\tilde h^+$ is unique modulo $\tE_0^+ \mH^\perp$ and
$\|\tilde h^+\|_{L_2/\mH^\perp} \lesssim \|f\|_\mY$, and if furthermore 
$I-S_A$ is invertible on $\mY$
then \begin{equation}\label{eq:ansatzDf}
f=(I-S_{A})^{-1} De^{-t\tilde\Lambda}\tilde h^+
\end{equation}
with $\|f\|_\mY\lesssim \|\tilde h^+\|_{L_2/\mH^\perp}$.

(ii)
If (\ref{eq:ODEonintegralDir}) holds, let $v_t:= e^{-t\tilde\Lambda}\tilde h^+ + \tS_A f_t$.
Then $f=Dv$ and $\pd_{t}v+(BD-\tfrac{n-1}2 N)v=0$,  and $v_t$ has $L_2$ limit
\begin{equation}    \label{eq:dirlimits}
   \lim_{t\to 0} \| v_t  - v_0\|_2=0,
\end{equation}
where $v_0:= \tilde h^+ + \tilde h^-$ and
$\tilde h^-:= -\int_0^\infty e^{-s\tilde\Lambda}\tE_0^- \E_s f_s ds\in \tE_0^- L_2$, with
estimates $\|\tilde h^-\|_2\lesssim \|f\|_\mY$ and
\begin{equation}    \label{eq:vtest}
  \|v_t\|_2 \lesssim \|\tilde h^+\|_2 + \|f\|_\mY,\qquad\text{for all } t>0.
\end{equation}
\end{thm}

\begin{proof}
The proof is an adaption, with some modifications, of \cite[Thm.~9.2]{AA1} to which we refer  for  omitted details.

(i)
Assume (\ref{eq:firstorderODE}).
We apply Lemma~\ref{lem:preints} to $f$. 
Letting $\epsilon\to 0$ and applying Theorem~\ref{thm:XYbounded}, we obtain for $f$
 the equation
$$
   f_t = \tilde f_t + S_A f_t,
$$
with the limit
$\tilde f_t:= \lim_{\epsilon\to 0} \epsilon^{-1}\int_\epsilon^{2\epsilon} e^{-(t-s)\Lambda} E_0^+ f_s ds$.
From here, one can proceed as in \cite[Thm.~9.2]{AA1} to represent $\tilde f_t$
as $D_0e^{-t\Lambda} h^+$ for some $h^+\in E_0^+\mH$, 
or use a simpler argument (owing to the boundedness of the boundary here):
since $D_0: E_0^+\mH\to E^+_0\mH$ is surjective,
there exists  $h_t\in E^+_0\mH$  such that $\tilde f_t= D_0 h_t$.
From there and $\tilde f_{t_0+t}= e^{-t\Lambda} \tilde f_{t_0}$,
we conclude as in \cite{AA1} that the weak $L_2$-limit $h^+:= \lim_{t\to 0} h_t$
exists and that $\tilde f_t= D_0 e^{-t\Lambda} h^+$.

To write $D_0 e^{-t\Lambda} h^+$ as $D e^{-t\tilde \Lambda} \tilde h^+$ for some
$\tilde h^+\in\tE_0^+ L_2$, we use Lemma \ref{lem:htotildeh}. Indeed, there is an isomorphism  $M:\mH \to   L_{2}/\mH^\perp$ with $D_{0}=D\circ M$ on $\dom(D_{0})$.  It is easy to see that the restriction of $M$  to $ E_{0}^+\mH$ maps onto $\tE_0^+L_2 / \tE_0^+\mH^\perp$. Now, on $\dom(D_0)$,   $D_0 e^{-t\Lambda} = e^{-t\Lambda} D_0  = e^{-t\Lambda} D \circ M=
D e^{-t\tilde\Lambda} \circ M$. By density and boundedness, the left and right terms agree on $\mH$. 
Thus,  $\st h^+=Mh^+\in \tE_0^+L_2 / \tE_0^+\mH^\perp$ satisfies $D_0 e^{-t\Lambda} h^+=D e^{-t\tilde \Lambda} \tilde h^+$.

We conclude that $f_t= D e^{-t\tilde \Lambda} \tilde h^+ + \tS_A f_t$, with estimates
\begin{equation}  \label{eq:tildehestimate}
\|\tilde h^+\|_{L_2/\mH^\perp} \approx \|h^+\|_2 \approx
 \|D_0 e^{-t\Lambda} h^+\|_{\mY}
 = \|f-S_A f\|_\mY\lesssim \|f\|_\mY.
\end{equation}
The middle equivalence uses Theorem~\ref{thm:NT}. 

(i')
Conversely, if $f\in\mY$ satisfies (\ref{eq:ODEonintegralDir}) for some 
$\tilde h^+\in \tE^+_0 L_2$, 
then we apply Lemma~\ref{lem:inttodiff}
with $f^o= f- D e^{-t\tilde\Lambda} \tilde h^+ =  f- D_{0} e^{-t\Lambda} h^+$ with $h^+\in E_{0}^+\mH$ given by the isomorphism above. Since $(\pd_t+ D_0)D_{0}e^{-t\Lambda} h^+=0$, it follows that $f$ satisfies (\ref{eq:firstorderODE}).
For the estimate of $\|f\|_\mY$ when $I-S_{A}$ is invertible on $\mY$, 
use that the last estimate in \eqref{eq:tildehestimate} in this case is $\approx$.

(ii)
Lemma~\ref{lem:inttodiff} and Theorem~\ref{thm:Cbounded}
show the ODE satisfied by $v$, existence of the limit $v_0$ and
the estimates of $\|v_t\|_2$ and $\|\tilde h^-\|_2$.
This completes the proof.
\end{proof}

\begin{cor}   \label{cor:diransatz}
Assume that $\|\E\|_*<\infty$.
With the notation from Theorem~\ref{thm:inteqforDir}, the following holds.

(i) Any $\mY^o$-solution $u$ to the divergence form equation
has representation $u_{r}=r^{-\frac{n-1}{2}} (v_{t})_{\no}$ with $r=e^{-t}$, for some $v$ as in 
Theorem~\ref{thm:inteqforDir},  
boundary trace in the sense $\lim_{r\to 1}\|u_r-u_1\|_2=0$, and there are estimates
$$
 \|u_r\|_2 \lesssim r^{-\frac{n-1}{2}} \|\nabla_\bx u\|_{\mY^o}+ \left|\int_{S^n} u_1(x) dx\right|, \quad r\in (0,1).
$$

(ii) The map taking $\mY^o$-solutions $u$ to boundary functions 
$\tilde h^+=\tE_{0}^+v_{0} \in  \tE_0^+L_2$  is well-defined and bounded in the sense that
$$ \| \tilde h^+\|_2\lesssim \|\nabla_\bx u\|_{\mY^o}+ \left|\int_{S^n} u_1(x) dx\right|.$$

(iii) If furthermore $I-S_A$ is invertible on $\mY$, then this map is an isomorphism  
and its inverse  $\tE_0^+ L_2 \ni \tilde h^+\to u\in \{\mY^o \text{-solutions}\}$
is given by
\begin{equation}\label{eq:ansatzDu}
  u_r:= r^{-\frac{n-1}{2}} \Big((I + \tS_A (I-S_A)^{-1} D ) e^{-t\tilde \Lambda}  \tilde h^+\Big
   )_\no,
\end{equation} 

with estimates $\|\nabla_\bx u\|_{\mY^o}+ |\int_{S^n} u_1(x) dx|\approx \| \tilde h^+\|_2$. 
\end{cor}

\begin{proof}
(i)
Let $f$ be the conormal gradient of $u$ and define $\tilde h^+$ and $v$ applying Theorem~\ref{thm:inteqforDir}.
As in the proof of Proposition~\ref{prop:ODEtoPotential}, it follows that
\begin{equation*}
u_r= r^{-\sigma} (v_t)_\no+ c
\end{equation*} 
for some $c\in\C^m$, where $r= e^{-t}\in (0,1)$ and $\sigma=\frac{n-1}{2}$.

Recall  that by  (\ref{eq:ODEonintegralDir}), $\tilde h^+$ is uniquely defined in $\tE_0^+ L_2$  modulo 
$ \tE_0^+ \mH^\perp$ and we now use this freedom to  
choose it in $\tE_0^+ L_2$  such 
that $c=0$. 
Indeed, by 
Lemma~\ref{lem:spectralsplits},  
$\tE_0^+ \mH^\perp= N^-\mH^\perp= \{\begin{bmatrix} c & 0 \end{bmatrix}^t ; c\in \C^m\} $ and since 
$\st\Lambda=\sigma I$ on $\mH^\perp$, we have
\begin{equation}\label{eq:c}
e^{-t\tilde\Lambda} (\begin{bmatrix} c & 0 \end{bmatrix}^t)= e^{-\sigma t} \begin{bmatrix} c & 0 \end{bmatrix}^t, \qquad c\in \C^m.
\end{equation} 
(The superscript $t$ of the brackets denotes transpose.)
Replacing 
$\tilde h^+$ by $\tilde h^+ - \begin{bmatrix} c & 0 \end{bmatrix}^t$, then  $f_t$ remains unchanged,
 $e^{-t\tilde\Lambda} \tilde h^+$ is replaced by 
$e^{-t\tilde\Lambda} \tilde h^+- e^{-\sigma t} \begin{bmatrix} c & 0 \end{bmatrix}^t$,
 and $(v_t)_{\no}$  by 
$(v_t)_{\no}-  e^{-\sigma t}c$.
Thus we may assume $c=0$.

As $v_{t}$ has an $L_{2}(S^n; \C^{(1+n)m})$ limit $v_{0}$ when $t\to 0$, one can set $u_1:= (v_0)_\no$ and $u_{r}$ converges in $L_{2}(S^n;\C^m)$ to $u_{1}$.
For the estimate on $\|u_r\|_2$ it suffices to prove
$$
 \left\|u_r- m\right\|_2 \lesssim r^{-\frac{n-1}{2}} \|\nabla_\bx u\|_{\mY^o}, \quad r\in (0,1).
 $$
with $m$ the mean value of $u_{1}$ on $S^n$. We may assume that $m=0$  as by \eqref{eq:c} this amounts to modifying $\tilde h^+$ modulo $N^-\mH^\perp$ without changing the conormal gradient $f$ of $u$.
We have 
$$
\|u_r\|_2 \le r^{-\sigma}\|v_t\|_2 \lesssim r^{-\sigma} (\|\tilde h^+\|_2+ \|f\|_\mY).
$$
By orthogonal projection onto $N^-\mH^\perp$, it follows
$\|\tilde h^+\|_2 \approx \|\tilde h^+\|_{L_2/\mH^\perp} +  |\int_{S^n} (\tilde h^+)_\no dx|$ 
since $\tilde h^+\in \tE_0^+L_2$. 
We can now conclude since $\|\tilde h^+\|_{L_2/\mH^\perp}\lesssim \|f\|_\mY$ and,  since $m=0$,
$$\left|\int_{S^n} (\tilde h^+)_\no(x) dx\right| = \left|\int_{S^n} (u_1-(\tilde h^-)_\no)(x)dx\right|\lesssim \|\tilde h^-\|_2\lesssim \|f\|_\mY.$$

(ii) The argument using \eqref{eq:c} shows that given a $\mY^o$-solution $u$ and its conormal gradient $f$, there exists   $\st h^+ \in \tE_0^+L_2$ such that $u_r= r^{-\sigma} (e^{-t\tilde\Lambda} \tilde h^+ + \tS_A f_t)_\no$. Moreover,  $\tilde h^+=\tE_{0}^+v_{0}$ by construction and  the estimate $ \| \tilde h^+\|_2\lesssim \|\nabla_\bx u\|_{\mY^o}+ \left|\int_{S^n} u_1(x) dx\right| $ follows from the above argument. To define the map and prove its boundedness, it suffices to show  uniqueness of such $\st h^+  \in \tE_0^+L_2 $.
So assume $u_r= r^{-\sigma} (e^{-t\tilde\Lambda} \tilde h^+ + \tS_A f_t)_\no=  r^{-\sigma} (e^{-t\tilde\Lambda} \tilde h_{1}^+ + \tS_A f_t)_\no$ with $f$ the conormal gradient of $u$ and $\st h^+, \st h^+_{1}\in \tE_{0}^+L_{2}$. This implies
that $f_{t}=De^{-t\st\Lambda} \st h^+ + S_{A}f_{t}=  De^{-t\st\Lambda} \st h^+_{1} + S_{A}f_{t}$ so we know that $\st h^+-\st h_{1}^+\in \tE_{0}^+\mH^\perp$ by Theorem~\ref{thm:inteqforDir}.  As $\tE_{0}^+\mH^\perp =N^-\mH^\perp$, write $\st h^+-\st h_{1}^+=\begin{bmatrix} c & 0 \end{bmatrix}^t$ with $c\in \C^m$.
We have from \eqref{eq:c}, $ 0= r^{-\sigma} (e^{-t\tilde\Lambda} (\tilde h^+- \tilde h^+_{1}))_{\no}= c$.

(iii)
Given $\tilde h^+\in  \tE_0^+L_2$, define $f_t := (I-S_A)^{-1} D e^{-t\tilde \Lambda}\tilde h^+$,
$v_t:= e^{-t\tilde\Lambda} \tilde h^+ + \tS_A f_t$ and $u_r:= r^{-\sigma}(v_t)_\no$.
By Theorem~\ref{thm:NTtilde} and Lemma~\ref{lem:inttodiff},
$v$ satisfies the equation $\pd_t v+ \tD_0 v=0$, and by
Proposition~\ref{prop:ODEtoPotential}, $u$ extends to a $\mY^o$-solution and $f$ is the conormal gradient of $u$. For the continuity estimate $ \|\nabla_\bx u\|_{\mY^o}+ \left|\int_{S^n} u_1(x) dx\right| \lesssim  \| \tilde h^+\|_2$, 
Theorem~\ref{thm:inteqforDir} implies $\|f\|_\mY\lesssim \|\tilde h^+\|_2$
and 
$|\int_{S^n} u_1(x) dx|\lesssim \|u_1\|_2\lesssim \|v_0\|_2\lesssim \|\tilde h^+\|_2+ \|f\|_\mY\lesssim \|\tilde h^+\|_2$.
This map  clearly inverts the map in (ii).
This completes the proof.
\end{proof}

It is worth specifying the Corollary~\ref{cor:diransatz} in the case of radially independent coefficients.

\begin{cor}\label{cor:uniquenessdir} 
Assume $A$ is radially independent. Then any $\mY^o$-solution is given by  $u=r^{-\frac{n-1}{2}}(e^{-t\tilde \Lambda}\tilde h^+)_{\no}$ for a unique $\tilde h^+ \in  \tE_0^+L_2$  with 
$
\|\tilde h^+\|_{2}\approx \|\nabla_\bx u\|_{\mY^o}+ \left|\int_{S^n} u_1 dx\right|.
$
\end{cor}

\subsection{Conclusion}

It is clear from \eqref{eq:ansatzRN} that provided $I-S_{A}$ is invertible on $\mX$,
 the ansatz 
$$
  E^+_0\mH \to \mX: h^+\mapsto f_t= (I-S_A)^{-1} e^{-t\Lambda} h^+
$$
maps onto all conormal gradients of  $\mX^o$-solutions  
to the divergence form equation with
coefficients $A$.

Similarly,   \eqref{eq:ansatzDu}  implies  that  provided $I-S_{A}$ is invertible on $\mY$,
 the ansatz 
$$
  \tE^+_0\mH \to \mY^o: \st h^+\mapsto u_r:= r^{-\frac{n-1}{2}} \Big((I + \tS_A (I-S_A)^{-1} D ) e^{-t\tilde \Lambda}  \tilde h^+\Big
   )_\no,
$$
maps onto all  $\mY^o$-solutions  
to the divergence form equation with
coefficients $A$.
 
Thus we have a way of constructing solutions and our two main goals towards well-posedness results are the following.

First  understand when invertibility of $I-S_{A}$ holds. This will be done in Section~\ref{sec:fredholm}.

Secondly,  introduce   the boundary maps that connect  the traces of solutions to the data for the BVPs and invert them. This is the object of Section~\ref{sec:BVPs}. 

Before we do this, we continue with different \textit{a priori} representations of solutions in the next section. This will be useful to prove non-tangential maximal estimates and obtain convergence of Fatou type at the boundary.

%
%
%
%
%
\section{Conjugate systems}\label{sec:conjugate}

The results in the preceding section allow to  represent $\mX^o$-solutions in terms of the conormal gradient $f$. Actually, if one is interested in $u$ itself, one can try to further describe   the corresponding potential vector $v$. Similarly,  representation of  $\mY^o$-solutions is embedded into a potential vector $v$ but it could be interesting to describe the properties of the conormal gradient $f$.  Both are related by the rule $Dv=f$. This leads us to the following notion.

\begin{defn} 
A {\em pair of conjugate systems} to the divergence equation with coefficients $A$ is a pair $(v,f) \in L_2^\loc(\R_+;L_2(S^n; \V)) \times  L_2^\loc(\R_+;L_{2}(S^n; \V))$ with 

(i)  $v_{t}\in \dom(D)$ for almost every $t$
and  $\int_1^\infty\|Dv_t\|_2^2 dt<\infty$,

(ii)  $v$ is an $\R^+\times S^n$-distributional solution   of 
\eqref{eq:BDODE},

(iii) $f_{t}=Dv_{t}$ for almost every $t>0$, 

(iv)  $f$ is a $\mH$-valued  $\R^+\times S^n$-distributional solution  of  
\eqref{eq:firstorderODE}.  
\end{defn}

By Proposition~\ref{prop:ODEtoPotential} 
and  its proof,  a pair of conjugate systems is completely determined by $v$ satisfying (i-ii). That is, $f$ defined by (iii) automatically satisfies (iv). Moreover,  the function  
\begin{equation}\label{eq:associated}
  u_r:= r^{-(n-1)/2} (v_{t})_\no,\qquad r=e^{-t}\in (0,1),
\end{equation}
extends to  a weak solution of $\divv_\bx A \nabla_\bx u=0$ in $\bO^{1+n}$ and $f$ must be the conormal gradient  of $u$. 
 {\em We say that a weak solution $u$ and a pair of conjugate systems $(v,f)$ to the divergence form equation for which  \eqref{eq:associated} holds are associated.} 

 It is our goal to give a  description of the  pair (not only $f$ or $v$) in each case. Recall that in integrating $Dv=f$, $v_t$ is only determined by $f_t$ modulo $\mH^\perp$ so there is a choice to make. 

\begin{thm} \label{thm:conj}
Assume $\|\E\|_{*}<\infty$. Let $u$ be an $\mX^o$- or $\mY^o$-solution.   
Then $u$ has an $L_{2}(S^n; \C^m)$ trace $u_{1}$ at the boundary and there exists an associated pair of conjugate systems given  
by
\begin{equation}
\label{eq:repconjugatepairs}
\begin{cases}  v_{t}&=e^{-t\st\Lambda}v_{0}+ \st w_{t}\\
f_{t}&=e^{-t\Lambda}f_{0}+ w_{t}
\end{cases}
\end{equation}
with the following properties.
\begin{enumerate}
  \item[(i)] If $u$ is an $\mX^o$-solution, then $u_{1}\in W^1_{2}(S^n; \C^m)$,  $(v_{0}, f_{0})\in \dom(D) \times \mH$ with $Dv_{0}=f_{0}$, $\|\nabla_{S}u_{1}\|_{2}\lesssim\|f_{0}\|_{2}\lesssim \|\nabla_{\bx}u\|_{\mX^o}$,  $\|v_{0}\|_{2} \lesssim \|\nabla_\bx u\|_{\mX^o}+ \left|\int_{S^n} u_1 dx\right|$, $D\tilde w_{t}= w_{t} \in \mY^*$, $v_{t}\in C(\R^+; L_{2})$ and $\|v_{t}-v_{0}\|_{2}+\|\st w_{t}\|_{2}=O(t)$ for $t>0$. 
  \item[(ii)] If $u$ is a $\mY^o$-solution, then $u_{1}\in L_{2}(S^n; \C^m)$,   $(v_{0}, f_{0})\in L^2 \times \dot W^{-1}_{2}(S^n;\V)$ with $Dv_{0}=f_{0}$,   $\|u_{1}\|_{2}\lesssim \|v_{0}\|_{2} + \|f_{0}\|_{\dot W^{-1}_{2}}   \lesssim \|\nabla_\bx u\|_{\mY^o}+ \left|\int_{S^n} u_1 dx\right|$,  $D\tilde w_{t}= w_{t} \in \mY$, $v_{t}\in C(\R^+; L_{2})$ and $\|v_{t}-v_{0}\|_{2}+\|\st w_{t}\|_{2}=O(1)$ for $t>0$ and $o(1) $ for $t\to 0$.
\end{enumerate} 
\end{thm}

Besides $\mX^o$- and $\mY^o$-solutions to the divergence form equation, we shall 
in the following sections also consider the following  classical class of variational solutions.

\begin{defn}   \label{defn:variationalsolutions} 
  By a {\em variational solution to the divergence form equation}, with coefficients $A$, 
  we mean  a weak solution of $\divv_{\bx} A \nabla u=0$ in $\bO^{1+n}$ with 
  $\| \nabla_{\bx} u \|_{2}<\infty$.
\end{defn}

It is illuminating to see how the representation for variational solutions lies in between the ones for $\mX^o$- and $\mY^o$-solutions, independently of solvability issues which are well-known for variational solutions. We state this result without proof as it is not  used in this paper. 
Note that, as compared to Theorem~\ref{thm:conj}, the Carleson condition $\|\E\|_{*}<\infty$ is not needed in the following result.

\begin{prop}
Let $u$ be a variational solution to the divergence form equation with coefficients $A$.
Then $u$ has an $L_{2}(S^n; \C^m)$ trace $u_{1}$ at the boundary and there exists an associated pair of conjugate systems given  
by \eqref{eq:repconjugatepairs} with the following properties:  

$u_{1}\in W^{1/2}_{2}(S^n; \C^m)$, $ (v_{0}, f_{0})\in \dom(|D|^{1/2}) \times \dot W^{-1/2}_{2}(S^n;\V)$ with $Dv_{0}=f_{0}$,  $\|v_{0}\|_{2} \lesssim \|\nabla_\bx u\|_{2}+ \left|\int_{S^n} u_1 dx\right|$, $\|u_{1}\|_{\dot W^{1/2}_{2}}  \lesssim \|f_{0}\|_{\dot W^{-1/2}_{2}} \lesssim \|\nabla_\bx u\|_{2}$,  $D\tilde w_{t}= w_{t} \in L_{2}(\R^+; L_{2})$, $v_{t}\in C(\R^+; L_{2})$ and $\|v_{t}-v_{0}\|_{2}+\|\st w_{t}\|_{2}=O(t^{1/2})$ for $t>0$. 

Here $\dot W^{1/2}_{2}$ is equipped with homogeneous norm and $\dot W^{-1/2}_{2}$ is its dual.
\end{prop}

\begin{proof}[Proof of Theorem~\ref{thm:conj}]
(i)  From Theorem~\ref{thm:inteqforNeu},  we have
$$
f_t= e^{-t\Lambda} h^+ + S_Af_t = e^{-t\Lambda} f_{0} + w_{t},  \quad w_{t}:=S_{A}f_{t}- e^{-t\Lambda}h^-.
$$
with $f_{0}=h^++h^-\in \mH$, $\|f_{0}\|_{2}\lesssim \|\nabla_{\bx}u\|_{\mX^o}$ and $h^-=-\int_0^\infty e^{-s\Lambda} E_0^- D \E_s f_s ds$. 

We define $v_{0}, \st h^+, \st h^-$ and $v$ as follows: 
 $\st h^+$ is the unique element in $\tE_{0}^+L_{2}/\tE_{0}^+\mH^\perp$ such that
$D\st h^+=h^+(= D_{0}(D_{0}^{-1}h^+))$, $\tilde h^-:= -\int_0^\infty e^{-s\tilde\Lambda}\tE_0^- \E_s f_s ds$, $v_{0}=\tilde h^+ + \tilde h^-$ and
$$
v_{t}: =e^{-t\tilde\Lambda}\tilde h^+ + \tS_A f_t= e^{-t\st\Lambda}v_{0}+ \st w_{t}, \quad \st w_{t}:= \tS_A f_t- e^{-t\tilde\Lambda}\tilde h^-.
$$

Clearly, $ \st h^+ \in \dom(D)$. Next, $\st \Lambda \st h^- \in L_{2}$ because $\E f\in \mY^*$, so $\st h^- \in \dom(D)=\dom(\st \Lambda)$ and $D\st h^-=h^-$. So $v_{0}\in \dom(D)$ and $Dv_{0}=f_{0}$. 

The  estimate on  $\|e^{-t\st\Lambda}v_{0} -v_{0}\|_{2}$ follows from  $v_{0}\in \dom(D)$. 

Next, $D\st w_{t}=w_{t}$ by construction and $w_{t}\in \mY^*$ from the proof of Lemma~\ref{lem:inttodiff}. (In fact, $w_{t}$ is nothing but $Z_{A}f_{t}$ defined in that proof.) 

The estimate on $\|\st w_{t}\|_{2}$ follows from
\begin{multline*}
 \st  w_{t}=\int_0^t e^{-(t-s)\tilde\Lambda} \tE_0^+ \E_s f_s ds \\
 - \int_t^\infty (e^{-(s-t)\tilde\Lambda} - e^{-(t+s)\tilde\Lambda}) \tE_0^- \E_s f_s ds +  e^{-t\tilde\Lambda}\int_{0}^t e^{-s\tilde\Lambda} \tE_0^- \E_s f_s ds,      
 \end{multline*}
using  $\E f \in \mY^*$, the uniform boundedness of the semigroup  and  its decay at infinity.  Details are left to the reader.

Eventually,  as in Corollary~\ref{cor:diransatz}, one can adjust $\st h^+$  by adding an element in $N^-\mH^\perp$ such that $u$ and $v$ satisfy \eqref{eq:associated}. In particular, $u$ has an $L_{2}$ trace. It also follows that $f$ is the conormal gradient of $u$ with a limit $f_{0}$ when $t\to 0$ by \eqref{eq:neuavlim}. So $u_{1}\in W^1_{2}(S^n; \C^m)$ with $\|\nabla_{S}u_{1}\|_{2}\lesssim\|f_{0}\|_{2}$.

(ii) By  Corollary~\ref{cor:diransatz}, we have description of 
$$
v_{t}= e^{-t\tilde \Lambda}\st h^+ + \tS_{A}f_{t} =e^{-t\tilde \Lambda}v_{0} + \st w_{t}, \quad \st w_{t}=\tS_{A}f_{t}-e^{-t\st\Lambda}\st h^-, 
$$
with $v_{0}=\st h^++\st h^-$ such that $u$ and $v$ satisfy \eqref{eq:associated} and of trace and growth estimates for $\|e^{-t\tilde \Lambda}v_{0}-v_{0}\|_{2}+\|\st w_{t}\|_{2}$. It remains to consider the representation of $f$. We have by Theorem~\ref{thm:inteqforDir},
$$
f_{t}= De^{-t\tilde \Lambda} \tilde h^+ + S_{A}f_{t}= De^{-t\tilde \Lambda}v_{0} + w_{t}, \quad 
w_{t}= S_{A}f_{t}-De^{-t\st\Lambda}\st h^-=D\st w_{t}.
$$

Define $f_{0}:=Dv_{0}$  in distribution sense, so that $f_{0}\in \dot W^{-1}_{2}(S^n;\V)$ and 
$\|f_{0}\|_{\dot W^{-1}_{2}} \lesssim \|v_{0}\|_{2}$.
We obtain
$$
f_t = e^{-t\Lambda} f_{0}+w_{t}
$$  and here, the action of $e^{-t\Lambda}$ is extended to $\dot W^{-1}_{2}(S^n;\V)$ by extending the interwining formula $De^{-t\Lambda}=e^{-t\st \Lambda}D$.
\end{proof}

%
%
%
%
%
\section{Non-tangential maximal estimates}     \label{sec:NT}

\begin{thm} \label{thm:NTu}
Assume $\|\E\|_{C\cap L_{\infty}}<\infty$. Then any $\mY^o$-solution to the divergence form equation with coefficients $A$ satisfies
$$
 \|u_{1}\|_{2}^2\lesssim  \|\tN^o(u)\|^2_{2}  \lesssim  \int_{\bO^{1+n}} |\nabla_\bx u|^2 (1-|\bx|)d\bx
  + \left| \int_{S^n} u_1(x) dx \right|^2.
$$
When $n=1$, the conjugate $\tilde u$ of a $\mY^o$-solution $u$ also satisfies the estimates
$$
 \|\tilde u_1\|_{2}^2\lesssim  \|\tN^o(\tilde u)\|_{2}^2 \lesssim \int_{\bO^{1+n}} |\nabla_\bx u|^2 (1-|\bx|)d\bx
  + \left| \int_{S^n} \tilde u_1(x) dx \right|^2.
$$
\end{thm}

The proof follows the strategy of \cite{AA1} with a slight modification in view of preparing the proof of almost everywhere non-tangential convergence. 

\begin{proof} 
The estimate $\|\tN^o(u)\|_2\gtrsim \|u_{1}\|_2$ follows from Lemma~\ref{lem:XlocL2} and Corollary~\ref{cor:diransatz}(i). For the upper bound,  we proceed as follows. 
 From the  representation 
$u_{r}=r^{-\sigma} (v_{t})_{\no}$
with 
$ v_{t} =  e^{-t\tilde\Lambda}v_{0}+ \st w_{t}$ in Theorem~\ref{thm:conj}, it is enough to bound   $\|\tN((e^{-t\tilde\Lambda}v_{0})_{\no})\|_{2}$ and $\|\tN(\st w_{\no})\|_{2}$. 
Theorem~\ref{thm:NTtilde}, and Lemma~\ref{lem:NTofSA} below, show
\begin{multline}
   \|\tN^o(u)\|_2\lesssim  \|v_{0}\|_2+ \|f\|_\mY\lesssim  
    \|\tilde h^+\|_2+  \|\tilde h^-\|_2+ \|f\|_\mY \\\lesssim  \|\tilde h^+\|_{L_2/\mH^\perp}+\left| \int_{S^n} \tilde h^+_\no dx \right|+ \|f\|_\mY,
\end{multline}
and  $\|\tilde h^+\|_{L_2/\mH^\perp}\lesssim \|f\|_\mY$, 
$|\int_{S^n} \tilde h^+_\no dx|= |\int_{S^n} (u_1 -\tilde h^-_\no) dx|\lesssim |\int_{S^n} u_1 dx|+\|f\|_\mY$
as in the proof of Corollary~\ref{cor:diransatz}.

When $n=1$, replacing $A$ by the conjugate coefficients $\tilde A$ defined in Section~\ref{sec:disk}  in the above argument, and using $|\nabla_\bx \tilde u|\approx |\nabla_\bx u|$, 
proves the estimates of $\|\tN^o(\tilde u)\|_2$.
\end{proof}

\begin{lem}    \label{lem:NTofSA}
  Assume $\|\E\|_{C\cap L_{\infty}}<\infty$.
 Then we have for each $p<2$, 
 $$
 \|\tN^p(\st w)\|_{2} +\|\tN(\st w_{\no})\|_{2} \lesssim \|\E\|_{C\cap L_{\infty}} \|f\|_{\mY}.
 $$
Here $\tN^p$ is defined similarly to $\tN$, replacing $L_{2}$ averages by $L_{p}$ averages.
When $n=1$, we also have 
$$
\|\tN(\st w_{\ta})\|_{2}\lesssim \|\E\|_{C\cap L_{\infty}} \|f\|_{\mY}.
$$
Furthermore, these estimates hold with $\st w$ replaced by the truncation $\chi_{t<\tau}\st w$,
and $\|f\|^2_\mY$ replaced by $\int_0^\infty \|f_t\|_2^2 \min(t,\tau) dt$, for any $\tau<1$.
\end{lem}

\begin{proof}   
The proof will follow closely the strategy of \cite[Lem.~10.2]{AA1}  on $\R^{1+n}_+$.
We remark that $\tN^p \le \tN$ pointwise. Thus we will work with $\tN$, and indicate when we need to consider $\tN^p$  or  the normal component. 
Recall that $\tN$ estimates the truncation of the function to $t<1$.  

(i)
From $\st w_{t}= \tS_A f_t - e^{-t\tilde\Lambda}\tilde h^-$ and  the definition of $\tilde h^-$,
\begin{multline*}
\st  w_{t}=\int_0^t e^{-(t-s)\tilde\Lambda} \tE_0^+ \E_s f_s ds - \int_t^\infty e^{-(s-t)\tilde\Lambda}  \tE_0^- \E_s f_s ds +  e^{-t\tilde\Lambda}\int_{0}^\infty e^{-s\tilde\Lambda} \tE_0^- \E_s f_s ds
 \\
= \int_0^t e^{-(t-s)\tilde\Lambda}(I-e^{-2s\tilde\Lambda}) \tE_0^+ \E_s f_s ds - \int^\infty_{t} e^{-(s-t)\tilde\Lambda}(I-e^{-2t\tilde\Lambda}) \tE_0^- \E_s f_s ds\\+ e^{-t\tilde\Lambda} \int_0^t e^{-s\tilde\Lambda} \E_s f_s ds
 =I - II + III.
\end{multline*}
Note that  $\tE_0^+ + \tE_0^-= I$  (also in dimension $n=1$) is used in getting $III$. For the first two terms, we use Schur estimates as follows.
Since 
$\|e^{-(t-s)\tilde\Lambda}(I-e^{-2s\tilde\Lambda})\| \lesssim s/t$, we have as in 
 \cite[Lem.~10.2]{AA1}
 $$\|\tN(I)\|_2^2 \lesssim \int_0^1 \left(\int_0^t st^{-1} \|f_s\|_2 ds\right)^2dt/t\lesssim \|\chi_{t<1}f\|_{\mY}^2.$$
Similarly, as $\|e^{-(s-t)\tilde\Lambda}(I-e^{-2t\tilde\Lambda})\| \lesssim t/s$,  we have
\begin{multline*}
\|\tN(II)\|_2^2\lesssim \int_0^1\left(\int_t^\infty ts^{-1}\|f_s\|_2 ds\right)^2dt/t \\
\lesssim  \int_{0}^1 \left(\int_t^\infty t/s^2 ds\right) \left(\int_t^\infty t \|f_s\|_2^2 ds\right) dt/t \\
\lesssim \int_0^\infty \left( \int_0^{\min(s,1)} t dt/t \right) \|f_s\|_2^2 ds= \|f\|_\mY^2.
\end{multline*}

Note that the estimates so far hold for all $\tilde w$, not only for its normal component.
By inspection, the stated estimates of the truncated maximal function hold for these terms.

(ii)
It remains to consider $III= e^{-t\tilde\Lambda} \int_0^t e^{-s\tilde\Lambda} \E_s f_s ds$.
To make use of off-diagonal estimates in Lemma~\ref{lem:offdiagonal}, we need to replace $e^{-t\tilde\Lambda}$ by
the resolvents $(I+it\tD_{0})^{-1}$. To this end, define 
$\psi_t(z):= e^{-t|z|}-(1+itz)^{-1}$ and split the integral
\begin{multline*}
  e^{-t\tilde\Lambda} \int_0^t e^{-s\tilde\Lambda} \E_s f_s ds
  = \psi_t(\tD_{0}) \int_0^\infty e^{-s\tilde\Lambda} \E_s f_s ds -\int_t^\infty \psi_t(\tD_{0}) e^{-s\tilde\Lambda} \E_s f_sds \\
  +\int_0^t (I+ it\tD_{0})^{-1} (e^{-s\tilde\Lambda}-I) \E_s f_sds + (I+it\tD_{0})^{-1} \int_0^t \E_s f_s ds.
\end{multline*}
For the first term, square function estimates show that $\psi_t(\tD_{0}): L_2\to\mY^*\subset\mX$ is continuous, and Theorem~\ref{thm:Cbounded} shows
$\|\int_0^\infty e^{-s\tilde\Lambda} \E_s f_s ds\|_2\lesssim \|f\|_{\mY}$ (or $\lesssim \|f\|_{\mY_\delta}$ when $n=1$, but $\|f\|_{\mY_\delta} \lesssim \|f\|_{\mY}$ for conormal gradients of solutions by Proposition \ref{prop:reverseholder}).
For the second and third terms, we proceed as above for $I$ and $II$ by  Schur estimates    using
$
  \|\psi_t(\tD_{0}) e^{-s\tilde\Lambda}\|  \lesssim t/s,
$
and
$
  \|(I+ it\tD_{0})^{-1} (e^{-s\tilde\Lambda}-I)\|  \lesssim s/t.
$

(iii)
It remains to estimate
$(I+it\tD_0)^{-1}\int_0^t \E_s f_s ds$, and this is where we use $\|\E\|_{C}$. 
Consider first $\tN^p$. 
Fix a Whitney box $W_{0}=W(t_0, x_0)$.
We proceed by a duality argument in the spirit of Corollary~\ref{cor:resolventNT}, and
bound $\|(I+it\tD_0)^{-1}\int_0^t \E_s f_s ds\|_{L_p(W_{0})}$ by testing against
$h\in L_q(W_0; \V)$, $1/p+1/q = 1$.
As in step (iii) of the proof of \cite[Lem.~10.2]{AA1}, this leads to a pointwise estimate implying 
$$
   \left\|\tN^p\left( (I+it\tD_{0})^{-1}\int_0^t \E_s f_s ds  \right)\right\|_2 
   \lesssim \|\E\|_C \|f\|_\mY.
$$
Since the proof here is essentially the same as there, but replacing $\R^n$ by $S^n$, using 
area and maximal functions on $S^n$ instead, we omit the details.
The main ingredients are the $L_p$ off-diagonal estimates
for $(I+it\tD_{0}^*)^{-1}$ from Lemma~\ref{lem:offdiagonal}(i)
and the tent space estimate \cite[Thm.~1(a)]{CMS} of Coifman, Meyer and Stein.

To estimate $\tN((I+it\tD_{0})^{-1}\int_0^t \E_s f_s ds)_{\no})$,
we proceed by duality as above. 
We now instead test against $h\in L_{2}(W_{0};\V)$ with $h_{\ta}=0$ and use
the $L_2\to L_q$ off-diagonal estimates for $(I+it\tD_{0}^*)^{-1}$ from Lemma~\ref{lem:offdiagonal}(ii)
to obtain
$$
   \left\|\tN\left(\left( (I+it\tD_{0})^{-1}\int_0^t \E_s f_s ds \right)_\no \right)\right\|_2 
   \lesssim \|\E\|_C \|f\|_\mY.
$$

It remains to see that, when $n=1$, the $\tN$ estimate also applies to the tangential part $w_{\ta}$.
Consider the transformed conjugate coefficients $\widetilde B= \widehat{\tilde A}$ and
$\widetilde B_0= \widehat{\tilde A_1}$ from the proof of Proposition~\ref{prop:conjugates}, 
and let $\widetilde \E:= \widetilde B_0-\widetilde B$.
Then $\tilde f:= J^t f$ solves $(\pd_t +D \widetilde B)\tilde f=0$,
which yields the estimate of $\|\tN(w_\ta)\|_2$ since 
$(\tS_{A}f)_{\ta}=(J^t\tS_{A}f)_{\no}= (\tS_{\tilde A}\tilde f)_{\no}$.
This completes the proof.
\end{proof}

\begin{rem}    \label{rem:NTofSA}  Note that the proof also shows \textit{a priori} estimates for the operators $\tS_{A}$ when $f$ is not supposed to be a conormal gradient of a solution.
  Assume $\|\E\|_{C\cap L_{\infty}}<\infty$. 
 If $n\ge 2$, then we have for each $p<2$, 
 $$
 \|\tN^p(\tS_{A}f)\|_{2} +\|\tN((\tS_{A}f)_{\no})\|_{2} \lesssim \|\E\|_{C\cap L_{\infty}} \|f\|_{\mY},
 \qquad f\in \mY.
 $$
When $n=1$, we have for each $\delta>0$,
$$
\|\tN(\tS_{A}f)\|_{2}\lesssim \|\E\|_{C\cap L_{\infty}} \|f\|_{\mY_\delta}, \qquad 
f\in \mY_\delta.
$$
\end{rem}

%
%
%
%
%
\section{Almost everywhere non-tangential convergence}\label{sec:aecv}

 Since solutions are not defined in a pointwise sense,  the classical notion of  non-tangential convergence at a boundary point $x$ is replaced here by $$\lim_{r\to 1} |W^o(rx)|^{-1}\int_{W^o(rx)} h(\by) d\by \ \ \mathrm{exists},$$
which we call \emph{convergence of Whitney averages at $x$} because  the region $W^o(rx)$ is a Whitney ball. Note that since the Whitney balls at $x$ cover a truncated cone with vertex $x$, it really amounts to a non-tangential convergence. Besides, a slight modification of the proofs below yields limits of averages on Whitney regions $W^o({\bf z})$ for ${\bf z}$ in a fixed cone with vertex at  $x_{0}$,
as $|{\bf z}|\to 1$.  The exact choice of the Whitney balls does not matter.  

\begin{defn} 
Let $h$ be a function in $\bO^{1+n}$ with range in the bundle $\V$ in the sense that $h(rx) \in \V_{x}$ for all $r>0$ and $x\in S^n$. 
Let $x_{0}\in S^n$ and $1\le p<\infty$. We say that the {\em Whitney averages of $h$ converge  at $x_{0}$ in $L_{p}$ sense} to    $c\in \V_{x_{0}}$ if for any/some section $c_{x_{0}}\in C^\infty(S^n;\V)$ with $c_{x_{0}}(x_{0})=c$, 
$$\lim_{r\to 1} |W^o(rx_0)|^{-1}\int_{W^o(rx_0)} |h(\by)-c_{x_{0}}(y)|^p d\by=0.$$
Here $W^o(\bx)$ denotes a Whitney ball in $\bO^{1+n}$ centered at $\bx$.
We say that the {\em Whitney averages of $h$ converge in $L_{p}$ sense almost everywhere to $h_{0}$} with respect to surface measure if this happens with $c=h_{0}(x_{0})$ for almost every point $x_{0}\in S^n$.  
For functions with values in a trivial bundle, the sections $c_{x_{0}}$ are just constant functions. 

\end{defn}

Note that the limit does not depend on the choice of the section $c_{x_{0}}$, so  this  explains the ``any/some'' and it suffices to prove the  existence of the limit for one chosen section.
Clearly this notion entails convergence of Whitney averages. 

\begin{thm}\label{thm:aecv} 
Let $A$ be  coefficients with $\|\E\|_{C\cap L_{\infty}}<\infty$. Let $u$ be a $\mY^o$-solution to the divergence form equation with coefficients $A$ and let $u_{1}$ be the boundary trace of $u$ given by Corollary \ref{cor:diransatz}. Then  Whitney averages of $u$ converge in $L_{2}$ sense almost everywhere to $u_{1}$.
In particular, 
Whitney averages of $u$ converge almost everywhere to $u_{1}$.
\end{thm}

The result also holds for the $\R^{1+n}_{+}$ setup of \cite{AA1}, with almost identical proof.

\begin{proof} 
As in the proof of Theorem~\ref{thm:conj},  we can write 
$$
u(\bx)=e^{\sigma t}(e^{-t\tilde \Lambda}v_{0} + \st w_{t})_{\no}(x),
$$ 
where $\bx=e^{-t}x$, $\sigma= \frac{n-1}{2}$,   $v_{0}\in   L_{2}$  with 
$\|v_{0}\|_{2} \lesssim \|\nabla_\bx u\|_{\mY^o}+ \left|\int_{S^n} u_1 dx\right|$ and $u_{1}=(v_{0})_{\no}$. 

Let $p<2$ as in the third inequality of Corollary \ref{cor:resolventNT}. 
Let $x_{0}$ be a point on $S^n$, and let $B(x_{0},t)$ be the surface ball centered at $x_{0}$ 
with radius $t$.
Adapting the usual Lebesgue point argument for $p=1$, it is seen that for almost all points $x_{0}$
$$
\lim_{t\to 0} |B(x_{0}, t)|^{-1} \int_{B(x_{0},t)} |v_{0}(x)- v_{x_0}(x)|^p dx=0
$$
for any section $v_{x_0}\in C^\infty(S^n;\V)$ with   $v_{x_{0}}(x_{0})=v_{0}(x_{0})$ and one can further assume  $Dv_{x_0}=0$, which in particular implies that its
normal component is the constant scalar function $(v_{0}(x_{0}))_\no=u_{1}(x_{0})$. 
The key point is the identity
\begin{equation}
\label{eq:almostu}
u(\bx)-u_{1}(x_{0})= (e^{\sigma t} e^{-t\tilde \Lambda}(v_{0}-v_{x_0}))_{\no}(x) + e^{\sigma t}(\st w_{t})_{\no}(x),
\end{equation}
which follows since $\tD_0 v_{x_0}= -\sigma N v_{x_0}$, and hence $\tilde \Lambda v_{x_0}= \sigma v_{x_0}$ and $e^{\sigma t} e^{-t\tilde \Lambda}v_{x_0}=v_{x_{0}}$.

From Theorem \ref{thm:NTu},  $\| \tN(\chi_{t<\tau}\st w_\no)\|_2\to 0$ as $\tau\to 0$. 
Thus we can assume that the  Whitney averages of $\st w_{\no}$ converge to 0 in $L_2$ sense at $x_0$. It remains to show, with  $h_{x_0}:= v_{0}-v_{x_0}$,
$$
\lim_{t_{0}\to 0}|W(t_{0},x_{0})|^{-1}\int_{W(t_{0},x_{0})} |(e^{\sigma t} e^{-t\tilde \Lambda}h_{x_0})_{\no}(x)|^2 dtdx  = 0.
$$
As in \cite[Ch.~VII, Thm.~4]{Stein}, the rest of the argument consists in using the maximal estimates in Theorem~\ref{thm:NTtilde} with some adaptation. As we we do not have pointwise bounds on the operators that substitute the Poisson kernel we also have to handle more technicalities.
Let $0<c_{0}t_{0}<\tau$ with $t_{0},\tau<1$ to be chosen and  $c_{0}^{-1} t_{0}<t < c_{0} t_{0}$. 
In the $L_2$ average, write
$$
(e^{\sigma t}e^{-t\tilde \Lambda}h_{x_0})_{\no} = ((1+it\sigma)(I+it\tD_{0})^{-1}h_{x_0})_{\no} + (e^{\sigma t}e^{-t\tilde \Lambda}h_{x_0} - (1+it\sigma)(I+it\tD_{0})^{-1}h_{x_0})_{\no}.
$$
For the first term, we use \eqref{eq:decay}. 
Fixing $t$  and taking only the $L_{2}$ average in $x$, this gives us  a  bound  
$$
 \sum_{j\ge 2} 2^{-j} \left(|B(x_{0},2^j t)|^{-1}\int_{B(x_{0},2^j t)}  | h_{x_0}(x)|^p dx\right)^{1/p}. 
$$
This is controlled by 
$$
 M_{\tau}^p(h_{x_0})(x_{0}) + (t_{0}/\tau) M^p(h_{x_0})(x_{0}),
$$
where $M$ is the Hardy-Littlewood maximal operator over surface balls on $S^n$, $M^p(h):= M(|h|^p)^{1/p}$, and the subscript $\tau$ means that we restrict the maximal operator to balls  having radii less than $\tau$.  This control is obtained by truncating the sum at $2^j \approx \tau/t$ and using that $t\approx t_{0}$.  The average in $t$ now yields the same bound. 

For the second term, we note that 
$(e^{\sigma t}e^{-t\tilde \Lambda} - (1+it\sigma)(I+it\tD_{0})^{-1})v_{x_0}=0$. 
Thus we may replace $h_{x_0}$ by $v_{0}$ in this term, and write it
$$
(e^{\sigma t} \psi(t\tD_{0})v_{0})_{\no} + (e^{\sigma t} - (1+i\sigma t)^{-1}) ((I+it\tD_{0})^{-1}v_{0})_{\no}.
$$
with $\psi(\lambda):=e^{-|\lambda|} -(1+i\lambda)^{-1}$. 
The first term has estimates
$$
  \|\tN(\chi_{t<\tau}\psi(t\tD_{0})v_{0})\|_2^2 \lesssim \int_0^\tau \|\psi(t\tD_{0})v_{0}\|_2^2 \frac {dt}t\to 0,
  \qquad \tau\to 0,
$$
by Lemma~\ref{lem:XlocL2} and square function estimates.
Thus we can assume that Whitney averages of $(e^{\sigma t} \psi(t\tD_{0})v_{0})_{\no}$ 
converge to 0 in $L_2$ sense at $x_0$.
By Theorem \ref{thm:NTtilde}, the second is controlled by 
$$
\tau M^p(v_{0})(x_{0}).
$$
Thus it remains to show convergence to zero of
$$
M_{\tau}^p(h_{x_0})(x_{0}) + (t_{0}/\tau) M^p(h_{x_0})(x_{0}) + \tau M^p(v_{0})(x_{0}).
$$
Since $M^p(v_{0})\in L_{2}(S^n)$ as $p<2$, we can further assume for $x_{0}$ that
$M^p(v_{0})(x_{0})<\infty$.
For such fixed $x_0$ it follows that 
$M^p(h_{x_0})(x_{0}) \le M^p(v_{0})(x_{0}) + M^p(v_{x_0})(x_{0}) <\infty$. 
We now make $M_{\tau}^p(h_{x_0})(x_{0}) + \tau M^p(v_{0})(x_{0})$ small by 
choosing $\tau$ small. Then choose $t_0<\tau$ to make $(t_{0}/\tau) M^p(h_{x_0})(x_{0})$ small. All the constraints on $x_{0}$ are met almost everywhere and this completes the proof.
\end{proof}

\begin{rem} \label{rem:aecvv} 
The proof of almost everywhere convergence for averages applies to $v$ (with $\tN^p$, $p<2$, if $n\ge 2$). The starting point  is   
$$
e^{\sigma t} v_{t}(x)- v_{x_{0}}(x)=e^{\sigma t} e^{-t\tilde \Lambda}(v_{0}-v_{x_{0}})(x) + e^{\sigma t} \st w_{t}(x)
$$
replacing \eqref{eq:almostu} and the rest of the proof is as above. 
The only needed modification of the argument is that we now use \eqref{eq:decay0}
instead of  \eqref{eq:decay}. We obtain almost everywhere convergence of Whitney averages of $e^{\sigma t}v$ 
in $L_p$ sense to $v_{0}$ for $p<2$. Of course, the term $e^{\sigma t}$  can easily be removed in the end.
This factor was needed in order to have $e^{\sigma t}e^{-\sigma \st \Lambda} = I$ on $\nul(D)$.
\end{rem}

\begin{cor}\label{cor:aentcv} 
Assume that $A$ satisfies $\|\E\|_{C\cap L_{\infty}}<\infty$ and is such that all weak solutions $u$ to the divergence form equation with coefficients $A$, for some fixed constant $c>1$, 
satisfy the local boundedness property
$$
\sup_{\bx \in B}|u(\bx)| \le C \left( |cB|^{-1} \int_{cB} |u(\by)|^2 d\by\right)^{1/2},
$$
with a constant $C$ independent of $u$ and of closed balls $B$ with $cB\subset \bO^{1+n}$.  Then any 
$\mY^o$-solution to the divergence form equation with coefficients $A$ converges non-tangentially almost everywhere to its boundary trace. 
\end{cor}

The local boundedness property is a classical consequence of local H\"older regularity for weak solutions.  For real equations ($m=1$), the latter follows from \cite{Mo, DeG}. For   small complex $L_{\infty}$ perturbations of real equations, this is from \cite{Au2}.  
For two dimensional systems ($n=1$), local regularity follows immediately from  reverse H\"older inequalities described in Theorem \ref{thm:revholder} and Sobolev embeddings.  
For any dimension and system $(m\ge  1, n\ge 1)$, with continuous in $\overline{\bO^{1+n}}$ or $vmo$ coefficients, this is explicitly done in   \cite{AQ}.

\begin{proof} 
Applying the local boundedness property to  $u-u_{1}(x_{0})$ on Whitney balls  yields the desired convergence for almost every $x_{0}$ from Theorem~\ref{thm:aecv}. 
\end{proof}

We know describe new almost everywhere convergence results for $\mX^o$-solutions. 

\begin{thm}\label{thm:aecvReg/Neu} 
Let $A$ be  coefficients with $\|\E\|_{C\cap L_{\infty}}<\infty$. Let $g$ 
be an $\mX^o$-solution with potential $u$ to the divergence form equation with coefficients $A$. 
Then for any $p<2$,  Whitney averages of $g_\no= \pd_t u$, and of $(Ag)_\ta=(A\nabla_{\bx}u)_{\ta}$, converge in $L_{p}$ sense almost everywhere to $(g_{1})_\no$ and $(A_1 g_1)_\ta$ respectively,
where $g_{1}$ is the boundary trace of $g$ given by Theorem~\ref{thm:inteqforNeuandg}.

Furthermore, if we have pointwise ellipticity conditions on $A$, then the Whitney averages of $\nabla_{\bx}u$ and $\pd_{\nu_{A}}u$ converge in $L_{p}$ sense almost everywhere to $g_{1}$ and $(A_1 g_1)_\no$ respectively.

Finally, in all cases, Whitney averages of the potential $u$ converge almost everywhere in $L_{2}$ sense to $u_{1}$.
\end{thm}

Recall that pointwise ellipticity  holds  when $m=1$ (equations) or $n=1$ (two dimensional systems). If $A$ is continuous in $\overline{\bO^{1+n}}$, then  pointwise accretivity can be deduced from the strict accretivity  in the sense of \eqref{eq:accrasgarding}, for any $m,n$. See, \textit{e.g.} \cite{Fri}.
 We do not know if this convergence  of $\nabla_{\bx}u$ and $\pd_{\nu_{A}}u$ holds when $m\ge 2$ and $n\ge 2$ in general.  

\begin{proof}  We begin with the convergence for $u$. It is a straightforward consequence of the growth 
$\|v_{t}-v_{0}\|_{2}=O(t)$ for $t>0$ in Theorem~\ref{thm:conj} and $u(\bx)-u_{1}(x)=(e^{-\sigma t} v_{t}-v_{0})_{\no}(x)$. Let us turn to the gradient. 

By Theorem~\ref{thm:conj} 
we have
$f_t= e^{-t\Lambda} f_{0} + w_t$ for some $f_{0}\in \mH$ and $w\in \mY^*$.
From the correspondence between $g$ and $f$ in Proposition~\ref{prop:divformasODE}, it follows
that, modulo a rescaling, $(g)_\no \rad +(Ag)_\ta$ equals $Bf$.
Thus we need to prove convergence of Whitney averages of 
$$
   B_t f_t= e^{-t\tilde \Lambda} (B_0f_{0})+ (B_0 e^{-t\Lambda}- e^{-t\tilde \Lambda} B_0)f_0
-\E_t e^{-t\Lambda}f_0 + B_t w_{t}.
$$
It is clear that any $\mY^*$ element has  Whitney averages converging almost everywhere 
to $0$ in $L_2$ sense.
This applies to the last three terms. Indeed,
we have $\|w\|_{\mY^*}<\infty$, and hence $\|Bw\|_{\mY^*}<\infty$.  
Also $\|\E_t e^{-t\Lambda}f_0\|_{\mY^*}\lesssim \|\E\|_* \| e^{-t\Lambda}f_0\|_{\mX}<\infty$.
Furthermore, using $B_{0}(I+itDB_{0})^{-1}= (I+itB_{0}D)^{-1} B_{0}$, we write
\begin{multline*}
\big(B_{0}e^{-t\Lambda}- e^{-t\tilde\Lambda}B_{0}\big)f_{0}= B_{0}\big(e^{-t|DB_{0}+\sigma N|} - (I+it(DB_{0}+\sigma N))^{-1}\big)f_{0} \\+ B_{0}\big((I+it(DB_{0}+\sigma N))^{-1}- (I+itDB_{0})^{-1}\big)f_{0}\\
+\big((I+itB_{0}D)^{-1}-(I+it(B_{0}D-\sigma N))^{-1}\big)B_{0}f_{0} \\+ \big((I+it(B_{0}D-\sigma N))^{-1}- e^{-t|B_{0}D-\sigma N|}\big)B_{0}f_{0}.
\end{multline*}
Square function (that is, $\mY^*$) estimates hold for the first and fourth terms, whereas the second and third terms 
have $L_{2}$ norms bounded by $Ct$. Hence
$ \chi_{t<1} (B_{0}e^{-t\Lambda}- e^{-t\tilde\Lambda}B_{0}) \in \mY^*$.

For the  term $e^{-t\tilde \Lambda}  (B_{0}f_{0})$ we proceed as in the proof 
of Theorem~\ref{thm:aecv}, modified as in Remark~\ref{rem:aecvv}.

To complete  the proof, we now assume that $A$ is pointwise elliptic. Up to rescaling, we have to prove convergence of Whitney averages of  the conormal gradient $f$ of $u$. To see this, write  $f=B_{0}^{-1}(B_{0}f)$ using that $B_{0}$ is now invertible in $L_{\infty}(S^n; \mL(\V))$, seen as radial coefficients on $\bO^{1+n}$. Now the same argument as above replacing $B_{t}$ by $B_{0}$ shows that the Whitney averages of $B_{0}f$ converge  in $L_{p}$ sense to $B_{0}f_{0}$ almost everywhere for any $p<2$.
We claim that the notion of convergence in $L_{p}$-sense of Whitney averages is stable when $p<2$ under multiplication by bounded radially independent coefficients. Assume that $h$ has such a convergence property and let $M\in L_{\infty}(S^n; \mL(\V))$. Select smooth sections $h_{x_{0}}$ and $M_{x_{0}}$ with $h_{x_{0}}(x_{0})=h(x_{0})$ and $M_{x_{0}}(x_{0})=M(x_{0})$. Then take $L_{p}(W(t_{0},x_{0})$ average of 
 $$M(y)h(\by) - M_{x_{0}}(y)h_{x_{0}}(y)=(M(y)-M_{x_{0}}(y) )h(\by)+ M_{x_{0}}(y)(h(\by)-h_{x_{0}}(y))$$
 with $ \by= e^{-t}y \in W(t_{0},x_{0})$/  For the second term, one uses the assumption  on $h$ and that $M_{x_{0}}$ is bounded. For the first term, use H\"older inequality with exponents $1/p=1/r+1/q$ and $p<r<2$. The  exponent $q$ falls on  $M(y)-M_{x_{0}}(y)$ and Lebesgue convergence theorem applies (this is a further almost everywhere constraint on $x_{0}$). The exponent $r$ falls on $h$ which has uniform control by assumption. 
\end{proof}

%
%
%
%
%
\section{Fredholm theory for $(I-S_A)^{-1}$}       \label{sec:fredholm}

We saw in Section~\ref{sec:representation} that the invertibility of $I-S_{A}$ on $\mX$ (resp. $\mY$) allows to represent $\mX^o$  (resp. $\mY^o$) solutions through Cauchy type extensions
$$
  f= (I-S_A)^{-1} e^{-t\Lambda}E_0^+ f_0
$$
(resp. $f= (I-S_A)^{-1} De^{-t\tilde \Lambda}\tE_0^+ v_0)$).
Working in the space $\mX$ or $\mY$, it is clear from Theorem~\ref{thm:XYbounded} that $I-S_A$ is invertible 
provided $\|\E\|_*$ is small enough. 
In this section, we use Fredholm operator theory to relax this condition and 
show that it suffices to assume this smallness only near the boundary $t=0$. Our discussion  in this section is limited to the specific but relevant case where $\sigma=\frac{n-1}{2}$.

\begin{thm}   \label{thm:SAFredholm} 
  Assume that $\|\E\|_*<\infty$, so that $S_A$ is bounded on $\mX$ and $\mY$.  There exists $\epsilon>0$ such that if $\E$ satisfies the small Carleson condition
\begin{equation}  \label{eq:smallCarls}
\lim_{\tau\to 0} \|\chi_{t<\tau}\E\|_*<\epsilon,
\end{equation}
  then $I-S_A$ is invertible on $\mX$ and $\mY$.
\end{thm}

We remark that (\ref{eq:smallCarls}) is equivalent to the small Carleson condition 
(\ref{eq:smalllimCarleson}).
The proof of Theorem~\ref{thm:SAFredholm} requires the following lemmas.

\begin{lem}   \label{lem:injonX}
   Assume $\|\E\|_*<\infty$. 
   Then $I-S_A$ is  injective on $\mX$.
\end{lem}

\begin{proof}
Assume that $f\in\mX$ satisfies $f=S_Af$.
Lemma~\ref{lem:inttodiff} shows that $f$ has trace $h^-\in E_0^- \mH$. As $\mX\subset L_2(\R_+;L_2)$ and $f$ is valued in $\mH$, we have $f\in L_2(\R_+;\mH)$. Extend $f$ to  $f^1\in L_2(\R;\mH)$, letting
$$
  f^1_t:=
  \begin{cases}
    f_t, & t>0, \\
    e^{t\Lambda} h^-, & t\le 0.
  \end{cases}
$$
To verify that $f^1$ satisfies $\pd_t f^1+ (DB^1+\sigma N) f^1=0$ in $\R\times S^n$ distributional sense,
where $B^1_t:=B_t$ for $t>0$ and $B^1_t= B_0$ for $t\le 0$,
consider a test function $\phi\in C_0^\infty(\R\times S^n;\C^{(1+n)m})$ and let
$\xi_\epsilon(t):= 1- \eta^0(|t|/\epsilon)$, where $\eta^0$ is the function from Lemma~\ref{lem:preints}.
Then
\begin{multline*}
  \int_\R ((-\pd_t+ (B^1)^* D+\sigma N)\phi, f^1) dt \\
  = \int_\R \Big( ((-\pd_t+ (B^1)^* D+\sigma N)((1-\xi_\epsilon) \phi), f^1) 
 + ((-\pd_t+ (B^1)^* D+\sigma N)(\xi_\epsilon \phi), f^1) \Big) dt \\
 = 0+  \int_\R \xi_\epsilon((-\pd_t+ (B^1)^* D+\sigma N)\phi, f^1) dt 
 + \epsilon^{-1}\int_\epsilon^{2\epsilon} (\phi_t, f^1_t) dt- \epsilon^{-1}\int_{-2\epsilon}^{-\epsilon} (\phi_t, f^1_t) dt \\
 \to 0 + (\phi_0, h^-)- (\phi_0, h^-)=0,
\end{multline*}
with $\phi_0(x):= \phi(0,x)$,
using that the equation holds both in $\R_+$ and $\R_-$.
Hence $\pd_t f^1+ (DB^1+\sigma N) f^1=0$ in all $\R\times S^n$.
Since $\sigma=\frac{n-1}{2}$, extending Proposition~\ref{prop:divformasODE} from $\bO^{1+n}$ to all $\R^{1+n}$ (see Remark \ref{rem:extension}), we see that
$f^1$ corresponds to a function $g^1\in L_2(\R^{1+n};\C^{(1+n)m})$ solving $\divv_\bx(A^1 g^1)=0,\curl_\bx g^1=0$
in all $\R^{1+n}$, with $A^1$ corresponding to $B^1$. To verify that this forces $g^1$, and therefore $f^1$ and $f$, to vanish, note that for any fixed $R>0$ we can find 
$u$ such that $g^1=\nabla_\bx u$, where 
$\int_{|\bx|<2R} |u|^2 d\bx\lesssim R^2\int_{|\bx|<2R}|g^1|^2 d\bx$
by Poincar\'e's inequality and the implicit constant is independent of $R$. 
Take a 
test function $\eta\in C_0^\infty(|\bx|<2R)$ with $\eta=1$ on $|\bx|<R$ with $|\nabla_\bx \eta|\lesssim R^{-1}$,
and use that $\divv_\bx (A^1g^1)=0$ in the distributional sense to get
\begin{multline*} 
  \int_{|\bx|<R}|g^1|^2 d\bx \lesssim  \re\int (A^1g^1,\nabla_\bx u) \eta d\bx 
  =  -\re\int (A^1g^1,\nabla_\bx \eta) u d\bx \\
  \lesssim \left(\int_{R<|\bx|<2R} |g^1|^2 d\bx\right)^{1/2}
  \left( \int_{|\bx|<2R}|g^1|^2 d\bx \right)^{1/2}
  \lesssim \left(\int_{R<|\bx|<2R} |g^1|^2 d\bx\right)^{1/2} \|g^1\|_2.
\end{multline*}
Letting $R\to \infty$ this shows that $g^1=0$, which proves the lemma.
\end{proof}

\begin{lem}    \label{lem:L2lowerbound}
  Assume $\|\E\|_*<\infty$ and fix $\tau>0$.
  Then there are lower bounds 
$$
  \|f\|_{L_2(\tau,\infty;\mH)} \lesssim \|(I-S_A) f\|_{L_2(\tau/2,\infty; \mH)},
$$
where the implicit constant depends on $\tau$,
for all $f\in L_2(\R_+;\mH)$ such that $f_t=0$ for $t<\tau$.
\end{lem}

\begin{proof}
  By Lemma~\ref{lem:inttodiff}, $f$ and $f^0:= (I-S_A)f$ satisfy
 $(\pd_t + DB_0 +\sigma N)f^0= (\pd_t + DB +\sigma N)f$.
 As in Proposition~\ref{prop:divformasODE} combined with Proposition~\ref{prop:CR},
 this can be translated to
$$
\begin{cases}
  \divv_\bx(A_1 g^0)= \divv_\bx(Ag),\\
  \curl_\bx g^0 = \curl_\bx g,
\end{cases}
$$
in $\bO^{1+n}$ distributional sense, where
$g^0_r= r^{-(n+1)/2}( (B_0 f^0_t)_\no \rad+  (f^0_t)\ta)$ and $g_r= r^{-(n+1)/2}( (B f_t)_\no \rad+  (f_t)\ta)$.
Write $\bO^{1+n}_\tau:= \{|\bx |<e^{-\tau}\}$, so that $\bO^{1+n}_\tau\subset \bO^{1+n}_{\tau/2}$.
In particular, the last equation implies that there is a potential $u:\bO_{\tau/2}^{1+n}\to \C^m$
such that
$$
  g-g^0= \nabla_\bx u \qquad \text{in } \bO^{1+n}_{\tau/2},
$$
and we may choose $u$ so that $\|u\|_{L_2(\bO^{1+n}_{\tau/2})}\lesssim \|g-g^0\|_{L_2(\bO^{1+n}_{\tau/2})}$.
Fix $\eta\in C_0^\infty(\bO^{1+n})$ such that $\eta|_{\bO^{1+n}_\tau}=1$ and $\supp\eta\subset\bO^{1+n}_{\tau/2}$.
Using the first equation and $\supp g\subset \bO^{1+n}_\tau$ gives
\begin{multline*}
  \re \int (Ag,g-g^0) d\bx 
  = \re \int (Ag,\nabla_\bx(\eta u)) d\bx 
  =\re \int (Ag^0,\nabla_\bx(\eta u)) d\bx  \\
  = \re \int_{\bO^{1+n}_{\tau/2}} \Big(A_1g^0,\eta (g-g^0)+ (\nabla_\bx\eta)u\Big) d\bx 
  \lesssim  \|g^0\|_{L_2(\bO^{1+n}_{\tau/2})}  \|g-g^0\|_{L_2(\bO^{1+n}_{\tau/2})}.
\end{multline*}
Note that $(g_r)_\ta= r^{-(n+1)/2}(f_t)_\ta\in\ran(\nabla_S)$, so that $g_r\in \mH_1$.
The accretivity (\ref{eq:accrassumption}) of $A_r$, for each fixed $r\in(0,1)$, and integration for 
$0<r<e^{-\tau}$ imply that 
\begin{multline*}
  \|g\|^2_{L_2(\bO^{1+n}_{\tau})}\lesssim \re \int_{\bO^{1+n}_{\tau}} (Ag,g) d\bx
    \le \re \int_{\bO^{1+n}_{\tau}} (Ag,g-g^0) d\bx \\
    + \|g\|_{L_2(\bO^{1+n}_{\tau})}\|g^0\|_{L_2(\bO^{1+n}_{\tau})}\lesssim 
  \|g\|_{L_2(\bO^{1+n}_{\tau})}\|g^0\|_{L_2(\bO^{1+n}_{\tau/2})}
  + \|g^0\|^2_{L_2(\bO^{1+n}_{\tau/2})},
\end{multline*}
and hence that $\|g\|_{L_2(\bO^{1+n}_{\tau})}\lesssim \|g^0\|_{L_2(\bO^{1+n}_{\tau/2})}$.
By the isomorphism (\ref{eq:dfl2iso}), this translates to 
$\|f\|_{L_2(\tau,\infty; \mH)}\lesssim \|f^0\|_{L_2(\tau/2,\infty;\mH)}$
and proves the lemma.
\end{proof}

\begin{lem}    \label{lem:compactcomm}
Assume $\|\E\|_*<\infty$.
Let $\eta:\R_+\to \R$ be a Lipschitz function, that is $|\eta(t)-\eta(s)|\le C|t-s|$ for all $t,s>0$.
Then the commutator
$$
  [\eta, S_A]= \eta S_A-S_A \eta
$$
is a compact operator on $L_2(\R_+, dt;L_2)$.
\end{lem}

\begin{proof} 
Write $S_A= \hS_A-\sigma \vS_A$ as in Theorem~\ref{thm:XYbounded}.
Since $\vS_A=\Lambda^{-1}\hS_A$, except that $\hE_0^\pm$ are replaced by $\vE_0^\pm$,
it is enough to show compactness of $[\eta_0, \hS_A]$.
It suffices to verify that
\begin{equation}    \label{eq:cptcommutator}
 F(\Lambda): f_t\mapsto \int_0^t (\eta(t)-\eta(s))\Lambda e^{-(t-s)\Lambda} f_s ds,\qquad
\end{equation}
is a compact operator on $L_2(\R_+, dt;\mH)$.
(The proof below only depends on the fact that $\Lambda$ has compact resolvents.)
Indeed, by duality this implies that also
$f_t\mapsto \int_t^\infty (\eta(t)-\eta(s))\Lambda e^{-(s-t)\Lambda} f_s ds$
is compact, upon changing $\Lambda$ to $\Lambda^*$. 
Since $\hE^\pm_0 \E$ are bounded $L_2(\R_+;L_2)\to L_2(\R_+;\mH)$ and commute with $\eta$, 
we conclude that $[\eta, \hS_A]$ is compact.

Consider the symbol
$$
   F(\lambda) : f_t \mapsto  \int_0^t (\eta(t)-\eta(s))\lambda e^{-(t-s)\lambda} f_s ds.
$$
To estimate the norm of this integral operator, acting in $L_2(\R_+;\C)$ for fixed $\lambda\in S^o_{\nu,\sigma+}$,
we apply Schur estimates as in \cite[Lem. 6.6]{AA1}.
We need to estimate
$$
  \sup_{t>0} \int_0^t |(\eta(t)-\eta(s))\lambda e^{-(t-s)\lambda}| ds +
  \sup_{s>0} \int_s^\infty |(\eta(t)-\eta(s))\lambda e^{-(t-s)\lambda}| dt.
$$
Using Lipschitz regularity, the first integral has estimate
$$
  \int_0^t (t-s)\lambda_1 e^{-(t-s)\lambda_1} dt= \lambda_1^{-1} \int_0^{t\lambda_1} x e^{-x} dx\lesssim \lambda^{-1},
$$
where $\lambda_1:= \re\lambda\approx |\lambda|$ for $\lambda\in S^o_{\nu,\sigma+}$,
and a similar estimate for the second integral gives the bound
$$
  \|F(\lambda)\|_{L_2(\R_+;\C)\to L_2(\R_+;\C)}\lesssim \lambda^{-1}.
$$
It is also clear that $F(\lambda)$ defines a compact operator on $L_2(\R_+;\C)$ (for example
truncate the kernel and show from the Schur estimates that $F(\lambda)$ is a uniform limit of Hilbert-Schmidt operators).

Consider now the Dunford integral
$$
  F(\Lambda) = \frac 1{2\pi i}\int_{\partial S_{\theta,\sigma+}} F(\lambda) (\lambda- \Lambda)^{-1} d\lambda,\qquad
  \omega<\theta<\nu.
$$
From the compactness of $F(\lambda):L_2(\R_+;\C)\to L_2(\R_+;\C)$, and of $(\lambda-\Lambda)^{-1}:\mH\to \mH$
by Proposition~\ref{prop:DBprops}, we deduce compactess of
$F(\lambda)(\lambda-\Lambda)^{-1}: L_2(\R_+;\mH)\to L_2(\R_+;\mH)$ (for example by approximating 
$(\lambda-\Lambda)^{-1}$ uniformly by finite rank operators).
Since $\|F(\lambda)(\lambda-\Lambda)^{-1}\|\lesssim \lambda^{-2}$, the Dunford integral converges in norm,
at least when $\sigma>0$, and we conclude that $F(\Lambda)$ is a compact operator on 
$L_2(\R_+;\mH)$ (for example, approximate with Riemann sums, using norm continuity of
$\lambda\mapsto F(\lambda)(\lambda-\Lambda)^{-1}$). 
In dimension $n=1$, \textit{i.e.} $\sigma=0$, note that $\lambda=0$ does not belong to the
spectrum of $D_0$ on $\mH$. Hence it is not needed to integrate through $\lambda=0$ in the Dunford
integral, in which case the Dunford integral converges in norm also here.
This proves the lemma.
\end{proof}

\begin{lem}    \label{lem:offdiagcomm}
  Assume $\|\E\|_*<\infty$.
  Let $0<a<b<\infty$ and write $\chi_0:=\chi_{(0,a)}$ and $\chi_\infty:=\chi_{(b,\infty)}$ for the characteristic functions of these intervals.
Then 
$$
  \chi_0 S_A \chi_\infty:\mX\to \mX,\qquad \text{and}\qquad \chi_\infty S_A\chi_0: \mY\to \mY
$$
are compact operators.
\end{lem}

\begin{proof}
As in the proof of Lemma~\ref{lem:compactcomm}, we may replace $S_A$ by $\hS_A$
as straightforward modifications of the proof below give the result for $\vS_A$.
(i)
  We claim that the integral operator
$$
  F(\lambda) f_t := \int_0^a\lambda e^{-(t-s)\lambda} f_s ds
$$
is a Hilbert--Schmidt (hence compact) operator $F(\lambda): L_2(0,a; sds)\to L_2(b,\infty; dt)$.
Indeed, a straightforward calculation shows that
$$
  \int_b^\infty\int_0^a | \lambda e^{-(t-s)\lambda}|^2 sds dt\le \tfrac a4e^{-2(b-a)\lambda}.
$$

As in the proof of Lemma~\ref{lem:compactcomm}, it follows by operational calculus that
$$
  L_2(0,a; s ds;\mH)\to L_2(b,\infty; \mH): f_t \mapsto  \int_0^a\Lambda e^{-(t-s)\Lambda} f_s ds
$$
is compact.
Since $\hE_0^- \E$ is bounded on $L_2(0,a; s ds;\mH)$, this proves that 
$\chi_\infty\hS_A\chi_0: \mY\to \mY$ is compact.

(ii)
To prove that $\chi_0 \hS_A \chi_\infty:\mX\to \mX$ is compact, it suffices to show that 
\begin{equation}   \label{eq:offdiagcompact}
  L_2(b,\infty;\mH)\to \mX: f_t\mapsto \chi_0(t) \int_b^\infty \Lambda e^{-(s-t)\Lambda}  f_s ds
\end{equation}
is compact, since $\hE_0^- \E$ is bounded on $L_2(b,\infty;\mH)$.
To prove this, we write for $t<a$, 
\begin{multline*}
    \int_b^\infty \Lambda e^{-(s-t)\Lambda}  f_s ds 
    =  \int_b^\infty \Lambda e^{-(s+t)\Lambda} f_s ds+
     \int_b^\infty (I- e^{-2t\Lambda}) \Lambda e^{-(s-t)\Lambda} f_s ds \\
   = e^{-t\Lambda} e^{-\delta \Lambda} \int_b^\infty \Lambda e^{-(s-\delta)\Lambda} f_s ds \\
   + \left( \sqrt t e^{-(a-t)\Lambda}\frac{I- e^{-2t\Lambda}}{\sqrt{t\Lambda}}  \right)  e^{-\delta\Lambda} \int_b^\infty \Lambda^{3/2} 
   e^{-(s-a-\delta)\Lambda} f_s ds = I+II,
\end{multline*}
where $\delta>0$ is small enough.
Cauchy--Schwarz' inequality shows that the integral expressions in both $I$ and $II$ define bounded operators
$L_2(b,\infty;\mH)\to \mH$, whereas $e^{-\delta\Lambda}= D_0^{-1}(D_0 e^{-\delta|D_0|})$ is compact 
on $\mH$ by Proposition~\ref{prop:DBprops}.
For $I$, the factor $e^{-t\Lambda}:\mH\to \mX$ is bounded by Theorem~\ref{thm:NT}.
Since $\mY^*\subset \mX$, boundedness of the first factor in $II$ follows from
boundedness of $\sqrt t e^{-(a-t)\Lambda}$ for $t\in(0,a)$, and square function estimates for $\Lambda$
since $\psi(\lambda)= (1-e^{-2\lambda})/\sqrt\lambda\in \Psi(S^o_{\nu+})$.
This completes the proof.
\end{proof}

\begin{proof}[Proof of Theorem~\ref{thm:SAFredholm}]
(i)
Consider first invertibility in the space $\mX$.
By Theorem~\ref{thm:XYbounded}, we have $\|S_A\|_{\mX\to\mX}\lesssim \|\E\|_*$,
for any perturbation of coefficients $\E$.
Thus, for any $\tau>0$
$$
  \|S_A f\|_\mX\le C\|\chi_{t<\tau}\E\|_* \|f\|_\mX,\qquad\text{whenever } f_t=0\text{ for } t>\tau,
$$
with $C$ independent of $\tau$.
This follows upon writing $\E f= (\chi_{t<\tau}\E)f$.
Under the hypothesis, we can choose $\tau>0$ such that
$C\|\chi_{t<\tau}\E\|_*\le 1/2$.
We obtain
$$
  \|(I-S_A)f\|_\mX\ge \|f\|_\mX- \tfrac 12\|f\|_\mX= \tfrac 12\|f\|_\mX,\qquad\text{whenever } f_t=0\text{ for } t>\tau.
$$

Next consider an arbitrary $f\in\mX$. 
Pick $\eta_0\in C^\infty(\R_+)$ such $\supp\eta_0\subset [0,\tau]$ and $\eta_0=1$ for $t<\tau/2$.
Write $\eta_1:= 1-\eta_0$.
Then $\|(I-S_A)(\eta_0f)\|_\mX\ge \tfrac 12\|\eta_0 f\|_\mX$, and Lemma~\ref{lem:L2lowerbound} shows that 
$\| (I-S_A)(\eta_1 f)\|_\mX\gtrsim \|\eta_1 f\|_\mX$. 
This gives
\begin{multline*}
  \|f\|_\mX \le \|\eta_0 f\|_\mX + \|\eta_1 f\|_\mX 
  \lesssim \|(I-S_A)(\eta_0 f)\|_\mX + \| (I-S_A)(\eta_1 f)\|_\mX \\
    \le \|\eta_0(I-S_A) f\|_\mX + \|[\eta_0, S_A]f\|_\mX + \|\eta_1 (I-S_A)f\|_\mX + \|[\eta_1, S_A]f\|_\mX \\
    \lesssim \|(I-S_A)f\|_\mX + \|[\eta_0, S_A]f\|_\mX.
\end{multline*}
To show that $[\eta_0,S_A]:\mX\to\mX$ is compact, we write
$$
   [\eta_0, S_A]=  \chi_0  [\eta_0, S_A]+ (1- \chi_0) [\eta_0, S_A]
   = \chi_0 S_A (1-\eta_0) + (1-\chi_0)  [\eta_0, S_A],
$$
where $\chi_0:= \chi_{(0, \tau/4)}$.
Hence, compactness of the first term is granted from Lemma~\ref{lem:offdiagcomm}. Next,  as  the $\mX$ and $L_{2}$ norms are the same away from the boundary, 
Lemma~\ref{lem:compactcomm}
implies that the  second term is compact from $\mX\to \mX$.  This shows that $I-S_A:\mX\to\mX$ is a semi-Fredholm operator. 
 
 To see that it is a Fredholm operator with index 0, note that the lower estimate on $I-S_A$ 
 above goes through with $\E$ replaced by 
 $\alpha \E$, $\alpha\in[0,1]$. Apply the method of continuity. 
 Since $I-S_A$ is injective on $\mX$ by Lemma~\ref{lem:injonX}, it follows that
 it is invertible.
 
(ii)
Consider now invertibility in the space $\mY$.
That $I-S_A:\mY\to \mY$ is a Fredholm operator with index $0$ follows as in (i), provided
we show that $[\eta_0, S_A]:\mY\to \mY$ is compact.
Here we write
$$
   [\eta_0, S_A]=   [\eta_0, S_A]\chi_0+ [\eta_0, S_A](1-\chi_0)
   = (\eta_0-1) S_A \chi_0 + [\eta_0, S_A](1-\chi_0),
$$
and  Lemmas~\ref{lem:offdiagcomm} and \ref{lem:compactcomm}
are applied in the same way.

To verify bijectivity, we note that $\mX\subset\mY$ is a dense continuous inclusion,
where $I-S_A:\mX\to\mX$ is an isomorphism.
This implies that $I-S_A:\mY\to \mY$ has dense range, hence is an isomorphism since its index is 0.
\end{proof}

%
%
%
%
%
\section{Solvability of BVPs}     \label{sec:BVPs}

\subsection{Characterization of well-posedness}\label{sec:char}

For $A$ such that $I-S_{A}$ is invertible, we introduce boundary maps and 
characterize well-posedness in terms of their invertibility.

\begin{defn}
  For coefficients $A$ such that $\|\E\|_{*}<\infty$ and
  $I-S_A: \mX\to \mX$ is invertible, define the {\em perturbed
  Hardy projection}
$$
  E_A^+h := E_0^+h - E_0^-\int_0^\infty e^{-s\Lambda} D \E_s f_s ds, \qquad h\in L_{2}(S^n;\V),
$$
  where $f:= (I-S_A)^{-1} e^{-t \Lambda} E_0^+ h$.
  Write $E_A^-:= I-E_A^+$. Here, $E_0^\pm$ denote the Hardy projections associated to the corresponding  radially independent coefficients $A_{1}$.
\end{defn}

\begin{prop}\label{prop:EA}  The operators $E_A^\pm : L_{2}(S^n;\V)\to L_{2}(S^n;\V)$ are bounded projections and  the range $E_A^+ \mH\subset \mH$ consists of all traces
$f_0$ of conormal gradients $f$  of $\mX^o$-solutions  to the divergence
form equation with coefficients $A$ in $\bO^{1+n}$.
\end{prop}

\begin{proof} That $E_A^\pm$ are bounded follows from their construction. The projection property $(E_A^\pm)^2= E_A^\pm$ follows from $E_0^+E_0^-=0$. Next,  the statement about the range follows from Theorem~\ref{thm:inteqforNeu}. \end{proof}

\begin{defn}   \label{defn:perttildeHardy}
  For coefficients $A$ such that $\|\E\|_{*}<\infty$ and 
  $I-S_A: \mY\to \mY$ is invertible, define the {\em perturbed
  Hardy projection}
$$
  \tE_A^+ \tilde h := \tE_0^+\tilde h - \tE_0^-\int_0^\infty e^{-s\tilde\Lambda} \E_s f_s ds,
  \qquad \tilde h\in L_2(S^n;\V),
$$
  where $f:= (I-S_A)^{-1}D e^{-t \tilde\Lambda} \tE_0^+ \tilde h$.
  Write $\tE_A^-:= I-\tE_A^+$. Here, $\tE_0^\pm$ denote the Hardy projections associated to the corresponding  radially independent coefficients $A_{1}$. 
\end{defn}

\begin{prop}\label{prop:tEA} The operators $\tE_A^\pm : L_{2}(S^n;\V)\to L_{2}(S^n;\V)$ are bounded projections and  $\sett{(\tE_A^+ \tilde h^+)_\no}{\tilde h^+\in \tE_0^+L_2}$ consists of all traces
 of $\mY^o$-solutions  to the divergence
form equation with coefficients $A$ in $\bO^{1+n}$.
\end{prop}

\begin{proof} That $\tE_A^\pm$ are bounded follows from their construction. The projection property $(\tE_A^\pm)^2= \tE_A^\pm$ follows from $\tE_0^+\tE_0^-=0$. Next,  the statement about the trace space follows from Corollary~\ref{cor:diransatz}(ii). \end{proof}

We remark that, unlike the case of $r$-independent coefficients, the complementary projections $E_{A}^-$ and 
$\tE_A^-$ are in general not related to solutions of a divergence form equation in 
the complementary domain $\R^{1+n}\setminus \bO^{1+n}$.

\begin{prop}    \label{prop:wpequiviso}
  For coefficients $A$ such that 
 $I-S_{A}$ is invertible on $\mX$ for (i) and (ii), or $I-S_{A}$ is invertible on $\mY$ for (iii),
   the following hold.
  
(i)
  The Neumann problem (with coefficients $A$) is well-posed in the sense of 
  Definition~\ref{defn:wpbvp} if and only if
\begin{equation}
\label{eq:Neumap}  
  E_0^+\mH\to \mH_\no: h^+\mapsto (E_A^+ h^+)_\no
\end{equation}
  is an isomorphism.

(ii)
  The regularity problem (with coefficients $A$) is well-posed in the sense of 
  Definition~\ref{defn:wpbvp} if and only if
  \begin{equation}
\label{eq:Regmap}
E_0^+\mH\to \mH_\ta: h^+\mapsto (E_A^+ h^+)_\ta
\end{equation}
  is an isomorphism.

(iii)
  The Dirichlet problem (with coefficients $A$) is well-posed in the sense of 
  Definition~\ref{defn:wpbvp} if and only if
    \begin{equation}
\label{eq:Dirmap}
\tE_0^+L_2(S^n;\V) \to L_2(S^n;\C^m): \tilde h^+\mapsto (\tE_A^+ \tilde h^+)_\no
\end{equation}
  is an isomorphism.
\end{prop}

\begin{proof}
(i)
The ansatz \eqref{eq:ansatzRN} in Theorem~\ref{thm:inteqforNeu} gives is a one-to-one correspondence
between $h^+ \in E_0^+\mH$ and conormal gradients 
$f= (I-S_A)^{-1} e^{-t\Lambda} h^+$  of  
$\mX^o$-solutions to the divergence form equation. Moreover, $f_{0}=E_{A}^+h^+$ by  Proposition~\ref{prop:EA}. 
Under this correspondence, invertibility of $h^+\mapsto (E_A^+ h^+)_\no$ translates 
to well-posedness of the Neumann problem.
The proof of (ii) is similar. 

(iii)
The ansatz \eqref{eq:ansatzDu} from Corollary~\ref{cor:diransatz}(iii) gives a one-to-one 
correspondence between $\tilde h^+\in  \tE_0^+L_2$ and
 $\mY^o$-solutions $u$ to the divergence form equation. Moreover, $(\tE_A^+ \tilde h^+)_\no=u_{1}$ by Proposition~\ref{prop:tEA}. 
Under this correspondence, invertibility of $\tilde h^+\mapsto (\tE_A^+ \tilde h^+)_\no$ translates 
to well-posedness of the Dirichlet problem.
\end{proof}

\subsection{Equivalence between Dirichlet and Regularity problems} \label{sec:equiv}

We show that the Dirichlet and regularity problems are the same up to taking adjoints.

\begin{prop}    \label{prop:reg=dir}
  Assume that $A$ are coefficients such that $I-S_A$ is invertible on $\mX$ and
  $I-S_{A^*}$ is invertible on $\mY$. 
  Then the regularity problem with coefficients $A$ is well-posed if and only if the 
  Dirichlet problem with coefficients $A^*$ is well-posed.
\end{prop}

It is not clear to us whether invertibility of $I-S_{A}$ on $\mX$ implies or is implied by invertibility of $I-S_{A^*}$ on $\mY$. Thus we assume both. 
We need three lemmas, the first being useful reformulations of invertibility of the Dirichlet boundary map, the second  an identity between Hardy projections and the third an abstract principle.

\begin{lem}  \label{lem:dir}
The maps
$$
    \tE_0^+L_2(S^n;\V) \to L_2(S^n;\C^m): \tilde h^+\mapsto (\tE_A^+ \tilde h^+)_\no
$$
 and
$$
    \tE_0^+(L_2(S^n;\V)/\mH^\perp) \to L_2(S^n;\C^m)/\C^m: \tilde h^+\mapsto (\tE_{A}^+\tilde h^+)_\no
$$
are simultaneous isomorphisms.
\end{lem}

\begin{proof} 
This amounts to mod out $\mH^\perp$.  
We recall that  $\mH^\perp$ is preserved by $\tilde \Lambda$ and $\tE_{0}^\pm$, and annihilated by $D$, so from the definition $\tE_{A}^+\tilde h^+=\tE_{0}^+\tilde h^+\in \mH^\perp$ for $\tilde h^+\in \mH^\perp$.  
 By Lemma~\ref{lem:spectralsplits}, $(\tE_{0}^+\tilde h^+)_{\no}  = (\tilde h^+)_\no$ for $\tilde h^+\in \mH^\perp$,
so $\tE_0^+(L_2(S^n;\V)/\mH^\perp) \to L_2(S^n;\C^m)/\C^m: \tilde h^+ \mapsto (\tE_{A}^+\tilde h^+)_{\no} $ is a well defined map.
That the two maps simultaneously are isomorphisms can now be verified from 
$\{(\tE_{A}^+\tilde h^+)_{\no} ; \tilde h^+ \in \mH^\perp\}=\C^m$. 
\end{proof}

\begin{lem} On $L_{2}(S^n;\V)$ we have the duality relation \begin{equation}   \label{eq:dualityregtodir}
  (E_A^-)^*= N \tE_{A^*}^+ N.
\end{equation}
\end{lem}

\begin{proof} 
The proof of this duality  builds on the 
formula
$$
  (D_{A_1})^* = -N \tD_{A_1^*} N
$$
on $L_{2}(S^n;\V)$
from Lemma~\ref{lem:intertwduality}  with $A_{1}$ equal to the boundary trace of $A$ and where we  used the notation at the end of Definition~\ref{defn:operators}. 
Using this observation and short hand notation $E_{0}^\pm=E^\pm_{A_{1}}$, $\Lambda=|D_{A_{1}}|$, 
$\tE_{0}^\pm=\tE^\pm_{A_{1}^*}$ and $\st \Lambda=|\tD_{A_1^*}|$,  it follows that we have
$$
  (E^\pm_{0})^*= N\tE^\mp_{0} N,\qquad \Lambda^*= N\tilde \Lambda N.
  $$
  Note that when $n=1$, these identities can be also checked from the extensions of the projections  in Definition~\ref{defn:hardyprojs1D}.
  This implies that 
$$
  \int_0^\infty (N \tilde f_t, \E_t (S_A f)_t) dt
  = \int_0^\infty (N (S_{A^*} \tilde f)_s,\E_s f_s ) ds, \qquad \tilde f\in \mY, f\in\mX,
$$
which follows from Fubini's theorem and  the formula defining $S_A^\epsilon$ from
Lemma~\ref{lem:SAdecomp}, and then letting $\epsilon\to 0$ using boundedness on $\mX$ and $\mY$. Details are left to the reader.
Note that $S_{A^*}$ is defined using the coefficients $ \widetilde \E_t := \widehat{A_1^*}-\widehat{A^*}$, while $\E_{t}=\widehat{A_1}-\widehat{A}$. This duality relation between $S_{A}$ and $S_{A^*}$ clearly extends to their resolvents.

For $h,\tilde h\in L_2$, using the isomorphism assumption on $I-S_{A}$ and $I-S_{A^*}$,  we let $f= (I-S_{A})^{-1} e^{-t \Lambda} E_0^+h\in \mX$ and $\tilde f:= (I-S_{A^*})^{-1} De^{-s \tilde\Lambda} \tE_0^+\tilde h\in \mY$ and
  calculate
\begin{multline*}
  (N \tilde h, E^+_A h)= (N \tilde h, E_0^+ h)-
  \int_0^\infty (N\tilde h, E_0^-e^{-s\Lambda} D\E_s f_s) ds \\
  =   (N \tE_0^- \tilde h, h)+
  \int_0^\infty (N De^{-s\tilde\Lambda} \tE_0^+ \tilde h, \E_s ((I-S_A)^{-1} e^{-t\Lambda} E_0^+h)_{s}) ds \\
  =   (N \tE_0^- \tilde h, h)+
  \int_0^\infty (N ((I-S_{A^*})^{-1} De^{-s\tilde\Lambda} \tE_0^+ \tilde h)_{t}, \E_t  e^{-t\Lambda} E_0^+ h) ds \\
  =   (N \tE_0^- \tilde h, h)+
  \int_0^\infty (N \tE_0^- e^{-t\tilde\Lambda} \tilde \E_t \tilde f_t,   h) dt
 = (N \tE_{A^*}^-\tilde h, h).
\end{multline*}
This completes the proof.
\end{proof}

\begin{lem}\label{lem:abstract} Assume that $N^\pm$ and $E^\pm$ are two pairs of complementary projections
in a Hilbert space $\mH$, \textit{i.e.} $(N^\pm)^2=N^\pm$ and $N^++N^-=I$, and similarly for $E^\pm$.
Then the adjoint operators $(N^\pm)^*$ and $(E^\pm)^*$ are also two pair of complementary projections on $\mH^*$, and the restricted projection $N^+: E^+\mH\to N^+ \mH$ is an isomorphism if and only if 
$(N^-)^*: (E^-)^*\mH^*\to (N^-)^* \mH^*$ is an isomorphism.
\end{lem}

\begin{proof} This is \cite[Prop.~2.52]{AAH}
\end{proof}

\begin{proof}[Proof  of Proposition~\ref{prop:reg=dir}]  
We apply the abstract result  as follows. Here $\mH$ is the Hilbert space $\ran(D)\subset L_2(S^n;\V)=L_{2}$
and we realize its dual $\mH^*$ as $L_2/\mH^\perp$. The operators $N^\pm$ are $N^+: f\mapsto \begin{bmatrix} 0  \\ 
   f_\ta \end{bmatrix}$  and $N^-: f\mapsto \begin{bmatrix} f_{\no}  \\ 
   0 \end{bmatrix}$ from Definition~\ref{defn:DNops}. As both preserve  $\mH$, their adjoints induce operators on $\mH^*$. 
We choose 
$E^+= E_A^+$ and $E^-= E^-_A$.  
By Proposition~\ref{prop:wpequiviso}(ii) and reformulating \eqref{eq:Regmap} using $N^+$, well-posedness of the regularity problem for $A^*$ is equivalent to $N^+: E_{A^*}^+\mH \to N^+\mH$ being an isomorphism. By Lemma~\ref{lem:abstract}  this is equivalent to 
$(N^-)^*: (E_{A^*}^-)^*\mH^* \to (N^-)^*\mH^*$ being an isomorphism.
By (\ref{eq:dualityregtodir}) with the roles of $A$ and $A^*$ reversed, and written as an identity  on $\mH^*$ since both terms preserve $\mH^\perp$, this 
 translates into   $(N^-)^*: \tE_{A}^+\mH^* \to (N^-)^*\mH^*$ is an isomorphism. Using the definition of $\tE_{A}^+$, $(N^-)^*=N^-$ and $\mH^*=L_{2}/\mH^\perp$, this amounts to 
$\tE_0^+(L_2/\mH^\perp)\to L_2(S^n;\C^m)/\C^m: \tilde h^+\mapsto (\tE_A^+ \tilde h^+)_\no$
is an isomorphism. Using Lemma~\ref{lem:dir} and Proposition~\ref{prop:wpequiviso}(iii), 
this means that the Dirichlet problem for $A$ is well-posed.
\end{proof}

\subsection{Perturbations results}\label{sec:pert}

Proposition~\ref{prop:reg=dir} shows that it suffices to consider the 
Neumann and regularity problems and to study invertibility of the maps  \eqref{eq:Neumap}   and \eqref{eq:Regmap}.
Note that for $r$-independent coefficients $A= A_1$, we have $E_A^+= E_0^+$ and therefore
$(E_A^+ h^+)_\no = h^+_\no$ and $(E_A^+ h^+)_\ta= h^+_\ta$. 

\begin{lem}   \label{lem:automaticinj}
  Assume that $A$ are coefficients such that $I-S_{A}$ is invertible on $\mX$.
  Then the maps (\ref{eq:Neumap}) and (\ref{eq:Regmap}) are injective.
\end{lem}

\begin{proof}
  Assume that $h^+\in E_0^+\mH$ is such that $(E_A^+ h^+)_\perp=0$.
  As in Theorem~\ref{thm:inteqforNeu}, let $f\in\mX$ be such that $f_0= E_A^+ h^+$,
  so that we are assuming $(f_0)_\no=0$.
  For the corresponding $\mX^o$-solution $g= \nabla_\bx u$ to $\divv_\bx A g=0$,
  Green's formula shows that
$$
  \int_{\bO^{1+n}} (Ag, g) d\bx = \int_{S^n} (A_{1}g_1)_\no u_1 dx,
$$
where $g\in\mX^o\subset L_2(\bO^{1+n}; \C^{(1+n)m})$, $(A_{1}g_1)_\no= (f_0)_\no\in L_2(S^n;\C^m)$
and $u\in H^1(\bO^{1+n}; \C^m)$.
The accretivity of $A$ then shows that $g=0$. Hence $f=0$ and 
$h^+= E_0^+ f_0=0$.

The proof that the map $h^+\mapsto (E_A^+ h^+)_\ta$ is injective is similar.
In this case, we use that $u_1$ is constant, and $f_0\in\mH$ so that 
$\int_{S^n}(f_0)_\no dx=0$.
\end{proof}

We can now derive two perturbations results.
Our first result is about $L_{\infty}$ perturbation within the class of radially independent coefficients. We need two preliminary lemmas. 

\begin{lem}   \label{lem:vardomains}
Let $P_t$ be bounded projections in a Hilbert space $\mH$ which depend continuously on a parameter $t\in(-\delta,\delta)$,
and let $S:\mH\rightarrow\mK$ be a bounded operator into a Hilbert space $\mK$.
If $S: P_0\mH\rightarrow \mK$ is an isomorphism, then there exists $0<\epsilon<\delta$,
such that $S:P_t \mH\rightarrow\mK$ is an isomorphism when $|t|<\epsilon$.
If each $S: P_t\mH\rightarrow \mK$ is a semi-Fredholm operators with index $i_t$,
then all indices $i_t$ are equal.
\end{lem}

\begin{proof}
The first conclusion is in
\cite[Lem.~4.3]{AKMc} and the second one is proved similarly using in addition the continuity method.
\end{proof}

\begin{prop}   \label{prop:contofhardy}
  The operators
$\chi^+(DB_0+\sigma N)\in \mL(\mH)$,
defined for strictly accretive coefficients $A_1\in L_\infty(S^n; \mL(\V))$ and $\sigma\in\R$, depend continuously on $A_{1}$
 and $\sigma$.
\end{prop}

\begin{proof}
This is a corollary of Theorem~\ref{thm:QE} and  \cite[Prop.~2.42]{AAH}.
\end{proof}

 Here, note that for fixed $\sigma$ we called this operator $E_{0}^+$. Only its action on $\mH$ matters for well-posedness issues. In particular, this does not depend on the extension defined in Definition~\ref{defn:hardyprojs1D} when $\sigma=0$.

\begin{thm}    \label{cor:rindeppert}
  Assume that $A_1$ are $r$-independent coefficients for which the
  Neumann problem is well-posed.
  Then there exists $\epsilon>0$ such that the Neumann problem is well-posed for any 
  $r$-independent coefficients $A_1'$ such that $\|A_1-A_1'\|_\infty<\epsilon$.
  The corresponding results for the regularity and Dirichlet problems hold.
\end{thm}

\begin{proof} Lemma~\ref{lem:vardomains} and Proposition~\ref{prop:contofhardy} give the
 result for Regularity and Neumann problems as in \cite{AAH}.
For the Dirichlet problem, apply Proposition~\ref{prop:reg=dir}.
\end{proof}

The second result is perturbation from radially independent to radially dependent coefficients.

\begin{thm}     \label{prop:rdeppert}
  Assume that $A_1$ are $r$-independent coefficients for which the
  Neumann problem is well-posed.
  Then there exists $\epsilon>0$ such that the Neumann problem is well-posed for any 
  $r$-dependent coefficients $A$ such that $\lim_{\tau\to 0}\|\chi_{t<\tau}\E_t\|_*<\epsilon$.
  The corresponding results for the regularity and Dirichlet problems hold.
\end{thm}

\begin{proof} The condition on the coefficients implies that $I-S_{A}$ is invertible on $\mX$ and $I-S_{A^*}$ invertible on $\mY$ by Theorem~\ref{thm:SAFredholm}.

  We write the map (\ref{eq:Neumap}) as 
\begin{multline*}
  (E_A^+ h^+)_\no = h^+_\no+ \left( E_0^-\int_0^\tau e^{-s\Lambda} D\E_s f_s \right)_\no
  + \left( e^{-(\tau/2)\Lambda}E_0^-\int_\tau^\infty e^{-(s-\tau/2)\Lambda} D\E_s f_s \right)_\no \\
  =: h^+_\no + (h_1)_\no + (e^{-(\tau/2)\Lambda} h_2)_\no,
\end{multline*}
for $h^+\in E_0^+\mH$,
where $\|f\|_\mX\lesssim \|h^+\|_2$ by Theorem~\ref{thm:inteqforNeu}.
By assumption the map $E_0^+\mH\to \mH_\no:h^+\mapsto h^+_\no$ is invertible.
By \cite[Lem. 6.9]{AA1}, the norm of $E_0^+\mH\to \mH_\no:h^+\mapsto (h_1)_\no$
is $\lesssim \|\chi_{t<\tau}\E_t\|_*$.
Fix $\tau$ small enough so that $E_0^+\mH\to \mH_\no:h^+\mapsto (h^++h_1)_\no$
is invertible.
For the last term, we then have estimates 
\begin{multline*}
  \|h_2\|_2 \lesssim \int_\tau^\infty \| e^{-(s-\tau/2)\Lambda} D \|_{2\to 2} \|\E\|_\infty \|f_s\|_2 ds 
  \lesssim \|\E\|_\infty\int_\tau^\infty s^{-1}\|f_s\|_2 ds \\
  \lesssim \|\E\|_\infty\left( \int_\tau^\infty \|f_s\|_2^2 ds \right)^{1/2}
   \lesssim\|\E\|_\infty\|f\|_\mX \lesssim\|\E\|_\infty\|h^+\|_2.
\end{multline*}
Here we used the estimate
$\|e^{-(s-\tau/2)\Lambda} Dg\|_2\lesssim \|\Lambda e^{-(s-\tau/2)\Lambda} (D_0-\sigma N)B_0^{-1}P_{B_0\mH}g\|_2
\lesssim  ((s-\tau/2)^{-2}+ \sigma (s-\tau/2)^{-1} ) \|g\|_2$.

It follows that $E_0^+\mH\to \mH_\no:h^+\mapsto (e^{-(\tau/2)\Lambda}h_2)_\no$ is a compact
operator since $e^{-(\tau/2)\Lambda}$ is compact as a consequence of 
Proposition~\ref{prop:DBprops}.
We conclude that $E_0^+\mH\to \mH_\no:h^+\mapsto (E_A^+ h^+)_\no$ is a Fredholm operator
with index $0$. Lemma~\ref{lem:automaticinj} shows that it is injective,
hence an isomorphism.

Replacing normal components $(\cdot)_\no$ by tangential parts $(\cdot)_\ta$ in the proof above shows the
result for the regularity problem. 
Proposition~\ref{prop:reg=dir} then gives the result for the Dirichlet problem.
\end{proof}

\subsection{Positive results}\label{sec:positive}

We now give examples of radially dependent coefficients for which one has well-posedness.
Given Theorems~\ref {cor:rindeppert}
 and \ref{prop:rdeppert}, this induces results for perturbed coefficients.

\begin{prop}    \label{prop:blockwp}
  If $A$ are $r$-independent coefficients, and if $A$ is a block matrix, \textit{i.e.} $A_{\no\ta}= 0= A_{\ta\no}$,
  then the Neumann, regularity and Dirichlet problems with coefficients $A$ are well-posed.
\end{prop}

\begin{proof}
  By Proposition~\ref{prop:reg=dir}, it suffices to consider the Neumann and regularity problems.  Consider the projections $E_A^\pm= E_0^\pm$. As the maps \eqref{eq:Neumap} and \eqref{eq:Regmap} act on $E_{0}^+\mH\subset \mH$, it suffices to consider their action on $\mH$ throughout this proof. In this case, we have  
  $E_0:=\sgn(DB_0+\sigma N)= E_0^+- E_0^-$.
  Consider also the $\mH$ preserving projections $N^\pm$ from Definition~\ref{defn:DNops}.
  Define the anti-commutator
$$
  C:= \tfrac 12(E_0N+NE_0).
$$
  Since $B_0$ is a block matrix, $N$ commutes with $B_0$, which shows that
  $NE_0 N= N\sgn(DB_0+\sigma N)N= \sgn(N(DB_0+\sigma N)N)= -\sgn(DB_0-\sigma N)$,
  using $ND= -DN$.
  Hence, 
\begin{multline*}
  C= (E_0+ NE_0 N)N/2= (\sgn(DB_0+\sigma N)- \sgn(DB_0-\sigma N))N/2 \\
  = ((DB_0)^2+\sigma^2)^{-1/2} ((DB_0+\sigma N)-(DB_0-\sigma N))N/2
  = \sigma((DB_0)^2+\sigma^2)^{-1/2},
\end{multline*}
and it follows from Proposition~\ref{prop:DBprops} that $C$ is a compact operator on $\mH$.

We claim that 
\begin{alignat*}{2}
   (2E_0^+) N^+|_{E_0^+\mH}&= I+C|_{E_0^+\mH}, \qquad  N^+(2E_0^+)|_{N^+\mH}&= I+C|_{N^+\mH}, \\
   (2E_0^+) N^-|_{E_0^+\mH}&= I-C|_{E_0^+\mH},  \qquad N^-(2E_0^+)|_{N^-\mH}&= I-C|_{N^-\mH}.
\end{alignat*}
The first identity follows from the computation
\begin{multline*}
  (2E_0^+) N^+h^+= E_0^+(I+N)h^+= h^++ \tfrac 12(I+E_0)Nh^+ \\
  = h^++\tfrac 12(Nh^++2Ch^+- NE_0h^+) 
  = h^++Ch^+,\qquad \text{for all } h^+\in E_0^+\mH,
\end{multline*}
and the other three identities are proved similarly.
This proves that the maps $E_0^+\mH\to \mH_\no: h^+\mapsto h^+_\no$ and 
$E_0^+\mH\to \mH_\ta: h^+\mapsto h^+_\ta$ are Fredholm operators for any $\sigma\in\R$,
and for $\sigma=0$ it follows that they are isomorphisms.
By Lemma~\ref{lem:vardomains}, the indices of these operators are zero for any $\sigma\in \R$,
and Lemma~\ref{lem:automaticinj} implies that in fact the operators are isomorphisms
for $\sigma =(n-1)/2$.
\end{proof}

\begin{prop}   \label{prop:hermwp}
  If $A$ are $r$-independent coefficients, and if $A$ is Hermitean, \textit{i.e.} $A^*=A$,
  then the Neumann, regularity and Dirichlet problems with coefficients $A$ are well-posed.
\end{prop}

\begin{proof}
  By Proposition~\ref{prop:reg=dir}, it suffices to consider the Neumann and regularity problems.
  Let $h^+\in E_0^+\mH$ and define $f_t:= e^{-t\Lambda} h^+$.
  By Theorem~\ref{thm:NT}, we have $\pd_t f_t+ D_0f_t=0$, $\lim_{t\to 0}f_t=h^+$ and
  rapid decay of $f_t$ as $t\to \infty$.
  We calculate
\begin{multline*}
   (Nh^+, B_0 h^+)= -\int_0^\infty \pd_t(Nf_t, B_0 f_t)dt \\
   =  \int_0^\infty \Big( (N D_0f_t, B_0 f_t)+ (Nf_t, B_0 D_0f_t) \Big)dt \\
   =   \int_0^\infty \Big( ((NDB_0+ DB_0^* N) f_t, B_0 f_t)+ \sigma(f_t, (B_0+ NB_0N)f_t) \Big)dt \\
   = \sigma\int_0^\infty (f_t, (B_0+ B_0^*) f_t)dt.
\end{multline*}
On the last line, we used that $A^*=A$, or equivalently $B_0^*= NB_0 N$, so that $NDB_0+ DB_0^* N=0$.
This gives the estimate
$$
  | -(h^+_\no, (B_0h^+)_\no)+ (h^+_\ta, (B_0 h^+)_\ta) |\lesssim\sigma \int_0^\infty \|f_t\|^2_2 dt.
$$ 
From this we deduce the estimate
$$
  \|h^+\|_2^2 \lesssim \re(h^+, B_0 h^+)\lesssim | (h^+_\no, (B_0 h^+)_\no) |+ \|f\|_{L_2(\R_+;\mH)}^2
  \lesssim \|h^+_\no\|_2 \|h^+\|_2 + \|f\|_{L_2(\R_+;\mH)}^2.
$$
This shows that the map (\ref{eq:Neumap}) is a semi-Fredholm map, if we prove
that the map $\mH\to L_2(\R_+;\mH): h\mapsto (e^{-t\Lambda} h)_{t>0}$ is compact.
To see this, note that square function estimates for $D_0$ give the estimate
$$
  \int_0^\infty \|f_t\|_2^2 dt= \int_0^\infty \|\psi_t(D_0) (\Lambda^{-1/2}f)\|_2^2 \frac {dt}t
  \lesssim \| \Lambda^{-1/2}f \|_2^2,
$$
where $\psi_t(z):= \sqrt{t|z|}e^{-t|z|}$, and $\Lambda^{-1/2}$ can be seen to be a compact operator
on $\mH$ by Proposition~\ref{prop:DBprops}.
Taking $P_s= \chi^+(DB^{s}+\sigma N)$ in Lemma~\ref{lem:vardomains}, where $B^{s}$, $s\in[0,1]$, denotes the straight line in $L_\infty(S^n;\mL(\V))$ from $I$ to $B_0$, shows that the index 
of the map (\ref{eq:Neumap}) is $0$. By Lemma~\ref{lem:automaticinj}, this map is in fact an isomorphism.

The proof for the regularity problem is similar, using instead the estimate
$$
  \|h^+\|_2^2 \lesssim | (h^+_\ta, (B_0 h^+)_\ta) |+ \|f\|_{L_2(\R_+;\mH)}^2
  \lesssim \|h^+_\ta\|_2 \|h^+\|_2 + \|f\|_{L_2(\R_+;\mH)}^2.
$$
\end{proof}

\begin{prop}    \label{prop:smoothwp}
  If $A$ is a H\"older regular $ C^{1/2+\varepsilon}(S^n;\mL(\C^{(1+n)m}))$,  $r$-independent coefficients, for some $\varepsilon>0$,  
  then the Neumann, regularity and Dirichlet problems with coefficients $A$ are well-posed.
\end{prop}

For the proof, we need the following lemmas.

\begin{lem}\label{lem:rellichcom} Let $B_{0} \in C^{1/2+\varepsilon}(S^n;\mL(\V))$ be the matrix associated to $A$. Then for all $f,g\in \mH$, $
  | ([ |D|^{1/2}, B_0]f, g) | \lesssim  \|f\|_2 \| g \|_2.
$
\end{lem}

\begin{lem} \label{lem:rellichdom} Under the same assumptions, $\dom(|D|^{1/2})\cap \mH=\dom(|D_{0}|^{1/2})\cap \mH$ with equivalent graph domain norms. 
\end{lem}

\begin{proof}[Proof of Proposition \ref{prop:smoothwp}]
  Consider first the Neumann and regularity problems.
  Let $h^+\in E_0^+\mH$ and define $f_t:= e^{-t\Lambda} h^+$. 
  By Theorem~\ref{thm:NT}, we have $\pd_t f_t+ D_0f_t=0$ and $\lim_{t\to 0}f_t=h^+$ and $\lim_{t\to \infty}f_t=0$ with rapid decay. 
  We begin with the observation that $(\sgn(D)h^+, h^+)= \re (\nabla_S(-\divv_S\nabla_S)^{-1/2}h^+_{\no}, h^+_{\ta})$.
  Thus $|(\sgn(D)h^+, h^+)| \le \|h^+_{\no}\|_{2}\|h^+_{\ta}\|_{2}$. Now, we calculate for fixed $T>0$
\begin{multline*}   \label{eq:smoothrellich}
(\sgn(D)h^+, h^+)- (\sgn(D)f_{T}, f_{T})= -\int_0^T \pd_t(\sgn(D) f_t, f_t)dt  \\
   =  \int_0^T \Big( (\sgn(D) (DB_0+\sigma N)f_t, f_t)+ (f_t, \sgn(D)(DB_0+\sigma N)f_t) \Big)dt
\\
   =  2\re \int_0^T (|D|B_0 f_t, f_t) dt \\
  =  2\re \int_0^T\Big( (B_0 |D|^{1/2}f_t, |D|^{1/2} f_t) dt+ ([ |D|^{1/2}, B_0]f_t, |D|^{1/2} f_t) \Big)dt,
     \end{multline*}
using that $\sgn(D)D= |D|$ and $\sgn(D)N+ N\sgn(D)=0$ in the third equality and  Lemma \ref{lem:rellichdom} in the last since $f_{t}\in \dom(|D_{0}|^{1/2})\cap \mH\subset \dom(|D|^{1/2})$. 
Accretivity of $B_0$  and Lemma \ref{lem:rellichcom} lead to the estimate
$$
  \int_0^T\||D|^{1/2}f_t\|_2^2 dt\lesssim \|h^+_\no\|_2 \|h^+_\ta\|_2 + |(\sgn(D)f_{T}, f_{T})|+ \int_0^T \|f_t\|_2\||D|^{1/2}f_t\|_2 dt,
$$
and by absorption, to the same estimate but with last term equal $\int_0^T \|f_t\|_2^2 dt$. 
Due the rapid decay of $\|f_{t}\|_{2}$ when $t\to \infty$, we conclude that
$$
  \int_0^\infty\||D|^{1/2}f_t\|_2^2 dt\lesssim \|h^+_\no\|_2 \|h^+_\ta\|_2 + \int_0^\infty \|f_t\|_2^2 dt.
$$
Since   $\||D_0|^{1/2}f_{t}\|_2 \lesssim \||D|^{1/2}f_{t}\|_2 + \|f_{t}\|_2$ from 
Lemma \ref{lem:rellichdom}, we may replace $D$ by $D_{0}$ in the left hand side. 
Since square function estimates for $D_0$ give
$$
  \int_0^\infty\||D_0|^{1/2}f_t\|_2^2 dt= \int_0^\infty \|(t|D_0|)^{1/2} e^{-t|D_0|}h^+\|_2^2 \frac{dt}t \approx \|h^+\|_2^2,
$$
this implies
$$
  \|h^+\|_2^2 \lesssim  \|h^+_\no\|_2 \|h^+_\ta\|_2 + \int_0^\infty \|f_t\|_2^2 dt.
$$
Well-posedness of the Neumann and regularity problems now follows as in the proof
of Proposition~\ref{prop:hermwp}. Proposition~\ref{prop:reg=dir} then gives the result for the Dirichlet problem.
\end{proof}

 \begin{proof}[Proof of Lemma~\ref{lem:rellichcom}]
Note that $ D^2$ agrees on  $\mH$ with  the (positive) Hodge--Laplace operator
$$
   \Delta_S:= -\begin{bmatrix} \divv_S \nabla_S & 0 \\ 0 & \nabla_S \divv_S -\curl_S^* \curl_S
    \end{bmatrix},
$$
where $\curl_S: L_2(S^n; (T_\C S^n)^m)\to L_2(S^n; \wedge^2(T_\C S^n)^m)$ is the tangential curl/exterior derivative on $S^n$.
Since  $f, g\in \mH$, we have 
$$([ |D|^{1/2}, B_0]f, g)= ([ \Delta_S^{1/4}, B_0]f, g)$$ and it suffices to prove that $[ \Delta_S^{1/4}, B_0]$ is bounded on $L_{2}$. Since  the action of $B_{0}$ mixes functions and vector fields, some care has to be taken.  

(i) First, by functional calculus we can replace $\Delta_{S}$ by $T_{0}=\Delta_{S}+\lambda$ for any $\lambda\in \R^+$, to be chosen large later, as $ \Delta_S^{1/4} -  (\Delta_S + \lambda)^{1/4}$ is bounded.  

(ii) Next,  the commutator estimate is a local problem and by a partition of unity argument and rotational invariance of the assumptions,  we can assume that $f$ is supported in the lower hemisphere and it is enough to show that $\|\zeta[T_{0}^{1/4}, B_{0}]f\|_{2}\lesssim \|f\|_{2}$ when  the smooth scalar function $\zeta$ is 1 a neighborhood of the support of $f$. Indeed $(1-\zeta)[T_{0}^{1/4}, B_{0}]f= -[[\zeta, T_{0}^{1/4}], B_{0}]f$, where the inner commutator is seen to be bounded on $L_{2}$.  

(iii) Now  using rescaled pullback $\rho^*$ to $\R^n$ from the proof of 
Theorem~\ref{thm:QE} yields  $ \rho^*(T_{0}f) = T_{1}(\rho^*f)$ with 
$$
 T_1:=  -\begin{bmatrix} \divv_{\R^n} d^{2-n} \nabla_{\R^n} d^n & 0 \\ 0 & \nabla_{\R^n} d^n \divv_{\R^n} d^{2-n}-d^{n-2} \curl_{\R^n}^* d^{4-n} \curl_{\R^n} \end{bmatrix} + \lambda I
 $$
in $L_2(\R^n; \C^{(1+n)m})$, with $d(y)= (|y|^2+1)/2$ inside  $|y|<1$ and extended to a smooth function on $\R^n$, with $d(y)=2$ for $|y|>2$
and $1/2\le d(y)\le 2$ for all $y$.  Any extension would do since $\rho^*f$ is supported in $|y|<1$.
(The proof of this equality builds on the fundamental differential geometric fact that the standard pullback operation intertwines $\nabla$ on $S^n$ and $\R^n$, as well as $\curl$, and the adjoint results for
$\divv$ and $\curl^*$. Note that the rescaled pullback $\rho^*$ from Theorem~\ref{thm:QE}
equals the standard pullback on vectors, but is $d^{-n}$ times the standard pullback on scalars.)
A further calculation shows that $T_{1}= - \divv_{\R^n}d^2\nabla_{\R^n} + R +\lambda I$,
where $R$ is a first order differential operator with smooth coefficients and 
$\divv_{\R^n}d^2\nabla_{\R^n}$ acts componentwise on $\C^{(1+n)m}$-valued 
functions.  Note that the coefficients of $R$ must vanish outside $|y|<2$ by construction.
We now  choose  $\lambda$ large enough to guarantee the accretivity condition
 $\re (T_{1}g,g)\ge \delta \|g\|_{W^{1}_{2}}^2$ with $\delta >0$ and all $g\in W^1_{2}(\R^n; \C^{(1+n)m})$.
Consider $K$, $\eta$ and $g$ as in the proof of 
Theorem~\ref{thm:QE} and $\zeta=(\rho^*)^{-1}\eta$ and $f=(\rho^*)^{-1}g$. We claim that
$\|\zeta T_{0}^{1/4}f - (\rho^*)^{-1}\eta^2 T_{1}^{1/4}g \|_{2}\lesssim \|g\|_{2}\approx \|f\|_{2}$.
For both operators $T_{i}$, we use the identity 
\begin{equation}
\label{eq:1/4}
T_{i}^{1/4}= c \int_0^\infty s^{1/2} T_{i}(I+s^2 T_{i})^{-1} ds= c\int_0^\infty ( I-(I+s^2T_{i})^{-1} )\frac{ds}{s^{3/2}}.
\end{equation}
The part with $s>1$ gives rise to a bounded operator for each $T_{i}$. For the $s<1$ integral of the difference, we use
 the identity obtained as in Theorem \ref{thm:QE}
 \begin{multline*}
\zeta (I+s^2T_{0})^{-1}f  -(\rho^*)^{-1} \eta^2 (I+ s^2 T_{1})^{-1} g \\ 
   =\zeta (I+s^2T_{0})^{-1}(\rho^*)^{-1}  s^2[\eta, T_{1}]  (I+ s^2T_{1})^{-1} g
\end{multline*}
so that 
$$
  \| \zeta (I+s^2T_{0})^{-1}f  -(\rho^*)^{-1} \eta^2 (I+ s^2 T_{1})^{-1} g      \|_2\lesssim s \|g\|_2,
$$
using that the commutator $[\eta, T_{1}] $ is a first order operator.

(iii) We are reduced  to showing  that $[T_{1}^{1/4}, \widetilde B_{0}]$ is bounded on $L_2(\R^n; \C^{(1+n)m})$ with  $\widetilde B_0 := \rho^* B_0(\rho^*)^{-1}$ of $B_{0}$ on $|y|\le 1$ extended  to a   bounded matrix function of class $C^{1/2+\varepsilon}$ on $\R^n$.  We now eliminate the $R$ part of $T_{1}$.
Set $T_{2}:=- \divv_{\R^n}d^2\nabla_{\R^n} + 1$ acting componentwise in $L_2(\R^n; \C^{(1+n)m})$. The chosen extension of $d$ insures that $T_{2}$ is  accretive (in fact self-adjoint)  as $T_{1}$. We claim that 
$T_{1}^{1/4} - T_{2}^{1/4}$ is bounded. We use again \eqref{eq:1/4} for each $T_{i}$. The part with $s>1$ gives rise to a bounded operator for each $T_{i}$. For the $s<1$ integral of the difference, we use
$\|(I+s^2T_{1})^{-1}-(I+s^2T_{2})^{-1}\|\lesssim s$ by the resolvent formula, because $T_{i}$  have same second order term. This proves the claim. 

(iv) Hence, it remains to estimate the commutator $C= [T_{2}^{1/4}, \widetilde B_{0}]$.  
Since $T_{2}$ acts componentwise, so does $T_{2}^{1/4}$ and the commutator consists of a matrix of commutators with each component of $\widetilde B_{0}$. Thus it suffices to estimate $C= [T_{2}^{1/4}, b]$  in $L_2(\R^n; \C)$, with 
$b$  scalar-valued. 
To see this,  we use the the different representation  for $T_{2}^{1/4}$ to obtain
$$
C= c\int_{0}^{\infty} [s^2T_{2}e^{-s^2T_{2}} , b] \frac{ds}{s^{3/2}}.
$$
The $s>1$ integral is trivially bounded, using boundedness of $b$ and $s^2T_{2}e^{-s^2T_{2}}$. 
For $s<1$, we have $\| [s^2T_{2}e^{-s^2T_{2}} , b] \|_{L_2\to L_2}\lesssim s^{1/2+\epsilon}$ 
using pointwise  decay and regularity for the  kernel of $s^2T_{2}e^{-s^2T_{2}}$  and regularity of $b$.  See, for example, \cite{Au2} where it is proved that under continuity of the coefficients (here $d^2$), the kernel of the semigroup $e^{-sT_{2}}$, $s<1$, has Gaussian estimates (this is in fact due to Aronson for real measurable coefficients) and   H\"older regularity in each variable with any exponent in $(0,1)$, in particular larger that $1/2 + \epsilon$. From here, the same estimates hold for $sT_{2}e^{-sT_{2}}=-s\pd_{s}e^{-sT_{2}}$ by analyticity of the semigroup. This takes care of the $s<1$ integral. Further details are left to the reader. 
\end{proof}

\begin{proof}[Proof of Lemma \ref{lem:rellichdom}] Recall that $D_{0}=DB_{0}+\sigma N$. As before, by a representation formula it is easy to prove that $|DB_{0}+\sigma N|^{1/2} - |DB_{0}|^{1/2}$ is bounded on $L_{2}$. 
Hence we may replace $D_{0}$ by $DB_0$. We remark that $\mH$ is invariant for both $D$ and $DB_{0}$.

As $P_{\mH}B_{0}$ is an isomorphism of $\mH$, for $f\in \mH$, $f\in \dom(|DB_{0}|)$ if and only if $P_{\mH}B_{0}f \in \dom(|D|)$ and in this case
$$
\||DB_0|f\|_{2} \approx \|DB_{0}f\|_{2}\approx \|D(P_{\mH}B_{0}f)\|_{2} \approx \||D|(P_{\mH}B_{0}f)\|_{2}.
$$
Complex interpolation for sectorial operators  (see \cite{AMcNhol}) shows that for $f\in \mH$, $f\in \dom(|DB_{0}|^{1/2})$ if and only if $P_{\mH}B_{0}f \in \dom(|D|^{1/2})$ and
$$
\||DB_0|^{1/2}f\|_{2}\approx \||D|^{1/2}(P_{\mH}B_{0}f)\|_{{2}}.$$
Next, for $f\in \mH \cap \dom(|D|^{1/2})$, we have $|D|^{1/2}f \in \mH$ so that
$$\||D|^{1/2}f\|_{2} \approx \| P_{\mH}B_{0}|D|^{1/2}f\|_{2}.$$
Thus it suffices to show that for $f\in \mH$,  $f\in \dom(|D|^{1/2})$ if and only if $P_{\mH}B_{0}f \in \dom(|D|^{1/2})$. This is where we use the regularity of $B_{0}$ to yield
 $\||D|^{1/2}(P_{\mH}B_{0}f)-  P_{\mH}B_{0}|D|^{1/2}f\|_{2}\lesssim \|f\|_{2}$ when $f\in \mH$
as  a direct consequence of Lemma \ref{lem:rellichcom} and the fact that $D$ and $P_{\mH}$ commute.
\end{proof}

\begin{rem} 
Using the $T1$ theorem, the commutator $C$ of the proof of Lemma \ref{lem:rellichcom} is bounded on $L_{2}$ when $(-\Delta+1)^{1/4}b\in$ BMO (and $b\in L_{\infty}$). The converse is also true. This can be shown to be a regularity condition between $C^{1/2}$ and $C^{1/2+\varepsilon}$. So well-posedness holds under this condition (expressed in local coordinates  on the coefficients of $B_{0}$). This is probably the best  conclusion we can draw from this method. However, we suspect that $C^{\varepsilon}$ should be enough in general.  
\end{rem}

%
%
%
%
%
\section{Uniqueness}\label{sec:unique}

The following is the class of solutions in Definition~\ref{def:wpdahlberg}.

\begin{defn}   \label{defn:Dsolutions}
 
   By a {\em $\mD^o$-solution to the divergence form equation}, with coefficients $A$, 
   we mean a weak solution of $\divv_{\bx} A \nabla u=0$ in $\bO^{1+n}$ with $\|\tN^o(u)\|_{2} <\infty$. 
   \end{defn}
   
   Note that unlike the previous classes, $\mD^o$-solutions are defined through an estimate on
 $u$ itself, not on the gradient $\nabla_\bx u$.
 
Under the Carleson control on the discrepancy, we know that $\mY^o$-solutions are $\mD^o$-solutions. We would like to know the converse.  At this stage we need assumption of well-posedness in the sense of Definition~\ref{defn:wpbvp}.  It goes via identification with variational solutions for smooth data which will be also useful later. 

\begin{lem}  \label{prop:sioequalvarsol}
  Let  $A$ be coefficients such that $\|\E\|_*<\infty$ and $I-S_A$ is invertible on $\mY$ and on $\mX$,
  and assume that the regularity problem and the Dirichlet problem in the sense of Definition~\ref{defn:wpbvp} both are well posed.
  Let $\bphi\in L_{2}(S^n;\C^m)$ be Dirichlet datum such that $\nabla_S \bphi\in L_2(S^n; (T_{\C}S^n)^m)$.
  Then the solution $u$ to the Dirichlet problem in the sense of Definition~\ref{defn:wpbvp}
  coincides with the variational solution with datum $\bphi$.
\end{lem}

\begin{proof}
  By Proposition~\ref{prop:wpequiviso}, there is a unique $h^+\in E_0^+ \mH$ such that
  $(E_A^+ h^+)_\ta= \nabla_S \bphi$, since the regularity problem is well-posed.
From Lemma~\ref{lem:spectralsplits}, we know that $D:\tE_0^+ L_2\to E_0^+ \mH$
  is surjective. Let $\tilde h^+\in \tE_0^+L_2$ be such that $D\tilde h^+ =h^+$.
  Consider now $\tilde \bphi:= (\tE_A^+ \tilde h^+)_\no$.
  We claim that $\nabla_S \tilde \bphi= \nabla_S \bphi$. Indeed, this follows from taking the tangential part in the
  intertwining formula
$$
  D \tE_A^+ = E_A^+ D,
$$
which is readily verified from Lemma~\ref{lem:intertwduality} and definitions of $\tE_{A}^+$, $E_{A}^+$. 
Thus $\tilde \bphi-\bphi$ is constant.
As in the proof of Corollary~\ref{cor:diransatz}, by adding a normal constant in $\tE_0^+\mH^\perp$ 
to $\tilde h^+$, we may assume that $\tilde \bphi=\bphi$.

Given this $\tilde h^+$, the solution $u$ to the Dirichlet problem with datum $\bphi$ is
given by the normal component of 
$$
  v:= \Big( I+ \tS_A(I-S_A)^{-1} D \Big) e^{-t\tilde\Lambda} \tilde h^+
$$
as in Corollary~\ref{cor:diransatz}(iii). Next, we have
$$
  f :=Dv= (I-S_A)^{-1} e^{-t\Lambda} h^+
$$
and  $f$ is  the  conormal gradient to the solution to the regularity problem with datum $\nabla_S \bphi$.
In particular $f\in \mX\subset L_2(\R_+\times S^n;\V)$ by Lemma~\ref{lem:XlocL2}.

Translated to $\bO^{1+n}$, this shows that the solution $u$ to the Dirichlet problem with datum $\bphi$
has $\nabla_\bx u\in L_2(\bO^{1+n}; \C^{(1+n)m})$.
This shows that $u$ is  a variational solution. Uniqueness of the Dirichlet problem in this class completes the proof.
\end{proof}

\begin{rem}
Note that since $\mX\subset L_2(\R_+\times S^n;\V)$, solutions to the regularity
and Neumann problem always coincide with the variational solutions, by the uniqueness
of such.
In the setting of the half-space, as in \cite{AAM, AA1}, it was shown in \cite{AxNon} that this uniqueness
result does not hold. 
As pointed out in \cite[Rem. 5.6]{AAM}, the problem occurs at infinity for the regularity and Neumann 
problems, which explains why uniqueness holds for the bounded ball.
Although the analogue of \cite{AxNon} for the Dirichlet problem on the ball is not properly understood 
at the moment, Theorem~\ref{thm:equiv2} below shows that uniqueness of solutions essentially
holds also for the Dirichlet problem on the unit ball.
\end{rem}

\begin{prop} \label{pro:repradial} 
Let $A$ be radially independent coefficients and
 assume that the regularity problem and the Dirichlet problem in the sense of Definition~\ref{defn:wpbvp}  are both well-posed. Then all $\mD^o$-solutions are given by $u=e^{-\sigma t}( e^{-t\tilde\Lambda} \tilde h^+)_{\no}$ for a unique $\tilde h^+ \in \tE_{0}^+L_{2}$.  In particular, the class of $\mD^o$-solutions is the same as the class of $\mY^o$-solutions, and the estimate
$$
 \|\tN^o(u)\|^2_{2}  \approx \int_{\bO^{1+n}} |\nabla_\bx u|^2 (1-|\bx|)d\bx
  + \left| \int_{S^n} u_1(x) dx \right|^2  
  $$
  holds for all weak solutions. 
\end{prop}

\begin{proof} Let $u$ be a $\mD^o$-solution. For almost every $\rho\in(0,1)$, $\nabla_{S}u_\rho \in L_{2}(S^n; (T_{\C}S^n)^m)$ and $u_\rho \in L_{2}(S^n;  \C^m)$. Fix such $\rho$. As in the proof of 
Proposition~\ref{prop:sioequalvarsol}, we can find $ h^+_\rho\in E_{0}^+\mH$, $\tilde h^+_\rho\in \tE_{0}^+L_{2}$ with $D\tilde h^+_\rho=h^+_\rho$,  $(h^+_\rho)_\ta= \nabla_S u_\rho$ and $ (\tilde h^+_\rho)_\no= u_\rho$ on $S^n$.
Using radial independence, the function $\tilde u_\rho(r x):=e^{\sigma t}(e^{-t \tilde \Lambda} \tilde h^+_\rho)_{\no}(x)$ (here, $\rho$ is fixed and $e^{-t}=r \in (0,1)$) thus extends to a solution of the divergence form equation with coefficients $A$, and it is a variational solution by Proposition~\ref{prop:sioequalvarsol}. Since $\bx \mapsto u(\rho\bx)$ is also a variational solution and agrees with $\tilde u_\rho$ on $S^n$, we conclude by uniqueness that
$u(\rho r\cdot)= e^{\sigma t}(e^{-t\tilde \Lambda} \tilde h^+_\rho)_{\no}$ as $L^2(S^n;\C^m)$-functions for all $e^{-t}=r \in (0,1]$, and almost every $\rho\in (0,1)$. 

From this representation, we see that the right hand side is continuous in $t$, with range in $L_{2}$, so the left hand side is continuous in $r$. We also have  $\|u_{\rho r}\|_{2} \lesssim \|\tilde h^+_\rho\|_{2}\approx \|u_{\rho}\|_{2}$ for every $r\in (1/2, 1]$ and almost every $\rho\in (0,1)$. 
The last equivalence comes from the well-posedness of the Dirichlet problem, and the implicit constants are independent of $\rho$. 
 As $\sup\limits_{1/2<\rho<1} (1-\rho)^{-1} \int_{\rho}^{\frac{\rho +1}{2}}\|u_s\|_{2}ds \lesssim \|\tN^o(u)\|_{2}<\infty$, we conclude that $\|\tilde h^+_\rho\|_{2}$ is bounded for $1/2 < \rho <1$. Consider a weak limit $\tilde h^+ \in \tE_{0}^+L_{2}$ of a subsequence  $\tilde h^+_{\rho_{n}}$ with $\rho_{n}\to 1$. 
 Reversing the roles of $\rho $ and $r$,   for almost every  $r<1$,  $u_{\rho_{n }r}$ converges in $L_{2}(S^n;\C^m)$ to $u_{r}$, so that $u_{r}=e^{\sigma t}( e^{-t\tilde\Lambda} \tilde h^+)_{\no}$. Extending to all $r$, the representation is proved.
  
In particular, this shows that the classes of $\mY^o$-solutions and of $\mD^o$-solutions  of $Lu=0$ coincide under our assumptions.  
\end{proof}

Note that the full force of $\|\tN^o(u)\|_{2}<\infty$ is not used and the   condition $$\sup\limits_{1/2<r<1} r^{-1} \int_{1-2r}^{1-r}\|u_{\rho}\|_{2}d\rho <\infty$$ suffices in the proof of Proposition~\ref{pro:repradial}.

\begin{rem}
If $A$ is not $r$-independent, we need to know that $A(\rho\cdot)$ satisfies the large Carleson condition for all $1/2<\rho<1$ to run the argument. This is not clear if we just assume this for $A$. However, if we assume that $A$ is continuous on $\overline{\bO^{1+n}}$ and satisfies the square Dini condition of Theorem~\ref{thm:Dinisquarerealintro}, then this can be checked.
\end{rem}

\begin{proof}[Proof of Theorem~\ref{thm:semigroup}] We consider $A_1\in L_\infty(S^n; \mL(\C^{(1+n)m}))$, radially independent coefficients  which are
  strictly accretive  in the sense of \eqref{eq:accrasgarding}. Assume that the Dirichlet problem with coefficients $A_{1}$ is well-posed.
By Corollary~\ref{cor:uniquenessdir},  we have $\mP_{r}u_{1}=r^{-\sigma}(e^{-t\tilde\Lambda}v_{0})_{\no}$ with 
$r=e^{-t}$ and $v_{0}$ given by the inverse of the well-posedness map \eqref{eq:Dirmap} from applied to $u_{1}$. The assumed uniqueness of the solution $u$ allows us to prove the product rule of $\mP_{r}$ by considering $\mP_{r}u_{1}$ as another boundary data.  The existence of the generator with domain contained in $W^1_{2}(S^n;\C^m)$ is as in \cite{Au} in the setting of the upper-half space. There, the if direction was deduced   using the duality
 principle between Dirichlet and regularity. An examination of the argument there reveals that  the only if direction was implicit. We can repeat  the same duality argument using Proposition~\ref{prop:reg=dir}.
 \end{proof}

\begin{proof} [Proof of Theorem~\ref{thm:radind}] By Proposition~\ref{pro:repradial} we know that the two classes of $\mD^o$- and $\mY^o$-solutions are the same. Thus the assumed well-posedness for $\mY^o$-solutions carries over to $\mD^o$-solutions. This completes the proof. \end{proof}

%
%
%
%
%
\section{New well-posedness results for real equations}\label{sec:realequations}

We now specialize to the case of equations ($m=1$) with real coefficients, and make this assumption for  the coefficients $A$ throughout this section unless mentioned otherwise. For such equations the theory of solvability for the Dirichlet problem using non-tangential maximal control is rather complete for real symmetric equations, but not so much for non symmetric equations. In \cite{KKPT}, the extensions of the tools for real non symmetric equations are discussed and we refer there for details.

We have developed a strategy using square functions  rather  than non-tangential maximal functions and our goal here is  to tie this up. It is convenient to introduce the square function
$$
S(u)(x)= \left( \int_{\by\in \Gamma_x} |\nabla u(\by)|^2 \frac{d\by}{(1-|\by|)^{n-1}} \right)^{1/2}, \qquad x\in S^n,
$$
($\Gamma_x$ denoting a truncated cone with vertex $x$ and  axis the line $(0,x)$)
and the divergence form operator $L:= -\divv_\bx A \nabla_\bx$.
We note that a weak solution to $Lu=0$ is in $\mY^o$ if and only if $S(u)\in L_{2}(S^n)$, the measure being the surface measure.  
We have so far studied $\mY^o$-solutions  and well-posedness in this class, which is convenient to denote here by {\em well-posedness in $\mY^o$}. (This was called ``in the sense of Definition~\ref{defn:wpbvp}'' in the introduction.) 

Recall that by  a {\em $\mD^o$-solution} of $Lu=0$, we mean a weak solution with $\|\tN^o(u)\|_{2} <\infty$.  As said in the introduction,  we may replace $\tN^o(u)$ by the usual  pointwise non-tangential maximal function. 
For the Dirichlet problem, we shorten   well-posedness in the sense of Dahlberg  in Definition~\ref{def:wpdahlberg} to  {\em well-posedness in $\mD^o$}.
  
On regular domains such as $\bO^{1+n}$, there is always  a  unique variational solution $u\in W^{1}_{2}(\bO^{1+n})$, which is  in addition continuous in $\overline{\bO^{1+n}}$, to the Dirichlet problem with data 
$\bphi\in C(S^n)$ by results of Littman, Stampacchia and Weinberger \cite{LSW} which extend to real non-symmetric equations (see \cite{KKPT}). 
  Thus, it is natural to ask whether this solution satisfies $\|\tN^o(u)\|_{2} \le C \|\bphi\|_{2}$ with $C$ depending on the Lipschitz character of $S^n$. By a density argument, it suffices to do this for smooth $\bphi$, say $\bphi \in C^1(S^n)$. If this is the case, then the Dirichlet problem $(D)_{2}$ is said to be {\em solvable}. 
 
 From the maximum principle and Harnack's inequalities, one can study the $L$-elliptic measure $\omega$, say at 0, which is the probability measure $C(S^n) \ni \bphi \mapsto u(0)$ with $u$ the above solution.  The question whether $\omega$ is absolutely continuous with respect to surface measure is central. 
 
The result, somehow folklore but we have not seen it stated explicitely in the literature, summarizing the state of the art is the following.
 
 \begin{thm}\label{thm:equiv1} 
 Let $L=-\divv_\bx A \nabla_\bx$ 
 be a real elliptic operator in $\bO^{1+n}$, $n\ge 1$.   
 Then the following statements are equivalent
 \begin{itemize}
  \item[{\rm (i)}] The Dirichlet problem is well-posed in $\mD^o$.
  \item[{\rm (ii)}]  $(D)_{2}$ is solvable.
  \item[{\rm (iii)}]  The $L$-elliptic measure $w$ is absolutely continuous with respect to surface measure and its Radon-Nikodym derivative $k$ satisfies the reverse H\"older $B_{2}$ condition, \textit{i.e.} there is a constant $C<\infty$ such that for all surface balls $B$ on $S^n$, 
  $$
  \left( |B|^{-1} \int_{B} k^2(x) \, dx \right)^{1/2} \le C |B|^{-1} \int_{B} k(x) \, dx.
  $$
\end{itemize}
\end{thm}

\begin{proof}
The proof that (ii) is equivalent to (iii) is stated for real non-symmetric operators in \cite[p.241]{KKPT}. The proof that (i) implies (ii) is trivial. For $\bphi\in C(S^n)$, the variational solution is bounded, hence satisfies $\|\tN^o(u)\|_{2}<\infty$ since $\bO^{1+n}$ is bounded. By uniqueness in (i), it is the unique solution and the continuity estimate that follows from well-posedness shows $\|\tN^o(u)\|_{2} \le C \|\bphi\|_{2}$. So (ii) holds. It remains to see (ii) implies (i). Existence and continuity estimate are granted from $(D)_{2}$. Uniqueness follows the argument in \cite[p.125-126]{FJK}, using the equivalent assumption (iii) instead of (ii).
The extension to non-symmetric real operators is allowed from the details in \cite{KKPT}. 
\end{proof}

\begin{thm}\label{lem:DJK} 
Let $L$ be an elliptic operator with real coefficients. Then all weak solutions to $Lu=0$ satisfy $\|S(u) \|_{L_{2}( d\mu)}\lesssim \|\tN^o(u)\|_{L_{2}(d\mu)}$ for all $A_{\infty}$ measure $\mu$ with respect to $L^*$-elliptic measure. 
\end{thm}

\begin{proof} 
This is the result of  \cite{DJK} where this is proved when $L=L^*$. Aside from properties of solutions that are valid for all real operators, the proof relies on the use of \cite[eq.~(7) p.101]{DJK}.  Since $Lu=0$ and there is an integration by parts,  it has to be modified with $g(X,Y)$  and $\nu$ being respectively the Green's function and  the elliptic measure of the adjoint  on the domain $\Omega$ there. This is why the $A_{\infty}$ property with respect to $L^*$-elliptic measure intervenes in the hypotheses. Further details are in \cite{DJK}.  
\end{proof}

\begin{cor}\label{cor:DJK} 
Let $L$ be an elliptic operator with real coefficients.  Assume that the Dirichlet problem is well-posed in $\mD^o$ for $L^*$. Then all weak solutions to $Lu=0$ satisfy $\|S( u) \|_{2}\lesssim \|\tN^o(u)\|_{2}$. In particular, $\mD^o$-solutions of $Lu=0$ are $\mY^o$-solutions of $Lu=0$ under this assumption. 
\end{cor}

\begin{proof}
By Theorem \ref{thm:equiv1},  $L^*$-elliptic measure is $A_{\infty}$ with respect to surface measure, and vice-versa by \cite{CF}. So  $\|S(u) \|_{\mY^o}\lesssim \|\tN^o(u)\|_{2}$ follows from Lemma~\ref{lem:DJK}. 
\end{proof}

Note that Corollary~\ref{cor:DJK} and Proposition~\ref{pro:repradial} are close but incomparable. First,
Proposition~\ref{pro:repradial} applies to systems of equations, whereas Corollary~\ref{cor:DJK} 
applies to radially dependent coefficient. Secondly, the well-posedness assumptions are different. The next results reconciles the two approaches.

 \begin{thm}\label{thm:equiv2} 
 Let $L=-\divv_\bx A \nabla_\bx$ 
 be a real elliptic operator in $\bO^{1+n}$, $n\ge 1$.  
  Assume further that $L$ has coefficients with $\lim_{\tau\to 0}\|\chi_{t<\tau}\E_t\|_{C\cap L_{\infty}}$ sufficiently small. 
Then the following are equivalent.
 \begin{itemize}
  \item[{\rm (i)}] 
  The Dirichlet problem is well-posed in $\mD^o$ for $L$ and $L^*$.
  \item[{\rm (ii)}]  
  The Dirichlet problem is well-posed in $\mY^o$ for $L$ and $L^*$.
\end{itemize}
Moreover, in this case the solutions for $L$ (resp. $L^*$) from a same datum are the same.
\end{thm}

\begin{proof} It suffices to prove the conclusion for $L$ in each case as the assumptions are invariant under taking adjoints.

Assume (i). 
Uniqueness in $\mY^0$ is immediate since the class of $\mD^o$-solutions  \textit{a priori} contains the class $\mY^o$-solutions when $\|\E\|_{C\cap L_{\infty}}<\infty$. Next,  for the existence,  there is  by assumption  a unique $\mD^o$-solution with given boundary datum $\bphi\in L_2(S^n)$.
Since   the Dirichlet problem is well-posed in $\mD^o$ for $L^*$, Corollary~\ref{cor:DJK} shows that this solution is in fact a $\mY^o$-solution. 

Conversely, assume (ii).
By Theorem~\ref{thm:equiv1}, it suffices to show that $(D)_2$ is solvable for $L$.
To this end, it suffices to consider $\bphi\in C^1(S^n)$ and the associated variational solution $u$.
By Proposition~\ref{prop:sioequalvarsol}, 
which applies because of  Theorem~\ref{thm:SAFredholm} ($I-S_{A}$ is invertible on $\mX$ and on $\mY$) and  Proposition~\ref{prop:reg=dir},
 $u$ coincides with the solution in the sense of Definition~\ref{defn:wpbvp}, that is, it is a $\mY^o$-solution. 
Now Theorem~\ref{thm:NTu} provides the non-tangential maximal estimate that shows
that $(D)_2$ is solvable for $L$.
\end{proof}

\begin{rem} In the case of radially independent coefficients (or more generally for continuous, Dini square coefficients)  Proposition~\ref{pro:repradial} (or the remark that follows it)  proves the converse  also for systems.
\end{rem}

We can generalize results of \cite{KP} to non-symmetric perturbations of $r$-independent real symmetric operators. 

\begin{cor} In $\bO^{1+n}$, the Dirichlet problem is well-posed in $\mD^o$  for all real  operators $L$ with coefficients $A$ such that $\lim_{\tau\to 0}\|\chi_{t<\tau}\E_t\|_{C\cap L_{\infty}}$ is small enough and its boundary trace $A_{1}$ real symmetric.
\end{cor}

\begin{proof}  Let $L_{1}$ be the second order operator with $r$-independent coefficients $A_{1}$. By Proposition \ref{prop:hermwp}, we know that that the Dirichlet problem for $L_{1}=L_{1}^*$ is well-posed in $\mY^o$. Thus, by  Theorem \ref{prop:rdeppert}, it is well-posed in $\mY^o$ for $L$ and $L^*$. 
Thus, we conclude with Theorem \ref{thm:equiv2}. 
\end{proof}

We continue with generalizations of results in \cite{FJK}, where well-posedness for Dirichlet was obtained for real symmetric coefficients. Well-posedness for regularity (which we denote here by well-posedness in $\mX^o$) is new. 

\begin{thm}\label{thm:Dinisquarereal} Assume that $A$ are coefficients with $\lim_{\tau\to 0}\|\chi_{t<\tau}\E_t\|_{C\cap L_{\infty}}$ small enough and  boundary trace $A_{1}$ which is real and continuous. Then the Dirichlet problem is well-posed in $\mD^o$ and in $\mY^o$, and the regularity problem in $\mX^o$ is well-posed. In particular, this holds for real continuous coefficients in $\overline{\bO^{1+n}}$ satisfying the Dini square condition $\int_{0}w_{A}^2(t) \frac{dt}{t}<\infty$, where $w_{A}(t)=\sup\{|A(rx)-A(x)|  ;  x\in S^n, 1-r<t\}$.
\end{thm}

\begin{proof} Let $L_{1}$ be the operator with coefficients $A_{1}$. 
Recall that under smallness of $\lim_{\tau\to 0}\|\chi_{t<\tau}\E_t\|_{C\cap L_{\infty}}$, it suffices to prove the result for $L_{1}$ by Proposition \ref{prop:rdeppert}.  Next, by 
   Proposition \ref{prop:reg=dir}, the regularity problem (in $\mX^o$) for $L_{1}$ is well-posed if and only if the Dirichlet problem for $L^*_{1}$ is well-posed in $\mY^o$. On applying Theorem~\ref{thm:equiv2}, it suffices to prove that the Dirichlet problem with coefficients $A_{1}$ is well-posed in $\mD^o$, as the same would then hold for $A_{1}^*$ by symmetry of the assumptions. To do this, we prove that $L_{1}$-harmonic measure satisfies the property (iii) in Theorem \ref{thm:equiv1}. The argument is inspired by the one of \cite{FJK}, p.139-140. 
     
   Assume first we work on some boundary  region of $\bO^{1+n}$. For $r$ small, set  $Q_{r}= \{\rho y\in (0,1)\times S^n ; 1-r<\rho<1, y \in B(x_{0},r)\}$ where $B(x_{0},r)$ is a surface ball of radius $r$, with  real radially independent coefficients $A_{1}$ being the restriction of some matrix defined on $\bO^{1+n}$ that we still denote by $A_{1}$ and which is close in $L_{\infty}$ to the constant matrix $A_{1}(x_{0})$.  
  Let $g$ be a $C^1$ non-negative function supported on the part of the  boundary of $Q_{r/2}$ in $S^n$. Let $v$ be the variational solution to the Dirichlet problem $L_{1}v=0$ in $Q_{r/2}$ and $v=g$ on the  boundary of $Q_{r/2}$ in $S^n$ and $v=0$ on the part of the boundary that is contained in $\bO^{1+n}$. Recall that $v\in W^1_{2}( Q_{r/2}) \cap C(\overline{Q_{r/2}})$. By Theorem~\ref{cor:rindeppert}, because $A_{1}$ is $L_{\infty}$ close to a (constant) matrix for which one knows well-posedness by Proposition~\ref{prop:smoothwp}, one can construct  the unique solution $u$ in $\bO^{1+n}$ to the Dirichlet problem in $\mY^o$ with $u=g$ on $S^n$, that is  $L_{1}u=0$    with $\int_{\bO^{1+n}}|\nabla_{\bx} u|^2(1-|\bx|)\, d\bx \le C\|g\|_{2}^2$. As $g\in C^1(S^n)$, we know on applying Lemma~\ref{prop:sioequalvarsol}  that the solution $u$ is variational, \textit{i.e.} $ u \in W^1_{2}(\bO^{1+n})$. We can apply Stampacchia's minimum principle  to obtain first that $u\ge 0$ in $\bO^{1+n}$, and next the maximum principle in $Q_{r/2}$ to conclude that $v\le u$.    From there, it remains to repeat the argument in  \cite{FJK}, to obtain that $
  v(\rho y) \le C (1-\rho)^{-n/2}\|g\|_{2}$ for all $\rho y \in Q_{r/4}$, which in turn, yields an $L_{2}$ estimate on the Radon-Nikodym derivative of the $L_{1}$-elliptic measure. 
  
  The localisation argument as in \cite{FJK}, and using the continuity of $A_{1}$ to cover a layer of the boundary  with  a finite number of such small $Q_{r/2}$, allows us to conclude. 
 \end{proof}
 
 \begin{cor} With the same assumption as above and $n=1$, then the Neumann problem with coefficients $A$ is well-posed in $\mX^o$.
 \end{cor}
 
 \begin{proof} By the results in Section \ref{sec:disk}, it follows that the Neumann problem for coefficients $A$ is well-posed in $\mX^o$ if and only if the regularity problem for conjugate coefficients $\tilde A$ is well-posed in $\mX^o$.  The latter follows from the previous result since $\tilde A$ satisfies the same assumption as $A$. 
 \end{proof}

\begin{rem}As in \cite{FJK}, the Dini square condition in the normal direction can be replaced by a Dini square condition in a $C^1$ transverse direction to the sphere. It suffices to perform locally  changes of variables that transform the transverse direction to normal ones. 
\end{rem}

\bibliographystyle{acm}

\end{document}